\documentclass[10pt]{amsart}
\usepackage{amsfonts,amsthm,amssymb,euscript}
\usepackage[pdftex]{graphicx}
\usepackage{graphics}
\usepackage[matrix,arrow]{xy}
\usepackage[active]{srcltx}

\usepackage[colorlinks=true]{hyperref}

\newcommand{\C}{\mathbb{C}}
\newcommand{\CP}{\mathbb{C}\text{P}}
\newcommand{\R}{\mathbb{R}}
\newcommand{\Z}{\mathbb{Z}}
\newcommand{\Q}{\mathbb{Q}}

\newcommand{\M}{\mathcal{M}}
\newcommand{\J}{\mathcal{J}}

\newcommand{\F}{\mathcal{F}}

\newcommand{\U}{\mathcal U}

\renewcommand{\sl}{{\rm sl}}
\newcommand{\link}{{\rm link}}
\newcommand{\wind}{{\rm wind}}
\renewcommand{\P}{\mathcal{P}}
\newcommand{\util}{\widetilde u}
\newcommand{\vtil}{\widetilde v}
\newcommand{\jtil}{\widetilde J}
\newcommand{\jhat}{\widehat J}
\newcommand{\est}{e^{2\pi(s+it)}}
\newcommand{\maslov}{{\rm Maslov}}
\newcommand{\Img}{\text{image}}

\theoremstyle{plain}
\newtheorem{theorem}{Theorem}[section]
\newtheorem{proposition}[theorem]{Proposition}
\newtheorem{lemma}[theorem]{Lemma}
\newtheorem{corollary}[theorem]{Corollary}

\theoremstyle{definition}
\newtheorem{definition}[theorem]{Definition}

\theoremstyle{remark}
\newtheorem{remark}[theorem]{Remark}

\title[A Poincar\'e-Birkhoff Theorem for tight Reeb flows on $S^3$]{A Poincar\'e-Birkhoff Theorem for \\ tight Reeb flows on $S^3$}
\author{Umberto Hryniewicz}
\address[Umberto Hryniewicz]{Universidade Federal do Rio de Janeiro -- Departamento de Matem\'atica Aplicada, Av.\ Athos da Silveira Ramos 149, Rio de Janeiro RJ, Brazil 21941-909}
\email[Umberto Hryniewicz]{\texttt{umberto@labma.ufrj.br}}
\author{Al Momin}
\address[Al Momin]{Evermight -- http://evermight.com/}
\email[Al Momin]{\texttt{amomin@evermight.com}}
\author{Pedro A. S. Salom\~ao}
\address[Pedro A. S. Salom\~ao]{Universidade de S\~ao Paulo,  Instituto de Matem\'atica e Estat\'istica -- Departamento de Matem\'atica, Rua do Mat\~ao, 1010 - Cidade Universit\'aria - S\~ao Paulo SP, Brazil 05508-090}
\email[Pedro A. S. Salom\~ao]{psalomao@ime.usp.br}

\begin{document}

\begin{abstract}
We consider Reeb flows on the tight $3$-sphere admitting a pair of closed orbits forming a Hopf link. If the rotation numbers associated to the transverse linearized dynamics at these orbits fail to satisfy a certain resonance condition then there exist infinitely many periodic trajectories distinguished by their linking numbers with the components of the link. This result admits a natural comparison to the Poincar\'e-Birkhoff theorem on area-preserving annulus homeomorphisms. An analogous theorem holds on $SO(3)$ and applies to geodesic flows of Finsler metrics on $S^2$.
\end{abstract}

\maketitle

\section{Introduction}

Since the work of Poincar\'e and Birkhoff the notion of global surface of section has been used as an effective tool in finding periodic motions of Hamiltonian systems with two degrees of freedom, see~\cite{Po1,Po_book,Bi4,Bi3}. A global section of annulus-type on an energy level implies the existence of many closed orbits by the celebrated Poincar\'e-Birkhoff Theorem~\cite{Po2,Bi1,Bi2} if the associated return map satisfies a twist-condition.

Our goal is to describe a non-resonance condition for Reeb flows on the tight 3-sphere which implies the existence of infinitely many closed orbits, and generalizes the twist-condition on the Poincar\'e-Birkhoff Theorem to cases where a global surface of section might not be available. We assume instead that there is a pair of periodic orbits forming a Hopf link. The infinitesimal flow about the two components defines rotation numbers and, as we shall see, if these numbers do not satisfy a precise resonance condition then infinitely many closed orbits exist and are distinguished by their homotopy classes in the complement of the Hopf link. This lack of resonance can be seen as a twist-condition: one finds a non-empty open twist interval such that there is a closed orbit associated to every rational point in its interior.

In the presence of a disk-like global surface of section for the flow, an orbit corresponding to a fixed point of the return map and the boundary of the global section constitute a Hopf link. According to a remarkable result by Hofer, Wysocki and Zehnder~\cite{convex} this is the case for Reeb flows given by dynamically convex contact forms on the $3$-sphere. The return map restricted to the open annulus obtained by removing a fixed point is well-defined. In this case, the lack of resonance mentioned above is a twist-condition, and our result can be reduced to the Poincar\'e-Birkhoff Theorem, or rather to a generalization due to Franks~\cite{MR951509}. We will explain this analogy more thoroughly in Section~\ref{sec-intro_interpretation}.

There are examples of Hopf links and Reeb flows as above where both components of the link do not bound a disk-like global section. In this case a two-dimensional reduction is not available. To circumvent this difficulty, we use a different approach in place of the theory of global surfaces of section which is of a variational nature. The idea is to consider the homology of the abstract Conley index of a sufficiently large isolating block for the gradient flow of the action functional, as Angenent did in~\cite{angenent} for the energy. The analysis of~\cite{angenent} shows that properties of the curve-shortening flow are sufficient in order to define a Conley index associated to a so-called flat knot, which in special cases can be used to deduce existence results for closed geodesics on the 2-sphere. We shall consider instead cylindrical contact homology on the complement of the Hopf link, which is defined using the machinery of punctured pseudo-holomorphic curves in symplectizations as introduced by Hofer~\cite{93}. In this sense, the results are analogous to those of~\cite{angenent}, but in the more general setting of Reeb flows on the tight 3-sphere. We explain this analogy more thoroughly in Section~\ref{sec-intro_SO3}.

\subsection{Statement of main result}\label{section_main_result_major}

Recall that a 1-form $\lambda$ on a 3-manifold $V$ is a contact form if $\lambda \wedge d\lambda$ never vanishes. The 2-plane field
\begin{equation}\label{contact_structure}
  \xi = \ker \lambda
\end{equation}
is a co-oriented contact structure, and the associated Reeb vector field $X_\lambda$ is uniquely determined by
\begin{equation}\label{}
  \begin{array}{cc}
    i_{X_\lambda}\lambda = 1, & i_{X_\lambda}d\lambda = 0.
  \end{array}
\end{equation}
The contact structure $\xi$ is said to be tight if there are no overtwisted disks, that is, there does not exist an embedded disk $D\subset V$ such that $T\partial D \subset \xi$ and $T_pD \neq \xi_p, \ \forall p\in\partial D$. In this case we call $\lambda$ tight.

By a closed Reeb orbit we mean an equivalence class of pairs $P = (x,T)$ such that $T>0$ and $x$ is a $T$-periodic trajectory of $X_\lambda$, where pairs with the same geometric image and period are identified. The set of equivalence classes is denoted by $\P(\lambda)$. $P=(x,T)$ is called prime, or simply covered, if $T$ is the minimal positive period of $x$. Throughout a knot $L \subset V$ tangent to $\R X_\lambda$ is identified with the prime closed Reeb orbit it determines, in particular, $L$ inherits an orientation.

We are concerned with the study of the global dynamical behavior of Reeb flows associated to tight contact forms on $$ S^3 = \{(x_0,y_0,x_1,y_1)\in\R^4 \mid x_0^2+y_0^2+x_1^2+y_1^2=1\} $$ where $(x_0,y_0,x_1,y_1)$ are coordinates in $\R^4$. For instance consider the 1-form
\begin{equation}\label{std_Liouville_form}
  \lambda_0 = \frac{1}{2} (x_0dy_0 - y_0dx_0 + x_1dy_1 - y_1dx_1).
\end{equation}
It restricts to a tight contact form on $S^3$ inducing the so-called standard contact structure
\begin{equation}\label{std_contact_str}
  \xi_0 = \ker \lambda_0|_{S^3}.
\end{equation}
In dimension $3$ a contact structure $\xi$ induces an orientation of the underlying manifold $M$ in the following manner. If $p\in M$ then choose a contact form $\alpha$ defined near $p$ satisfying $\xi=\ker\alpha$. The $3$-form $\alpha\wedge d\alpha$ is nowhere vanishing on its domain of definition, and defines an orientation of $T_pM$ by declaring that a basis $\{v_1,v_2,v_3\} \subset T_pM$ is positive if, and only if, $\alpha\wedge d\alpha(v_1,v_2,v_3)>0$. This orientation of $T_pM$ is independent of the choice of $\alpha$, and we get a global orientation letting $p$ vary over $M$. If $M$ is already oriented then one calls $\xi$ positive if it induces the given orientation. Let us orient $S^3$ as the boundary of the unit ball in $\R^4$, which is oriented by $d\lambda_0\wedge d\lambda_0$. By a theorem due to Eliashberg~\cite{eli}, for every tight contact form $\lambda$ on $S^3$ defining a positive contact structure, there exists a diffeomorphism $\Phi:S^3\to S^3$ such that $\Phi^*\lambda=f\lambda_0$, for some smooth $f:S^3 \to (0,+\infty)$.

We use the term Hopf link to refer to a transverse link on $(S^3,\xi_0)$ which is transversally isotopic to $K_0 = L_0 \cup L_1$ where
\begin{equation}\label{std_hopf_link}
  L_i = \{ (x_0,y_0,x_1,y_1) \in S^3 \mid x_i=y_i=0 \}, \ \ i=0,1.
\end{equation}

\begin{remark}\label{std_representation_hopf_link}
Consider the set
\begin{equation}\label{adapted_link}
  \F = \{ f \in C^\infty(S^3,(0,+\infty)) \mid i_vdf=0 \ \forall v\in \xi_0|_{K_0} \}.
\end{equation}
The set $\F$ consists precisely of the functions $f:S^3\to(0,+\infty)$ such that the Reeb vector field of $f\lambda_0$ is tangent to $K_0$. Moreover, for every defining contact form $\lambda$ on $(S^3,\xi_0)$ admitting a pair of prime closed Reeb orbits that are components of a Hopf link, there exists some diffeomorphism $\Phi$ of $S^3$ such that $\Phi^*\lambda = f\lambda_0$, for some $f\in \F$, and $\Phi$ maps $K_0$ onto the Hopf link. To see this, first note that any such contact form is written as $\lambda=h\lambda_0$, for some $h:S^3\to \R\setminus\{0\}$ smooth.  Consider a transverse isotopy $g_t:K_0\to (S^3,\xi_0)$, $t\in[0,1]$, such that $g_0$ is the inclusion map $K_0\hookrightarrow S^3$ and $g_1(K_0)$ is a pair of prime closed Reeb orbits of $h\lambda_0$. By Theorem~2.6.12 from~\cite{geiges}, there exists a contact isotopy $\{\varphi_t\}_{t\in[0,1]}$ of $(S^3,\xi_0)$ such that $\varphi_0=id$ and $\varphi_t|_{K_0}\equiv g_t, \ \forall t$. Then $\varphi_1^*(h\lambda_0) = k\lambda_0$ for some $k:S^3\to \R\setminus\{0\}$ smooth. If $k$ is positive we take $f=k$ and $\Phi = \varphi_1$. If $k$ is negative we consider the diffeomorphism $T(x_0,y_0,x_1,y_1)=(x_0,-y_0,x_1,-y_1)$, which satisfies $T^*\lambda_0=-\lambda_0$, so we can take $\Phi = \varphi_1\circ T$ and $f=-k\circ T$. In both cases we must have $f\in\F$ since the Reeb vector field of $f\lambda_0$ is tangent to $K_0$.
\end{remark}

We define the transverse rotation number $\rho(P)$ of a closed Reeb orbit $P$ by looking at the rate at which the transverse linearized flow rotates around $P$, measured with respect to coordinates on the contact structure induced by a global positive frame. This is well-defined as a real number and equals half the mean Conley-Zehnder index. For a more detailed discussion see Section~\ref{cz_rotation_orbits} below.

Finally, we introduce some notation in order to simplify our statements. Given two pairs of real numbers $(s_0,t_0), (s_1,t_1)$ in the set $\{ (s,t) \mid s > 0 \mbox{ or } t > 0 \}$ we write $(s_0,t_0) < (s_1,t_1)$ if, viewed as vectors in $\R^2$, the argument of $(s_1,t_1)$ is greater than that of $(s_0,t_0)$ when measured counter-clockwise by cutting along the negative horizontal axis. A pair of integers $(p,q)$ will be called relatively prime if there is no integer $k > 1$ such that $(p/k,q/k)\in \mathbb{Z} \times \mathbb{Z}$. Our first main result reads as follows.

\begin{theorem}\label{thm-1}
Let $\lambda = f\lambda_0$, $f>0$, be a tight contact form on the $3$-sphere admitting prime closed Reeb orbits $L_0,L_1$ which are the components of a Hopf link. Define real numbers $\theta_0,\theta_1$ by
\begin{equation}
  \theta_i = \rho(L_i)-1, \ \ \text{ for } i=0,1,
\end{equation}
where $\rho$ is the transverse rotation number, and suppose that $(p,q)$ is a relatively prime pair of integers satisfying
\begin{equation}\label{non_resonance}
  \begin{array}{ccc}
    (\theta_0,1) < (p,q) < (1,\theta_1) & \text{or} & (1,\theta_1) < (p,q) < (\theta_0,1).
  \end{array}
\end{equation}
Then there exists a prime closed Reeb orbit $P \subset S^3\setminus (L_0 \cup L_1)$ such that $\link(P,L_0) = p$ and $\link(P,L_1) = q$.
\end{theorem}

In the above statement $P$, $L_0$ and $L_1$ are oriented by the Reeb vector field, $S^3$ is oriented by the contact structure $\xi_0$ as explained before, and the integers $\link(P,L_0)$ and $\link(P,L_1)$ are defined using these choices, see Figure~1 for an example with $p=7$ and $q=1$.

A weaker version of Theorem~\ref{thm-1} is found in~\cite{momin} under the restrictive assumption that the components of the Hopf link are irrationally elliptic Reeb orbits.

\subsection{Interpretation in terms of the Poincar\'e-Birkhoff Theorem}\label{sec-intro_interpretation}

In 1885 Poincar\'e \cite{Po1} introduced the rotation number \begin{equation}\label{eq.1}\rho(f)= \lim_{n\to \infty} \frac{F^n(x)}{n} \mod \mathbb{Z}\end{equation} of an orientation preserving circle homeomorphism $f:S^1 \to S^1$, $S^1 \equiv \mathbb{R} / \mathbb{Z},$ where $F:\mathbb{R}\to \mathbb{R}$ is one of its lifts. Notice that the limit in \eqref{eq.1} exists and does not depend on $x\in \mathbb{R}$ or on the lift $F$. He observed its intimate connection to the existence of periodic orbits.

\begin{figure}
\begin{center}
\includegraphics[width=100mm]{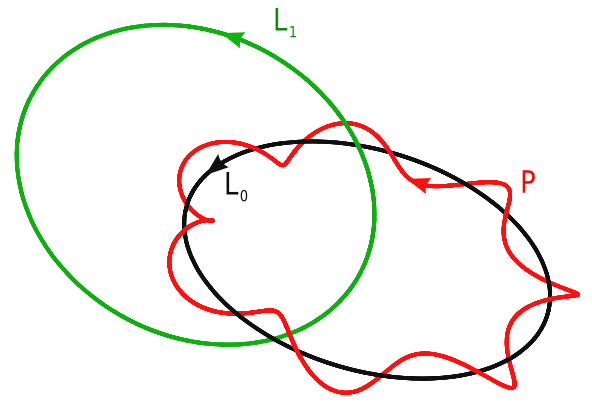}
\end{center}
\caption{A Hopf link $K_0 = L_0 \cup L_1$ and a closed Reeb orbit $P$ satisfying $\link(P,L_0)=7$, $\link(P,L_1)=1$.}
\end{figure}

\begin{theorem}[Poincar\'e]\label{teo.Poincare}
$f$ admits a periodic orbit if, and only if, $\rho(f)=p/q \in \mathbb{Q}/\mathbb{Z}$.
\end{theorem}

If one considers an area preserving annulus homeomorphism $$f: S^1 \times [0,1] \to S^1 \times [0,1],$$ isotopic to the identity map, much can be said about the existence of periodic orbits when $f$ satisfies a twist hypothesis. To be more precise, let us first recall the widely known Poincar\'e-Birkhoff Theorem in its original form. Let $$ F: \mathbb{R} \times [0,1] \to \mathbb{R} \times [0,1] $$ be a lift of $f$ with respect to the covering map $\pi:\mathbb{R} \times [0,1] \to S^1 \times [0,1]$ and denote by $I\subset\R$ the open (possibly empty) interval bounded by the points
\begin{equation*}
\lim_{n\to \infty} \frac{p_1 \circ F^n(x,0)}{n} \ \ \mbox{ and } \ \ \lim_{n\to \infty} \frac{p_1 \circ F^n(x,1)}{n}.
\end{equation*}
Here $p_1:\mathbb{R} \times [0,1] \to \mathbb{R}$ is the projection onto the first factor.

\begin{theorem}[Poincar\'e-Birkhoff, see \cite{Bi1,Bi2,Po2}]\label{theoPB}\label{teo.PoincareBirkhoff}
If $I \cap \mathbb{Z} \neq \emptyset$ then $f$ has at least $2$ fixed points.
\end{theorem}

A proof of a version of this theorem in the smooth category using pseudo-holomorphic curves can be found in~\cite{BH}.

A map $f$ on $S^1 \times [0,1]$ satisfying $I \neq \emptyset$ for some lift is said to satisfy a twist condition. Considering the iterates of $f$ one can find infinitely many periodic orbits under this twist condition. This argument can be found in \cite{neumann} where the following theorem is proved.

\begin{theorem}[Neumann~\cite{neumann}]\label{theoNeu} For any $q\in \mathbb{N}=\{1,2,\dots\}$, the number of periodic orbits of prime period $q$ is at least equal to  $$2\#\{p\in \mathbb{Z}: p/q\in I \text{ and } \gcd(p,q)=1\}.$$ \end{theorem}

J. Franks generalized Theorem \ref{theoNeu}, providing the existence of periodic orbits under a much weaker twist condition, even when $f$ is not defined on the boundary.

\begin{theorem}[J. Franks, see~\cite{MR951509,Fr1,Fr}]\label{teo.Franks}
If there exist $z_1,z_2\in \mathbb{R} \times [0,1]$ such that
\begin{equation}\label{eq.3}
  \lim_{n\to \infty} \frac{p_1 \circ F^n (z_1)}{n} \leq \frac{p}{q} \leq \lim_{n \to \infty} \frac{p_1 \circ F^n (z_2)}{n},
\end{equation}
then $f$ has a periodic point $z$ with period $q$ and $$ \lim_{n\to \infty} \frac{p_1 \circ F^n(z_0)}{n} = \frac{p}{q}, $$ for any $z_0$ satisfying $z_0 \in \pi^{-1}(z)$.
\end{theorem}

Both limits in~\eqref{eq.3} are assumed to exist. Let us refer to the periodic orbits obtained in Theorem~\ref{teo.Franks} as the $p/q$-orbits. In~\cite{Fr} the reader also finds a version of the above statement on the open annulus.

Theorem~\ref{thm-1} can be reduced to Theorem~\ref{teo.Franks} in the case one of the components of the Hopf link bounds a disk-like global surface of section. We very briefly sketch this argument and do not give full details since the more general Theorem~\ref{thm-1} does not require this surface of section at all.

\begin{definition}\label{def-global_surface}
Let $\lambda$ be a tight contact form on $S^3$ and denote by $X_{\lambda}$ its Reeb vector field. We say that an embedded disk $\Sigma\subset S^3$ is a disk-like global surface of section for the Reeb flow if $\partial \Sigma = P$ is a closed orbit, $X_{\lambda}$ is transverse to $\mathring \Sigma$ and all orbits in $S^3 \setminus P$ intersect $\mathring \Sigma$ infinitely often, both forward and backward in time.
\end{definition}

Let $L_0 \cup L_1$ be a Hopf link formed by closed Reeb orbits and assume that $L_1$ bounds a disk-like global surface of section for the Reeb flow of $\lambda=f\lambda_0$. Define $\theta_0,\theta_1$ as in Theorem~\ref{thm-1}. Assuming, for simplicity, that $\lambda$ is non-degenerate then results from~\cite{HLS,HS} tell us that there is an open book decomposition of $S^3$ with binding $L_1$ and disk-like pages which are global surfaces of section. See also~\cite{Hryn,hry} for the dynamically convex case. In particular, there is a diffeomorphism $S^3\setminus L_1 \simeq \R/\Z \times B$ where $B\subset\C$ is the open unit ball, such that $L_0 \simeq \R/\Z \times \{0\}$ and if we denote by $\vartheta$ the $\R/\Z$-coordinate then the Reeb flow satisfies $d\vartheta(X_\lambda)>0$. Moreover, the Conley-Zehnder index of $L_1$ is at least $3$, which implies $\theta_1>0$. We assume, in addition, that our coordinates are such that $\R/\Z\times [0,1)$ is contained on an embedded disk spanning $L_0$. The first return map $g$ to the page $0\times B$ has $0$ as a fixed point and, introducing suitable polar coordinates $B\setminus 0 \simeq \R/\Z\times(0,1)$, we get an area-preserving diffeomorphism of $\R/\Z\times (0,1)$ still denoted by $g$. The open book decomposition also induces an isotopy from the identity to $g$. Lifting the identity on $\R/\Z\times(0,1)$ to the identity on $\R\times(0,1)$, this isotopy distinguishes a particular lift $\widetilde g$ of $g$ to $\R\times(0,1)$. The map $\widetilde g$ can now be continuously extended to $\R\times[0,1)$ by using the transverse linearized flow at $L_0$. Under additional assumptions we can also extend $\widetilde g$ continuously to $\R/\Z\times[0,1]$. Since $L_0$ has self-linking number $-1$, the canonical basis of $0\times\C$ induces, via the identification $S^3\setminus L_1 \simeq \R/\Z\times B$, a homotopy class of symplectic frames of $\ker\lambda|_{L_0}$ with respect to which the transverse linearized Reeb flow has rotation number equal to $\theta_0$. Hence the rotation number of $\widetilde g|_{\R\times0}$ is $\theta_0$. By a similar reasoning, the rotation number of $\widetilde g|_{\R\times1}$ is $1/\theta_1$. One obtains $(p,q)$-orbits as asserted in Theorem~\ref{thm-1} from Franks' $p/q$-orbits in Theorem~\ref{teo.Franks} taking $z_1 \in \R\times0$ and $z_2\in \R\times 1$ under the assumption that $\theta_0 < p/q < 1/\theta_1$ or $1/\theta_1 < p/q < \theta_0$.

What is unsatisfactory about this argument is that one may construct examples of Reeb flows and Hopf links $L_0 \cup L_1$ as above satisfying the hypotheses of Theorem~\ref{thm-1} but neither $L_0$ nor $L_1$ bound a global disk-like surface of section. Such an example is provided in Section~\ref{example-2} below when choosing $\theta_0, \theta_1$ to be both negative numbers: in this case, the discussion thereafter shows that there are periodic orbits $P_i$ having linking number $0$ with $L_i$, for $i = 0,1$, which clearly conflicts with the assumption of a global surface of section. For this reason, we approach the problem with a different set of tools. The argument we pursue has instead the spirit of an argument of Angenent~\cite{angenent}, which we will recall shortly.

\subsection{The unit tangent bundle of $S^2$} \label{sec-intro_SO3}

Poincar\'e observed the importance of studying area-preserving annulus homeomorphisms by finding annulus-type global sections for the restricted $3$-body problem. In his book~\cite{Bi3}, G. D. Birkohff proved that the geodesic flow of a Riemannian metric $g$ on $S^2$ with positive curvature also admits annulus-type global sections. In fact, one can always find a simple closed geodesic $\gamma : \R/T\Z \to S^2$, with minimal period $T$ and parametrized by arc-length. Its image separates $S^2$ in two closed disks $C_1$ and $C_2$. For each $x\in \text{image}(\gamma) = \partial C_1 = \partial C_2$, let $n(x)\in M$ be the normal vector to $\partial C_1$ pointing outside $C_1$, where $M = \{ (x,v) \in TS^2 \mid g(v,v) = 1 \} \simeq SO(3)$ is the unit tangent bundle, and let
\[
  \Sigma=\{(x,v)\in M: x\in \text{image}(\gamma) \mbox { and } g(v,n(x)) \geq 0\}.
 \]
Denote by $\gamma_r$ the reverse orbit $\gamma_r(t)=\gamma(-t)$ of $\gamma$. Then $\gamma$, $\gamma_r$ admit natural lifts $\dot\gamma$, $\dot\gamma_r$ to $M$ and $\Sigma$ is an annulus-type global surface of section for the geodesic flow with boundary $\partial \Sigma = \text{image}(\dot\gamma) \cup \text{image}(\dot\gamma_r)$. The first return Poincar\'e map to $\mathring \Sigma$ can be extended to the boundary $\partial \Sigma$ using the second conjugate point, and this induces an area preserving annulus homeomorphism $f:S^1 \times [0,1] \to S^1 \times [0,1]$ isotopic to the identity. By Theorem~\ref{teo.Franks}, $f$ admits all the $p/q$-orbits as long as the twist condition \eqref{eq.3} is satisfied for a lift $F$ of $f$.

It is well-known that one might not expect the existence of these types of $(p,q)$-orbits for $C^1$ volume preserving flows on a $3$-manifold. In fact, inserting a plug of Kuperberg-Schweitzer-Wilson type, see~\cite{Ku,Sch,Wi}, one can destroy them without creating new ones. However, as an example, such orbits still exist for geodesic flows on $S^2$, even when an annulus-type global section does not exist. To be more precise, we recall Angenent's result~\cite{angenent} on curve shortening flows applied to the existence of $(p,q)$-satellites of a simple closed geodesic $\gamma$. A Jacobi field over $\gamma$ is characterized by a solution $y:\mathbb{R} \to \mathbb{R}$ of
\begin{equation}\label{jacobi_field}
  y''(t) = -K(\gamma(t))y(t),
\end{equation}
where $K$ is the Gaussian curvature of $(S^2,g)$. For a non-trivial solution $y$, we can write $y'(t)+iy(t)= r(t)e^{i\theta(t)},\ t \in \R$, for $r$ and $\theta$ smooth with non-vanishing $r$. The inverse rotation number of $\gamma$, denoted by $\rho(\gamma)$, is defined by
\begin{equation}\label{rhogeod}
\rho(\gamma) = T \lim_{t \to \infty} \frac{\theta(t)}{2\pi t},
\end{equation}
where $T$ is the minimal period of $\gamma$. The inverse rotation number coincides with the transverse rotation number explained before and we may use both terminologies in the context of geodesic flows.

Let $p$ and $q\not=0$ be relatively prime integers and $n(t)$ be a continuous normal unit vector to a simple curve $\gamma:\R/\Z \to S^2$. A $(p,q)$-satellite of $\gamma$ is any smooth immersion $\R/\Z \to S^2$ equivalent to
\[
  \alpha_{\varepsilon}:\R/\Z \to S^2 \ \ \ \alpha_{\varepsilon}(t) = \exp_{\gamma(qt)} \left( \varepsilon \sin(2 \pi p t) n(qt) \right),
\]
where $\varepsilon>0$ is small and $\exp$ is any exponential map. By equivalent immersed curves we mean curves which are homotopic to each other on $S^2$ through immersed curves, but tangencies with $\gamma$ and self-tangencies are not allowed in the homotopy. The resulting equivalence classes are called flat-knot types relative to $\gamma$.

\begin{theorem}[Angenent~\cite{angenent}]\label{theo_ang}
Let $g$ be a smooth Riemannian metric on $S^2$, and $\gamma$ be a closed prime geodesic which is a simple curve. If the rational number $p/q \in (\rho(\gamma),1)\cup (1,\rho(\gamma))$ is written in lowest terms, then $g$ admits a closed geodesic $\gamma_{p,q}$ which is a $(p,q)$-satellite of $\gamma$. The geodesic $\gamma_{p,q}$ intersects $\gamma$ at exactly $2p$ points and self-intersects at $p(q-1)$ points.
\end{theorem}

One remarkable aspect of Angenent's proof is that it does not use any surface of section: the geometric arguments available in the presence of a global surface of section are replaced with the analysis of the curve-shortening flow which allows for the definition of an isolating block in the sense of Conley theory~\cite{conley}. Theorem~\ref{theo_ang} is obtained by showing that a certain isolated invariant set has a non-trivial index.

According to Angenent the results from~\cite{angenent} were inspired by a question asked by Hofer in Oberwolfach 1993. Hofer asked if it was possible to apply the Floer homology construction to curve shortening, and which results could be obtained in this way. Here we apply Floer theoretic methods to generalize Theorem~\ref{theo_ang} to broader classes of Hamiltonian systems.

Let $g_0$ be the Euclidean metric on $\R^3$ restricted to $S^2 = \{ x\in \R^3 \mid g_0(x,x) = 1 \}$. In Section~\ref{sec-SO3}, we prove a version of Theorem~\ref{thm-1} on the unit sphere bundle $T^1S^2$ associated to $g_0$. Let $\lambda= f \bar \lambda_0$, $f>0$, be a contact form inducing the standard tight contact structure $\bar \xi_0 := \ker \bar \lambda_0$ on $T^1S^2$, where $\bar \lambda_0|_v \cdot w = g_0(v,d\Pi \cdot w)$ and $\Pi : T^1S^2 \to S^2$ is the bundle projection. Recall that there exists a natural double covering map $D:S^3 \to T^1S^2$ satisfying $D^* \bar \lambda_0=  4\lambda_0|_{S^3}$ and which sends the Hopf link $L_0 \cup L_1$ to the pair of closed curves $l_0 := D(L_0)$ and $l_1:=D(L_1)$, both transverse to $\bar \xi_0$. We call the link $l:=l_0 \cup l_1$ a Hopf link in $T^1S^2$, as well as any link which is transversally isotopic to it. The Hopf link $l$ is said to be in normal position. According to Theorem~2.6.12 from~\cite{geiges}, any Hopf link can be brought to normal position by an ambient contact isotopy. The homotopy class $[\gamma]\in \pi_1(T^1S^2 \setminus l,{\rm pt})$ of a closed curve $\gamma \subset T^1S^2 \setminus l$ is determined by two half-integers $$ \begin{array}{cccc} \wind_0(\gamma) \in \Z/2, & \wind_1(\gamma) \in \Z/2 & \text{satisfying} & \wind_0(\gamma)+\wind_1(\gamma) \in \Z. \end{array} $$ They are defined as follows: any lift of $\gamma$ to $S^3\setminus (L_0 \cup L_1)$ has well-defined arguments $\phi_0,\phi_1$ of the complex components $x_0+iy_0$ and $x_1+iy_1$, and $\wind_i(\gamma)$ is defined as the variation of a continuous lift of $\phi_i$ to $\R$ divided by $2\pi$, $i=0,1$. See Section~\ref{sec-SO3} for a more detailed discussion.

\begin{theorem}\label{teot}
Let $\lambda = f \bar \lambda_0$ be a contact form on $T^1S^2$ admitting prime closed Reeb orbits $l_i$, $i=1,2$, which are the components of a Hopf link $l$, assumed to be in normal position without loss of generality. Let $\eta_0$ and $\eta_1$ be the real numbers defined by
\begin{equation}\label{thetaeq}
\eta_i = 2\rho(l_i) - 1,\ i=0,1,
\end{equation}
where $\rho(l_i)$ are the transverse rotation numbers of $l_i$. Let $(p,q)\in \mathbb{Z}\times \mathbb{Z}$ be a relatively prime pair of integers. Assume that
\begin{equation}\label{non-resonance-eta-i}
\begin{array}{ccc} (1,\eta_1)<(p,q)<(\eta_0,1) & \mbox{ or } & (\eta_0,1)<(p,q)<(1,\eta_1). \end{array}
\end{equation}
Then one of the following holds.
\begin{itemize}
\item[(i)] If $p+q$ is even, then $\lambda$ admits a prime closed Reeb orbit $\gamma_{p,q} \subset T^1S^2 \setminus l$, non-contractible in $T^1S^2$, satisfying
\begin{equation}\label{windeven}
  \begin{array}{cc} \wind_0(\gamma_{p,q}) = p/2, & \wind_1(\gamma_{p,q}) = q/2. \end{array}
\end{equation}
\item[(ii)] If $p+q$ is odd, then $\lambda$ admits a prime closed Reeb orbit $\gamma_{p,q} \subset T^1S^2 \setminus l$, contractible in $T^1S^2$, satisfying
\begin{equation}\label{windodd}
  \begin{array}{cc} \wind_0(\gamma_{p,q}) = p, & \wind_1(\gamma_{p,q}) = q. \end{array}
\end{equation}
\end{itemize}
\end{theorem}

Theorem~\ref{teot} implies that if the resonance condition $\eta_0 = 1/\eta_1>0$ is not satisfied, then we obtain infinitely many $(p,q)$-orbits characterized by their homotopy classes in $T^1S^2 \setminus l$. This includes non-contractible orbits in $T^1S^2$.

Now we briefly discuss some applications of Theorem~\ref{teot} which, in particular, generalize Angenent's Theorem~\ref{theo_ang} to geodesic flows of Finsler metrics on the $2$-sphere.

Let $F:TS^2 \to \R$ be a Finsler metric with the associated unit tangent bundle $F^{-1}(1)$, and let $\mathcal L_F: T^*S^2 \setminus 0 \to TS^2\setminus 0$ be the associated Legendre transformation. This induces a cometric $F^* = F \circ \mathcal L_F$ on $T^*S^2$. Analogously we have $F_0 = \sqrt{g_0(\cdot,\cdot)}$, $\mathcal L_{F_0}$ and $F_0^*$ for the Euclidean metric. On $T^*S^2$ we have the tautological 1-form $\lambda_{\rm taut}$. The 1-form $\bar\lambda_F = (\mathcal L_F^{-1})^*\lambda_{\rm taut}$ is a contact form on $F^{-1}(1)$ inducing the contact structure $\bar\xi_F = \ker\bar\lambda_F$, and its Reeb flow coincides with the geodesic flow of $F$. Clearly $\bar \lambda_0 = (\mathcal L_{F_0}^{-1})^*\lambda_{\rm taut}$. Consider the map $\Psi : (F_0^*)^{-1}(1) \to (F^*)^{-1}(1)$, $p \mapsto p/F^*(p)$. Then
\begin{equation}\label{diffeo_G}
  G = \mathcal L_F \circ \Psi \circ \mathcal L_{F_0}^{-1} : (T^1S^2,\bar\xi_0) \to (F^{-1}(1),\bar\xi_F)
\end{equation}
defines a co-orientation preserving contactomorphism, that is, $G^*\bar\lambda_F = f\bar\lambda_0$ for some positive function $f$. A geodesic $\gamma$ of $F$ with unit speed admits a lift
\begin{equation}\label{lift_gamma_bar}
    \bar\gamma := G^{-1}(\dot\gamma)
\end{equation}
under the projection $\Pi$, which is a trajectory of the Reeb flow of $f\bar\lambda_0$. We call $\gamma$ contractible when $\bar\gamma$ is contractible in $T^1S^2$, or equivalently when $\dot\gamma$ is contractible in $F^{-1}(1)$.

\begin{corollary}~\label{cor-geod2}
Let $F$ be a Finsler metric on $S^2$, and $\gamma_0, \gamma_1$ be two closed geodesics that lift to a Hopf link $l = l_0 \cup l_1 \subset T^1 S^2$, that is, $l_0 = \bar\gamma_0$ and $l_1 = \bar \gamma_1$. Without loss of generality we assume $l$ is in normal position. Consider their inverse rotation numbers $\rho(l_i)$, $i=0,1$, and let $$ \eta_i=2\rho(l_i)-1, \ \ i=0,1. $$  If $(p,q)$ is a relatively prime pair of integers satisfying
\[
(\eta_0,1) < (p,q) < (1,\eta_1) \mbox{ or } (1,\eta_1) < (p,q) < (\eta_0,1)
\]
then we have one of the following cases.
\begin{enumerate}
 \item If $p+q$ is even, then $F$ admits a non-contractible prime closed geodesic $\gamma_{p,q}$ whose lift $\bar \gamma_{p,q}$ lies in $T^1S^2\setminus l$ and satisfies
\begin{equation*}
  \begin{array}{cc} \wind_0(\bar \gamma_{p,q}) = p/2, & \wind_1(\bar \gamma_{p,q}) = q/2. \end{array}
\end{equation*}
 \item If $p+q$ is odd, then $F$ admits a contractible prime closed geodesic $\gamma_{p,q}$ whose lift $\bar \gamma_{p,q}$ lies in $T^1S^2\setminus l$ and satisfies
\begin{equation*}
  \begin{array}{cc} \wind_0(\bar \gamma_{p,q}) = p, & \wind_1(\bar \gamma_{p,q}) = q. \end{array}
\end{equation*}
\end{enumerate}
\end{corollary}

Jacobi fields~\eqref{jacobi_field} are now defined using flag curvatures $K = K(T_\gamma S^2,\dot\gamma)$. To give a concrete example, Corollary~\ref{cor-geod2} can be applied to a pair of simple closed geodesics which intersect each other at exactly two points in $S^2$. Corollary~\ref{cor-geod2} also applies to any Finsler metric admitting an embedded circle $C \subset S^2$ which is a geodesic when suitably parametrized in both directions. In fact, $C$ and its reversed $C_r$ lift to components of a Hopf link which can be transversally isotoped to normal position. Note that the rotation numbers $\eta_0$, $\eta_1$ may not be related in this case, so that the ``twist interval'' may be empty.  This is the case in the examples of Katok~\cite{Katok}.

We specialize the discussion even further now, to make the comparison with Theorem \ref{theo_ang} clearer.  We shall say that a simple closed geodesic $\gamma$ of a Finsler metric on $S^2$ is reversible if the curve $t\mapsto \gamma(-t)$ is a reparametrization of another geodesic $\gamma_r$ and if, in addition, the inverse rotation numbers $\rho(\gamma)$ and $\rho(\gamma_r)$ coincide. The geodesics $\gamma$ and $\gamma_r$ determine a link in the unit sphere bundle $F^{-1}(1)$ defined by
\[
  l_\gamma = \{ \dot\gamma(t) \mid t\in\R \} \cup \{ \dot\gamma_r(t) \mid t\in\R \}
\]
where $\gamma$ and $\gamma_r$ are assumed to be parametrized by arc-length.  For example, if the Finsler metric $F$ is itself reversible and it has a simple closed geodesic $\gamma$, then $\gamma$ is reversible. Any $(p,q)$-satellite relative to $\gamma$ distinguishes a homotopy class in $F^{-1}(1)\setminus l_\gamma$.

\begin{corollary}\label{cor-geod3}
Let $F$ be a Finsler metric on $S^2$ admitting a reversible simple closed geodesic $\gamma$, and let $\rho\geq0$ denote its inverse rotation number. Let $p,q\in\Z\setminus0$ satisfy $\gcd(|p|,|q|)=1$. If $p/q \in (\rho,1)\cup (1,\rho)$ then there exists a geodesic $\gamma_{p,q}$ such that its velocity vector $\dot\gamma_{p,q}$ is homotopic in $F^{-1}(1) \setminus l_\gamma$ to the normalized velocity vector of a $(p,q)$-satellite of $\gamma$.
\end{corollary}

The proofs of Corollaries~\ref{cor-geod2}, \ref{cor-geod3} are found in Sections~\ref{sec-t1s2},~\ref{sec-reversibleFinsler} respectively. Under appropriate pinching conditions on the flag curvatures, it is possible to show that certain $(p,q)$-satellites do not exist when $p/q$ is out of the twist interval, see~\cite{global} for a non-existence result of $(1,2)$-satellites. \\

\noindent {\bf Organization of the paper.} In Section~\ref{background} we describe basic facts about the Conley-Zehnder index and pseudo-holomorphic curves. In Section~\ref{section_contact_homology} we recall the definition of cylindrical contact homology in the complement of a Hopf link from~\cite{momin}. Section~\ref{section_computing} is devoted to computing contact homology for special model forms. Theorem~\ref{thm-1} is proved in the non-degenerate case in Section~\ref{proof_non_deg_section} combining the results from the previous sections. In Section~\ref{deg_section} we pass to the degenerate case by a limiting argument. Section~\ref{sec-SO3} is devoted to proving Theorem~\ref{teot} and its applications to geodesics. Proofs of theorems related to contact homology in the complement of the Hopf link are included in the appendix, for completeness. \\

\noindent {\bf Acknowledgement.} This paper was partially developed while the authors visited the Institute for Advanced Study for the thematic year on Symplectic Dynamics. The authors would like to thank the IAS for the hospitality. This material is based upon work supported by the National Science Foundation under agreement No. DMS-0635607. Any opinions, findings and conclusions or recommendations expressed in this material are those of the authors and do not necessarily reflect the views of the National Science Foundation. U.H. was partially supported by CNPq grant 309983/2012-6; A.M. was partially supported through a postdoc at the Purdue University Department of Mathematics. P.S. was partially; P.S. was partially supported by CNPq grant 303651/2010-5 and by FAPESP grant 2011/16265-8.

\section{Background}\label{background}

\subsection{The Conley-Zehnder index in $2$ dimensions}

Here we review the basic facts about the Conley-Zehnder index for symplectic paths in dimension 2. Denoting by $Sp(1)$ the group of $2\times 2$ symplectic matrices, consider the set
\begin{equation*}\label{}
  \Sigma^* = \left\{ \varphi : [0,1] \to Sp(1) \text{ is piecewise smooth} \mid \varphi(0) = I,\ \det \left[\varphi(1) - I\right] \not= 0 \right\}.
\end{equation*}
Our convention is that piecewise smooth functions are always continuous. Throughout this Section we may freely identify $\R^2\simeq \C$ via the isomorphism $(x,y) \mapsto x+iy$.

\subsubsection{The axiomatic characterization}

According to~\cite{fols}, the Conley-Zehnder index can be axiomatically characterized as follows.

\begin{theorem}\label{axioms_cz}
There exists a unique surjective map $\mu : \Sigma^* \rightarrow \Z$ satisfying
\begin{itemize}
 \item {\bf Homotopy:} If $\varphi_s$ is a homotopy of arcs in $\Sigma^*$ then $\mu\left(\varphi_s\right)$ is constant.
 \item {\bf Maslov index:} If $\psi:\left(\R/\Z,0\right) \rightarrow \left(Sp(1),I\right)$ is a loop  and $\varphi\in\Sigma^*$ then $\mu(\psi\varphi) = 2\maslov(\psi) + \mu(\varphi)$.
 \item {\bf Invertibility:} If $\varphi \in \Sigma^*$ and $\varphi^{-1}(t) := \varphi(t)^{-1}$ then $\mu\left(\varphi^{-1}\right) = - \mu(\varphi)$.
 \item {\bf Normalization:} $\mu\left(t\mapsto e^{i\pi t}\right) = 1$.
\end{itemize}
\end{theorem}

We shall need more concrete descriptions of the index $\mu$.

\subsubsection{A geometric description}\label{geom_description_section}

If $\varphi:([0,1],\{0\}) \to (Sp(1),I)$ is a piecewise smooth path, consider the unique piecewise smooth functions $r,\theta:[0,1]\times[0,1] \to \R$ satisfying $\varphi(t) e^{i2\pi s} = r(t,s)e^{i\theta(t,s)}$, $r(t,s)>0$ and $\theta(0,s)=2\pi s$, for every $t$ and $s$. Here we identify $\R^2$ with $\C$. Let $\Delta:[0,1]\to \R$ be the piecewise smooth function defined by $2\pi\Delta(s) = \theta(1,s)-2\pi s$ and we consider the winding interval
\begin{equation}\label{winding_interval}
  I(\varphi) = \{ \Delta(s) \mid s\in [0,1] \}.
\end{equation}
It is possible to show that $I(\varphi)$ has length strictly less than $1/2$ and $\partial I(\varphi) \cap \Z \neq\emptyset \Rightarrow \varphi \not\in\Sigma^*$. The first fact is proved in~\cite[Appendix]{fols} and the second fact is proved in~\cite[Section 2.1]{hry}. If $\varphi \in \Sigma^*$ then define
\begin{equation}\label{geom_cz_def}
  \mu(\varphi) = \ \left\{ \begin{aligned} &2k \ \text{ if } k \in I(\varphi) \\ &2k+1 \ \text{ if } I(\varphi) \subset (k,k+1). \end{aligned} \right.
\end{equation}
Then $\mu$ satisfies the axioms of Theorem~\ref{axioms_cz}.

The path $\varphi$ can be continuously extended to all of $[0,+\infty)$ by
\begin{equation}\label{extension_varphi}
  t\mapsto \varphi(t-\lfloor t\rfloor)\varphi(1)^{\lfloor t\rfloor}
\end{equation}
where $\lfloor t \rfloor$ denotes the unique integer satisfying $\lfloor t\rfloor \leq t <\lfloor t\rfloor+1$. If $\varphi(1)$ has no roots of unity in its spectrum then for each integer $k\geq 1$ the path $\varphi^{(k)}(t) = \varphi(kt)$, $t\in[0,1]$, belongs to $\Sigma^*$. The following lemma is well-known and easy to check using the above description of the index, the argument is implicit in~\cite[Appendix]{fols}.

\begin{lemma}\label{cz_index_iterations_lemma}
Suppose $\varphi(1)$ has no roots of unity in its spectrum. The following assertions hold.
\begin{itemize}
  \item If $\sigma(\varphi(1)) \cap \R = \emptyset$ then $\exists \alpha\not\in \Q$ such that  $I(\varphi^{(k)}) \subset (\lfloor k\alpha \rfloor, \lfloor k\alpha \rfloor+1)$ and $\mu(\varphi^{(k)}) = 2 \lfloor k\alpha \rfloor + 1$, $\forall k\geq 1$.
  \item If $\sigma(\varphi(1)) \subset (0,+\infty)$ then $\exists l\in \Z$ such that $l\in I(\varphi)$ and $\mu(\varphi^{(k)}) = 2kl, \ \forall k\geq 1$.
  \item If $\sigma(\varphi(1)) \subset (-\infty,0)$ then $\exists l\in \Z$ such that $l+1/2 \in I(\varphi)$ and $\mu(\varphi^{(k)}) = k(2l+1)$, $\forall k\geq 1$. Moreover
      \[
        \begin{aligned}
          & k \in 2\Z+1 \Rightarrow I(\varphi^{(k)}) \subset (\lfloor k(l+1/2) \rfloor, \lfloor k(l+1/2) \rfloor + 1) \\
          & k \in 2\Z \Rightarrow k(l+1/2) \in I(\varphi^{(k)}).
        \end{aligned}
      \]
\end{itemize}
\end{lemma}

\subsubsection{An analytic description}\label{analytical_description_section}

Let $\varphi : ([0,1],\{0\}) \to (Sp(1),I)$ be a piecewise smooth map. The path of symmetric matrices $S = -i\dot\varphi\varphi^{-1}$ is piecewise continuous, where we identify $$ i \simeq \begin{pmatrix} 0 & -1 \\ 1 & 0 \end{pmatrix}. $$ As is explained in~\cite{props2},
\begin{equation}\label{}
  L = -i\partial_t - S
\end{equation}
is an unbounded self-adjoint operator in $L^2(\R/\Z,\R^2)$ with domain $W^{1,2}(\R/\Z,\R^2)$. Its spectrum, which is discrete, consists of real eigenvalues accumulating only at $\pm\infty$. Geometric and algebraic multiplicities coincide, see Chapter III \S6 from~\cite{kato} for the definition of algebraic multiplicity. An eigenvector $v$ does not vanish unless $v\equiv0$. Writing $v(t) = \rho(t)e^{i\vartheta(t)}$ we define its winding number as $\wind(v) = (\vartheta(1)-\vartheta(0))/2\pi$. This definition does not depend on the choice of the eigenvector for a given eigenvalue, thus we denote it by $\wind(\nu)$ with $\nu\in \sigma(L)$, see~\cite{props2}. For every $k\in\Z$ there are exactly two eigenvalues, counting multiplicities, with winding number $k$, and $\nu_0\leq\nu_1 \Rightarrow \wind(\nu_0)\leq \wind(\nu_1)$ if $\nu_0,\nu_1\in\sigma(L)$.

Following~\cite{props2} we distinguish two eigenvalues
\[
  \nu^{<0} = \max \{\nu \in \sigma(L) \mid \nu<0 \}, \ \nu^{\geq 0} = \min \{ \nu\in\sigma(L) \mid \nu\geq 0 \}
\]
and denote $\wind^-(L) = \wind(\nu^{<0})$, $\wind^+(L) = \wind(\nu^{\geq 0})$. Later $\nu^{<0},\nu^{\geq 0}$ will be referred as the extremal eigenvalues and $\wind^\pm(L)$ will be called the extremal asymptotic windings. Defining $p(L) = 0$ if $\wind^-(L)=\wind^+(L)$ or $p(L) = 1$ if $\wind^-(L)<\wind^+(L)$ we set
\begin{equation}\label{analytical_cz_def}
  \tilde \mu(\varphi) = 2\wind^-(L) + p(L).
\end{equation}

\begin{lemma}\label{extremal_interval_comparison}
If $I(\varphi)$ is the winding interval~\eqref{winding_interval} then $\wind^-(L) < \max I(\varphi)$ and $\wind^+(L) \geq \min I(\varphi)$, with strict inequality when $\varphi \in \Sigma^*$. Moreover, if $\varphi(1)$ is positive hyperbolic then $\wind^-(L)=\wind^+(L)$.
\end{lemma}

\begin{proof}
Write $I(\varphi)=[a,b]$, fix some $\nu \in \sigma(L) \cap (-\infty,0)$ and choose an eigenvector $v(t)$ for $\nu$. We consider $u(t)= \varphi(t) v(0)$, $z(t) = v(t)\overline{u(t)}$ and choose a piecewise smooth $\vartheta(t) \in \R$ such that $z(t) \in \R^+e^{i\vartheta(t)}$. Then $z$ satisfies $$ -i\dot z = (Sv)\bar u - v(\overline{Su}) + \nu z. $$ Whenever $v\in \R u$ we have $z\in \R$ and $(Sv)\bar u - v(\overline{Su}) \in i\R$, implying $\Re[-i\dot z/z] = \dot\vartheta = \nu < 0$ at these points (both lateral limits). So the total angular variation $\vartheta(1)-\vartheta(0)$ of $z$ is strictly negative since $u(0)=v(0)$, in other words, the total angular variation of $v$ is strictly smaller than that of $u$, which implies $\wind(\nu)<b$. The other inequalities are proved analogously.

To prove the assertion about the positive hyperbolic case, consider for any $\mu\in\R$ the winding interval $I_\mu$ associated to the differential equation $-i\dot u - Su = \mu u$. In particular, $I(\varphi) = I_0$. We claim that $\mu$ is an eigenvalue of $L=-i\partial_t-S$ if $\partial I_\mu \cap \Z \neq \emptyset$, in which case $\exists k\in \Z$ such that $\{k\} = \partial I_\mu \cap \Z$ and $\wind(\mu) = k$. Indeed, the fundamental solution $\varphi_\mu(t)$ of $-i\partial_t-S=\mu$ is a path in $Sp(1)$ starting at the identity. Define smooth functions $r,\theta:\R\times\R/\Z\times[0,1] \to \R$ by requiring
\[
\varphi_\mu(t)e^{i2\pi s} = r(\mu,s,t)e^{i2\pi\theta(\mu,s,t)}, \ \ r(\mu,s,t)>0, \ \ \theta(\mu,s,0) \in [0,1).
\]
We have $I_\mu = \{\theta(\mu,s,1)-\theta(\mu,s,0) \mid s\in\R/\Z \}$. Assume that $k\in\partial I_\mu \cap \Z$. If $s_0$ satisfies $\theta(\mu,s_0,1)-\theta(\mu,s_0,0) = k$ we must have $\partial_s\theta(\mu,s_0,1)=1$. Now we claim that $r(\mu,s_0,1)=1$ is an eigenvalue of $\varphi_\mu(1)$, which implies that $\mu$ is an eigenvalue of $L$ with winding $k$. We compute $\varphi_\mu(1)e^{i2\pi s_0} = r(\mu,s_0,1)e^{i2\pi\theta(\mu,s_0,1)} = r(\mu,s_0,1)e^{i2\pi s_0}e^{i2\pi k} = r(\mu,s_0,1)e^{i2\pi s_0}$ and
\[
\begin{aligned}
\varphi_\mu(1)&ie^{i2\pi s_0} = \frac{1}{2\pi} \left. \frac{d}{ds} \right|_{s=s_0} \varphi_\mu(1)e^{i2\pi s} = \frac{1}{2\pi} \left. \frac{d}{ds} \right|_{s=s_0} r(\mu,s,1)e^{i2\pi\theta(\mu,s,1)} \\
&= \frac{1}{2\pi} \partial_sr(\mu,s_0,1) e^{i2\pi \theta(\mu,s_0,1)} + \frac{1}{2\pi} i 2\pi \partial_s\theta(\mu,s_0,1)r(\mu,s_0,1)e^{i2\pi\theta(\mu,s_0,1)} \\
&= \frac{1}{2\pi} \partial_sr(\mu,s_0,1) e^{i2\pi s_0} + r(\mu,s_0,1)ie^{i2\pi s_0}
\end{aligned}
\]
Hence $1=\det \varphi_\mu(1) = r(\mu,s_0,1)^2 \Rightarrow r(\mu,s_0,1)=1$, and the claim follows.

Assume $m=\wind^-(L)<\wind^+(L)$. Hence $\wind^+(L)=m+1$, $m\in I_{\nu^{<0}}$ and $m+1 \in I_{\nu^{\geq0}}$. $\partial I_\mu \cap \Z = \emptyset$ $\forall \mu\in(\nu^{<0},\nu^{\geq0})$ because $L$ has no eigenvalues in $(\nu^{<0},\nu^{\geq0})$. Since $I_\mu$ varies continuously with $\mu$ and $|I_\mu|<1/2 \ \forall\mu$, we must have $m = \min I_{\nu^{<0}}$, $m+1 = \max I_{\nu^{\geq0}}$ and $I_\mu \subset (m,m+1) \ \forall \mu\in(\nu^{<0},\nu^{\geq0})$. This prevents $\varphi(1)$ from being positive hyperbolic since, otherwise, $\nu^{\geq0}>0$ and $I_0$ would contain an integer.
\end{proof}

Lemma~\ref{extremal_interval_comparison} and the non-trivial fact $p = \wind^+-\wind^-$, which was already used in the above lemma, imply together that $\mu(\varphi) = \tilde\mu(\varphi) \ \forall \varphi \in \Sigma^*$, where $\mu$ and $\tilde\mu$ are defined in~\eqref{geom_cz_def} and~\eqref{analytical_cz_def} respectively.

\begin{corollary}\label{extremal_winding_rotation}
Let $\varphi:([0,1],\{0\})\to (Sp(1),I)$ be a piecewise smooth path such that $\varphi(1)$ has no roots of unity in the spectrum. Extending $\varphi$ to $[0,+\infty)$ by~\eqref{extension_varphi}, consider the paths $\varphi^{(k)}(t) = \varphi(kt)$ and their associated self-adjoint operators $L^{(k)}$. If $\sigma(\varphi(1)) \cap \R = \emptyset$ then
\[
  \begin{array}{cccc}
    \wind^-(L^{(k)}) = \lfloor k\alpha \rfloor & \text{and} & \wind^+(L^{(k)}) = \lfloor k\alpha \rfloor + 1 & \forall k\geq 1
  \end{array}
\]
where $\alpha\not\in \Q$ is the unique number satisfying $\mu(\varphi^{(k)}) = 2\lfloor k\alpha \rfloor+1, \ \forall k$. If $\varphi(1)$ is hyperbolic, $\sigma(\varphi(1)) \subset (0,+\infty)$ and $l\in\Z$ satisfies $\mu(\varphi^{(k)}) = 2kl, \ \forall k$ then
\[
  \wind^-(L^{(k)}) = \wind^+(L^{(k)}) = kl \ \ \forall k \geq 1.
\]
If $\varphi(1)$ is hyperbolic, $\sigma(\varphi(1)) \subset (-\infty,0)$, and $l\in \Z$ satisfies $\mu(\varphi^{(k)}) = k(2l+1), \ \forall k$ then
\[
  \begin{aligned}
    & k \geq 1 \text{ is even } \Rightarrow \wind^-(L^{(k)}) = \wind^+(L^{(k)}) = k(l+1/2), \\
    & k \geq 1 \text{ is odd } \Rightarrow \left\{ \begin{aligned} & \wind^-(L^{(k)}) = \lfloor k(l+1/2) \rfloor \\ & \wind^+(L^{(k)}) = \lfloor k(l+1/2) \rfloor + 1. \end{aligned} \right.
  \end{aligned}
\]
\end{corollary}

\subsubsection{Mean index and rotation number}\label{mean_rotation_section}

Let $\varphi:\R\to Sp(1)$, $\varphi(0)=I$, be the solution of a 1-periodic linear Hamiltonian system $\dot\varphi = iS\varphi$, that is, $S(t)$ is a 1-periodic smooth path of symmetric matrices. This is equivalent to $\varphi(t+1)=\varphi(t)\varphi(1)$ for all $t$.

As in the geometrical description of the index in Section~\ref{geom_description_section}, consider the unique smooth $\theta:\R\times \R\to \R$ satisfying $\varphi(t)e^{i2\pi s} \in \R^+ e^{i\theta(t,s)}$ and $\theta(0,s)=2\pi s$. Then $\theta(t,s+1) = \theta(t,s)+2\pi$ so that $s\mapsto f(s) := \theta(1,s)/2\pi$ satisfies $f(s+1)=f(s)+1$ and induces an orientation preserving self-diffeomorphism of $\R/\Z$. It can be written in the form $f(s) = s+ \Delta(s)$, where $\Delta(s)$ is a 1-periodic smooth function used to define the winding interval in~\eqref{winding_interval}: $I(\varphi|_{[0,1]}) = \{ \Delta(s) \mid s\in [0,1] \}$. The associated rotation number
\begin{equation}\label{rotation_number_paths}
  \rho(\varphi) = \lim_{k\to +\infty} \frac{\Delta(s) + \Delta(f(s)) + \dots + \Delta(f^{k-1}(s))}{k}
\end{equation}
which is independent of $s\in[0,1]$, is well-defined and of particular interest to us.

As before we may consider the iterated path $\varphi^{(k)}(t) = \varphi(kt)$, $t\in[0,1]$, and the associated angular variation $s\mapsto \Delta^{(k)}(s)$. By the 1-periodicity of $S$ we must have $$ \Delta^{(k)}(s) = \Delta(s) + \Delta(f(s)) + \dots + \Delta(f^{k-1}(s)) $$ so that $\Delta^{(k)}(s)/k \to \rho$ as $k\to+\infty$, $\forall s$. In view of formula~\eqref{analytical_cz_def} and Lemma~\ref{extremal_interval_comparison} we have that $2\Delta^{(k)}(s)-\mu(\varphi^{(k)})$ is uniformly bounded in $k$, for each fixed $s$. Thus the so-called mean index
\begin{equation}\label{mean_index_def}
  \bar \mu(\varphi) = \lim_{k\to \infty} \frac{\mu(\varphi^{(k)})}{k}
\end{equation}
is well-defined and

\begin{lemma}\label{mean_rotation_lemma}
$\bar\mu(\varphi) = 2\rho(\varphi)$.
\end{lemma}

\subsubsection{Conley-Zehnder index and transverse rotation number of periodic orbits}\label{cz_rotation_orbits}

Consider the flow $\phi_t$ of the Reeb vector field $X_\lambda$ associated to a contact form $\lambda$ on the 3-manifold $V$. Throughout the rest of the paper we assume that any closed orbit $P$ has a marked point in its geometric image, and when we write $P=(x,T)$ it will be understood that $x(t)$ is chosen so that the marked point is $x(0)$.

The Reeb flow preserves $\lambda$, so we get a path of $d\lambda$-symplectic linear maps $d\phi_t: \xi_{x(0)} \to \xi_{x(t)}$ when $x(t)$ is a trajectory of $X_\lambda$. $P = (x,T)$ is non-degenerate if $1$ is not in the spectrum of $d\phi_T:\xi_{x(0)} \to \xi_{x(0)}$, and $\lambda$ will be called non-degenerate if this holds for every $P \in \P(\lambda)$; here $\P(\lambda)$ is the set defined in Section~\ref{section_main_result_major}. This is a residual condition in the set of contact forms on $V$ equipped with the $C^\infty$-topology.

Let $P=(x,T)$ be a closed Reeb orbit. The contact structure $\xi$ is given by~\eqref{contact_structure}, and we denote $x_T(t) = x(Tt)$. The orbit $(x,kT)$ is denoted by $P^k$. Fix a homotopy class $\beta$ of smooth $d\lambda$-symplectic trivializations of the bundle $(x_T)^*\xi$. A trivialization $\Psi:(x_T)^*\xi \to \R/\Z\times \R^2$ in class $\beta$ can be used to represent the linear maps $d\phi_{Tt} : \xi_{x(0)} \to \xi_{x(Tt)}$ as a path of symplectic matrices
\[
  \begin{array}{cc}
    \varphi : \R \to Sp(1), & \varphi(t) = \Psi_t \circ d\phi_{Tt} \circ (\Psi_0)^{-1}.
  \end{array}
\]
It satisfies $\varphi(t+1) = \varphi(t)\varphi(1) \ \forall t$, that is, $\varphi$ solves a 1-periodic linear Hamiltonian system as in Section~\ref{mean_rotation_section}. We define the transverse rotation number of $P$ with respect to the homotopy class $\beta$ as
\begin{equation}\label{eq_defn_rot_number}
  \rho(P,\beta) = \rho(\varphi)
\end{equation}
where $\rho(\varphi)$ is the rotation number~\eqref{rotation_number_paths}. Note that its value depends only on the homotopy class $\beta$ of the chosen trivialization, since for two trivializations in class $\beta$ the numerator inside the limit in~\eqref{rotation_number_paths} will differ by a quantity uniformly bounded in $k$. We also define
\begin{equation}\label{}
  \mu_{CZ}(P,\beta) = \mu(\varphi)
\end{equation}
where $\mu$ is the index for symplectic paths discussed in Section~\ref{analytical_description_section}. The class $\beta$ induces a homotopy class of $d\lambda$-symplectic trivializations of $(x_{kT})^*\xi$ for every $k\geq 1$ in an obvious way, which we denote by $\beta^k$. Lemma~\ref{mean_rotation_lemma} implies $$ \rho(P,\beta) = \lim_{k\to \infty} \frac{1}{2k}\mu_{CZ}(P^k,\beta^k). $$

\begin{remark}[Winding numbers]\label{winding_numbers}
Let $E$ be an oriented rank-2 real vector bundle over $\R/\Z$. If $Z$ and $W$ are non-vanishing continuous sections of $E$ then the relative winding number $\wind(W,Z) \in \Z$ is defined as follows. Let $Z'$ be any non-vanishing continuous section such that $\{Z(t),Z'(t)\}$ is an oriented basis for $E_t$, $\forall t$. Then $W(t) = a(t)Z(t) + b(t)Z'(t)$ for unique continuous functions $a,b : \R/\Z \to \R$, and we set $\wind(W,Z) = \theta(1)-\theta(0) \in \Z$, where $\theta \in C^0([0,1],\R)$ satisfies $a+ib \in \R^+e^{i2\pi\theta}$. When $E$ is endowed with a symplectic or complex structure then we use the induced orientation to compute relative winding numbers. Note also that $\wind(W,Z)$ depends only on the homotopy classes of non-vanishing sections of both $W$ and $Z$.
\end{remark}

If a trivialization $\Psi'$ in another class $\beta'$ is used to represent $d\phi_{Tt}$, we get numbers $\rho(P,\beta')$ and $\mu_{CZ}(P,\beta')$ satisfying $$ \begin{array}{ccc} \rho(P,\beta') = \rho(P,\beta) + m & \text{ and } & \mu_{CZ}(P,\beta') = \mu_{CZ}(P,\beta) + 2m \end{array} $$ where $m\in \Z$ is the Maslov index of the loop of symplectic maps $\Psi'_t \circ (\Psi_t)^{-1}$. Note that $m = \wind((\Psi_t)^{-1} \cdot u, (\Psi_t')^{-1} \cdot u)$ for any fixed non-zero vector $u\in \R^2$.

\subsection{Pseudo-holomorphic curves}

We take a moment to review the basics of pseudo-holomorphic theory in symplectic cobordisms. In the following discussion we fix a closed co-oriented contact 3-manifold $(V,\xi)$.

\subsubsection{Cylindrical almost-complex structures}\label{cyl_alm_cpx_str}

The space $\xi^\bot\setminus 0$, the annihilator of $\xi$ in $T^*V$ minus the zero section, can be naturally endowed with the symplectic form $\omega_\xi = d\alpha_{\rm taut}$, where $\alpha_{\rm taut}$ is the tautological 1-form on $T^*V$. The given co-orientation of $\xi$ orients the line bundle $TV/\xi \to V$ and, consequently, also $(TV/\xi)^* \simeq \xi^\bot$. We single out the component $W_{\xi} \subset \xi^\bot\setminus 0$ consisting of positive covectors, which we call the symplectization of $(V,\xi)$.

A choice of contact form $\lambda$ on $V$ satisfying~\eqref{contact_structure} and inducing the co-orientation of $\xi$ induces a symplectomorphism
\begin{equation}\label{coordinates_symplectization}
  \begin{aligned}
    \Psi_\lambda : (W_\xi,\omega_\xi) &\to (\R\times V,d(e^a\lambda)) \\
    \theta &\mapsto \left( \ln \frac{\theta}{\lambda} , \tau(\theta) \right)
  \end{aligned}
\end{equation}
where $a$ denotes the $\R$-coordinate and $\tau:T^*V \to V$ is the bundle projection. The free additive $\R$-action on the right side corresponds to $(c,\theta) \mapsto e^c\theta$ on the left side.

The bundle $\xi \to V$ becomes symplectic with the bilinear form $d\lambda$. We will denote by $\J_+(\xi)$ the set of $d\lambda$-compatible complex structures on $\xi$, which will be endowed with the $C^\infty$-topology. It does not depend on the choice of positive contact form $\lambda$ satisfying~\eqref{contact_structure}. As is well-known, $\J_+(\xi)$ is non-empty and contractible. Any $J \in \J_+(\xi)$ and $\lambda$ as above induce an almost complex structure $\jtil$ on $\R\times V$ by
\begin{equation}\label{alm_cpx_str_1}
  \begin{array}{cc}
    \jtil \cdot \partial_a = X_{\lambda}, & \jtil|_\xi = J
  \end{array}
\end{equation}
where $\xi$ is seen as a $\R$-invariant subbundle of $T(\R\times V)$. It is compatible with $d(e^a\lambda)$. The pull-back $\jhat = (\Psi_\lambda)^*\jtil$ is then a $\omega_\xi$-compatible almost complex structure on $W_\xi$. The set of $\jhat$ that arise in this way will be denoted by $\J(\lambda)$.

\subsubsection{Cylindrical ends}\label{cyl_ends_alm_cpx_str}

The fibers of $\tau: W_{\xi} \to V$ can be ordered in the following way: for given $\theta_0,\theta_1 \in \tau^{-1}(x)$, we write $\theta_0 \prec \theta_1$ (resp. $\theta_0 \preceq \theta_1$) when $\theta_1 / \theta_0 > 1$ (resp. $\theta_1 / \theta_0 \geq 1$). Given two positive contact forms $\lambda_-,\lambda_+$ for $\xi$, we define $\lambda_- \prec \lambda_+$ if $\lambda_-|_x \prec \lambda_+|_{x}$ pointwise and, in this case, we set
\[
 \overline{W}(\lambda_-,\lambda_+) = \left \{ \theta \in W_{\xi} \mid \lambda_-|_{\tau(\theta)} \preceq \theta \preceq \lambda_+|_{\tau(\theta)} \right\}
\]
which is an exact symplectic cobordism between $(V,\lambda_-),(V, \lambda_+)$. Let
\[
  \begin{split}
    W^-(\lambda_-) &= \left \{ \theta \in W_{\xi} \mid \theta \preceq \lambda_-|_{\tau(\theta)} \right\}, \\
    W^+(\lambda_+) &= \left \{ \theta \in W_{\xi} \mid \lambda_+|_{\tau(\theta)} \preceq \theta \right\}.
  \end{split}
\]
It follows that
\[
  W_{\xi} = W^-(\lambda_-) \bigcup_{\substack{\partial^+ W^-(\lambda_-) = \\ \partial^- \overline{W}(\lambda_-,\lambda_+)}} \overline{W}(\lambda_-,\lambda_+) \bigcup_{\substack{\partial^+ \overline{W}(\lambda_-,\lambda_+) \\ =  \partial^- W^+(\lambda_+)}} W^+(\lambda_+).
\]

An almost-complex structure $\bar J$ satisfying
\begin{itemize}
 \item $\bar J$ coincides with $\widehat{J}_+ \in \J(\lambda_+)$ on a neighborhood of $W^+(\lambda_+)$,
 \item $\bar J$ coincides with $\widehat{J}_- \in \J(\lambda_-)$ on a neighborhood of $W^-(\lambda_-)$,
 \item $\bar J$ is $\omega_{\xi}$-compatible
\end{itemize}
is an almost-complex structure with cylindrical ends. The set of such almost-complex structures will be denoted by $\J(\jhat_-,\widehat{J}_+)$.  It is well-known that this is a non-empty contractible set.  For $\bar J \in \J(\widehat{J}_-,\widehat{J}_+)$ the almost-complex manifold $(W_\xi,\bar J)$ is said to have \emph{cylindrical ends} $W^+(\lambda_+)$ and $W^-(\lambda_-)$.

\subsubsection{Splitting almost-complex structures}\label{splitting_alm_cpx_str}

Suppose we are given positive contact forms $\lambda_- \prec \lambda \prec \lambda_+$ for $\xi$. Let $\jhat_- \in \J(\lambda_-)$, $\jhat \in \J(\lambda)$ and $\jhat_+ \in \J(\lambda_+)$ be cylindrical almost-complex structures, and consider almost-complex structures $J_1 \in \J(\jhat_-,\jhat)$, $J_2 \in \J(\jhat,\jhat_+)$. Let us denote by $g_c(\theta) = e^c\theta$ the $\R$-action on $W_\xi$. Then there is a smooth family of almost-complex structures $\bar J_R$, $R \geq 0$, given by
\[
  \bar J_R = \left\{ \begin{aligned} &(g_{-R})^*J_2 \ \mbox{ on } W^+(\lambda) \\ &(g_R)^*J_1 \ \mbox{ on } W^-(\lambda) \end{aligned} \right.
\]
which is smooth since $\jhat$ is $\R$-invariant. We may denote $\bar J_R = J_1 \circ_R J_2$ if the dependence on $J_1$ and $J_2$ needs to be made explicit.

Note that if $\epsilon_0>0$ is small enough then $J_1 \circ_R J_2 \in \J(\jhat_-,\jhat_+)$ for all $0 < R \leq \epsilon_0$. For each $R>0$ we take a function $\varphi_R:\R\to \R$ satisfying $\varphi_R(a) = a+R$ if $a\leq -R-\epsilon_0$, $\varphi_R(a) = a-R$ if $a\geq R+\epsilon_0$ and $\varphi_R' > 0$ everywhere. The family $\{\varphi_R\}$ can always be arranged so that $\sup_{R,a} |\varphi_R'(a)| \leq 1$ and $$ \inf \{ \varphi_R'(a) \mid a\in (-\infty,-R] \cup [R,+\infty) \text{ and } R>0 \} \geq \frac{1}{2}. $$ In particular, the inverse function $\varphi_R^{-1}$ has derivative bounded in the intervals $(-\infty,\varphi_R(-R)]$ and $[\varphi_R(R),+\infty)$ uniformly in $R$. Consider the diffeomorphisms $\psi_R : \R\times V \to \R\times V$, $\psi_R(a,x) = (\varphi_R(a),x)$ and
\begin{equation}\label{diffeo_Phi_R}
  \Phi_R = \Psi_\lambda^{-1} \circ \psi_R \circ \Psi_\lambda : W_{\xi} \to W_{\xi}.
\end{equation}
It is straightforward to check that
\begin{equation}\label{jprime_R}
  J'_R := (\Phi_R)_*(J_1 \circ_R J_2)
\end{equation}
belongs to $\J(\jhat_-,\jhat_+)$, for every $R$ large.

\subsubsection{Finite-energy curves in symplectizations}\label{curves_symplectizations}

Let us fix a positive contact form $\lambda$ satisfying~\eqref{contact_structure}.

Consider the set $\Lambda = \{ \phi : \R\to\R \mid \phi(\R) \subset [0,1], \ \phi'\geq 0 \}$. For each $\phi \in \Lambda$ we denote by $\lambda_\phi$ the 1-form $(\Psi_\lambda)^*(\phi\lambda)$, where $\phi\lambda$ denotes the 1-form $(a,x) \mapsto \phi(a)\lambda|_x$ on $\R\times V$ and $\Psi_\lambda$ is the diffeomorphism~\eqref{coordinates_symplectization}.

\begin{definition}[Hofer~\cite{93}]\label{fe_curve_symp}
Let $(S,j)$ be a closed Riemann surface, $\Gamma \subset S$ be finite and $\jhat \in \J(\lambda)$. A finite-energy $\jhat$-curve is a pseudo-holomorphic map $$ \util : (S\setminus \Gamma,j) \to ( W_\xi,\jhat ) $$ satisfying
\begin{equation}\label{energy_symp}
  0< E(\util) = \sup_{\phi\in\Lambda} \int_{S\setminus\Gamma} \util^*d\lambda_\phi < \infty.
\end{equation}
The quantity $E(\util)$ is called the Hofer-energy.
\end{definition}

Each integrand in the definition of the energy is non-negative and $\util$ is constant when $E(\util)=0$. The elements of $\Gamma$ are the so-called punctures.

\begin{remark}[Cylindrical coordinates]\label{cyl_coord}
Fix $z\in\Gamma$ and choose a holomorphic chart $\psi:(U,z) \to (\psi(U),0)$, where $U$ is a neighborhood of $z$. We identify $[s_0,+\infty) \times \R/\Z$ with a punctured neighborhood of $z$ via $(s,t) \simeq \psi^{-1}(e^{-2\pi(s+it)})$, for $s_0 \gg 1$, and call $(s,t)$ positive cylindrical coordinates centered at $z$. We may also identify $(s,t) \simeq \psi^{-1}(\est)$ where $s<-s_0$ and, in this case, $(s,t)\in (-\infty,-s_0]\times \R/\Z$ are negative coordinates. In both cases we write $\util(s,t) = \util\circ\psi^{-1}(e^{-2\pi(s+it)})$ or $\util(s,t) = \util\circ\psi^{-1}(e^{2\pi(s+it)})$.
\end{remark}

Let $(s,t)$ be positive cylindrical coordinates centered at some $z\in\Gamma$, and write $ \Psi_\lambda \circ \util(s,t) = (a(s,t),u(s,t))$. $E(\util)<\infty$ implies that
\begin{equation}\label{mass_defn}
  m = \lim_{s\to+\infty} \int_{\{s\}\times \R/\Z} u^*\lambda
\end{equation}
exists. This number is the mass of $\util$ at $z$, and does not depend on the choice of coordinates. The puncture $z$ is called positive, negative or removable when $m>0$, $m<0$ or $m=0$ respectively, and $\util$ can be smoothly extended to $(S\setminus\Gamma)\cup\{z\}$ when $z$ is removable. Moreover, $a(s,t) \to \epsilon\infty$ as $s\to+\infty$, where $\epsilon$ is the sign of $m$.

\subsubsection{Finite-energy curves in cobordisms}\label{curves_cobordisms}

Let $\lambda_- \prec \lambda_+$ be positive contact forms for $\xi$ and consider $\jhat_\pm \in \J(\lambda_\pm)$, $\bar J \in \J(\jhat_-,\jhat_+)$. Recall the symplectomorphisms $\Psi_{\lambda_\pm} : (W_\xi,\omega_\xi) \to (\R\times V,d(e^a\lambda_\pm))$, the collection $\Lambda$ and the 1-forms $\lambda_{\pm,\phi}$ on $W_\xi$ with $\phi \in \Lambda$.

\begin{definition}[\cite{sftcomp}]\label{fe_curve_cobordism}
Let $(S,j)$ be a closed Riemann surface and $\Gamma \subset S$ be finite. A finite-energy $\bar J$-curve is a pseudo-holomorphic map $$ \util : (S\setminus \Gamma,j) \to ( W_\xi,\bar J ) $$ satisfying
\begin{equation}\label{energy_cobordism}
  0< E_-(\util) + E_+(\util) + E_0(\util) < \infty
\end{equation}
where the various energies above are defined as
\begin{equation*}\label{}
  \begin{aligned}
    & E_+(\util) = \sup_{\phi\in\Lambda} \int_{\util^{-1}(W^+(\lambda_+))} \util^*d\lambda_{+,\phi} \\
    & E_-(\util) = \sup_{\phi\in\Lambda} \int_{\util^{-1}(W^-(\lambda_-))} \util^*d\lambda_{-,\phi}
  \end{aligned}
\end{equation*}
and
\begin{equation*}\label{}
  E_0(\util) = \int_{\util^{-1}(\overline W(\lambda_-,\lambda_+))} \util^*\omega_\xi.
\end{equation*}
\end{definition}

As before, the elements of $\Gamma$ are called punctures. A puncture $z\in\Gamma$ is called positive if
\begin{itemize}
  \item there exists a neighborhood $U$ of $z$ in $S$ such that $\util(U\setminus \{z\}) \subset W^+(\lambda_+)$,
  \item writing $\Psi_{\lambda_+} \circ \util = (a,u)$ on $U\setminus \{z\}$ we have that $a(\zeta) \to +\infty$ as $\zeta\to z$.
\end{itemize}
Analogously $z$ is called negative if
\begin{itemize}
  \item there exists a neighborhood $U$ of $z$ in $S$ such that $\util(U\setminus \{z\}) \subset W^-(\lambda_-)$,
  \item writing $\Psi_{\lambda_-} \circ \util = (a,u)$ on $U\setminus \{z\}$ we have that $a(\zeta) \to -\infty$ as $\zeta\to z$.
\end{itemize}
Finally $z$ is said to be removable if $\util$ can be smoothly extended to $(S\setminus\Gamma) \cup \{z\}$. It turns out that the set of punctures can be divided into positive, negative and removable, see~\cite{sftcomp}.

\subsubsection{Finite-energy curves in splitting cobordisms}\label{fe_curves_splitting_cob}

As in Section~\ref{splitting_alm_cpx_str} we consider positive contact forms $\lambda_- \prec \lambda \prec \lambda_+$ for $\xi$, select $\jhat_- \in \J(\lambda_-)$, $\jhat \in \J(\lambda)$, $\jhat_+ \in \J(\lambda_+)$, and $J_1 \in \J(\jhat_-,\jhat)$, $J_2 \in \J(\jhat,\jhat_+)$. Then for each $R>0$ we have an almost complex structure $\bar J_1\circ_R \bar J_2$ which takes particular forms in various regions on $W_\xi$:
\begin{itemize}
  \item $\bar J_1\circ_R \bar J_2 = \jhat_+$ on $g_R(W^+(\lambda_+)) = W^+(e^R\lambda_+)$,
  \item $\bar J_1\circ_R \bar J_2 = \jhat$ on $\overline W(e^{-R}\lambda,e^R\lambda)$ and
  \item $\bar J_1\circ_R \bar J_2 = \jhat_-$ on $g_{-R}(W^-(\lambda_-)) = W^-(e^{-R}\lambda_-)$.
\end{itemize}

\begin{definition}[\cite{sftcomp}]\label{fe_curve_split_cobordism}
Let $(S,j)$ be a closed Riemann surface and $\Gamma \subset S$ be finite. A finite-energy $(\bar J_1\circ_R \bar J_2)$-curve is a pseudo-holomorphic map $$ \util : (S\setminus \Gamma,j) \to ( W_\xi,\bar J_1\circ_R \bar J_2 ) $$ satisfying
\begin{equation}\label{energy_splitting}
  0< E_{\lambda_-}(\util) + E_{\lambda_+}(\util) + E_\lambda(\util) + E_{(\lambda,\lambda_+)}(\util) + E_{(\lambda_-,\lambda)}(\util) < \infty
\end{equation}
where
\begin{equation*}\label{}
  \begin{aligned}
    & E_{\lambda_+}(\util) = \sup_{\phi\in\Lambda} \int_{\util^{-1}(W^+(e^R\lambda_+))} \util^*d\lambda_{+,\phi} \\
    & E_{\lambda}(\util) = \sup_{\phi\in\Lambda} \int_{\util^{-1}(\overline W(e^{-R}\lambda,e^R\lambda))} \util^*d\lambda_{\phi} \\
    & E_{\lambda_-}(\util) = \sup_{\phi\in\Lambda} \int_{\util^{-1}(W^-(e^{-R}\lambda_-))} \util^*d\lambda_{-,\phi}
  \end{aligned}
\end{equation*}
and
\begin{equation*}\label{}
  \begin{aligned}
    & E_{(\lambda,\lambda_+)}(\util) = \int_{\util^{-1}(\overline W(e^R\lambda,e^R\lambda_+))} \util^*(e^{-R}\omega_\xi) \\
    & E_{(\lambda_-,\lambda)}(\util) = \int_{\util^{-1}(\overline W(e^{-R}\lambda_-,e^{-R}\lambda))} \util^*(e^R\omega_\xi).
  \end{aligned}
\end{equation*}
Note that all integrands are pointwise non-negative.
\end{definition}

As before punctures are divided into positive, negative and removable, see~\cite{sftcomp}.

\subsubsection{A restricted class of almost-complex structures}\label{restricted_class_alm_cpx_str}

Consider $\jhat_\pm \in \J(\lambda_\pm)$, where $\lambda_\pm = f_\pm\lambda_0$ are positive contact forms on $S^3$ with $\lambda_0$ as in~\eqref{std_Liouville_form}, and $f_\pm \in \F$ satisfy $f_-<f_+$ pointwise. Here $\F$ is the set of functions $f:S^3\to(0,+\infty)$ such that $f\lambda_0$ realizes the standard Hopf link $K_0$ as a pair of closed Reeb orbits. Later we will need to consider the subset
\begin{equation}\label{}
  \J(\jhat_-,\jhat_+ : K_0) \subset \J(\jhat_-,\jhat_+)
\end{equation}
of almost complex structures for which $\tau^{-1}(K_0)$ is a complex submanifold, where $\tau:W_{\xi_0}\to S^3$ is projection onto the base point. It is easy to check that it is non-empty and, when equipped with the $C^\infty$-topology, it is a contractible space.

Note also that if $\lambda = f\lambda_0$ is another contact form for some $f\in\F$ satisfying $f_-<f<f_+$ pointwise, $\jhat \in \J(\lambda)$, $\bar J_1 \in \J(\jhat_-,\jhat:K_0)$ and $\bar J_2 \in \J(\jhat,\jhat_+:K_0)$ then $\tau^{-1}(K_0)$ is also a complex submanifold with respect to $\bar J_1 \circ_R\bar J_2$. Moreover, $J'_R = (\Phi_R)_*(\bar J_1 \circ_R\bar J_2) \in \J(\jhat_-,\jhat_+:K_0)$, where $\Phi_R$ is the map~\eqref{diffeo_Phi_R}.

\subsubsection{Asymptotic operators and asymptotic behavior}\label{asymptotic_behavior_section}

Let $P = (x,T) \in \P(\lambda)$ and denote $x_T(t) = x(Tt)$. Any given $J \in \J_+(\xi)$ induces an inner product for sections of $(x_T)^*\xi$ by
\begin{equation}\label{L2_inner_prod}
\left< \eta,\zeta \right> = \int_0^1 (d\lambda)_{x_T(t)} (\eta(t), J_{x_T(t)} \cdot \zeta(t)) dt
\end{equation}
On the corresponding space of square-integrable sections there is an unbounded self-adjoint operator defined by
\begin{equation}\label{asymp_op_def}
A_P \cdot \eta = J(-\nabla_t\eta + T\nabla_\eta X_\lambda)
\end{equation}
where $\nabla$ is a choice of torsionless connection on $TV$; $A_P$ does not depend on this choice.

Let us fix a homotopy class $\beta$ of $d\lambda$-symplectic trivializations of $(x_T)^*\xi$ and choose some $\Psi$ in class $\beta$. Then $A_P$ is represented as $-J(t)\partial_t - S(t)$, where $J(t)$ is the representation of $(x_T)^*J$ and $S(t)$ is some smooth 1-periodic path of $2\times 2$-matrices.  If $\Psi$ is $(d\lambda,J)$-unitary\footnote{There is always a unitary trivialization in any homotopy class.} then $J(t) \equiv i$ and $S(t)$ is symmetric for all~$t$, so that $A_P$ has all the spectral properties described in Section~\ref{analytical_description_section}. In particular, if $\eta$ is non-trivial and satisfies $A_P \cdot \eta = \nu \eta$ for some eigenvalue $\nu$ of $A_P$, then $v(t) = \Psi_t \cdot \eta(t) \in \R^2$ does not vanish and satisfies $-i\dot v-Sv=\nu v$. Defining a continuous $\vartheta:[0,1] \to \R$ by $v(t) \in\R^+e^{i\vartheta(t)}$ the integer
\begin{equation}\label{}
  \wind(\nu,P,\beta) = \frac{\vartheta(1)-\vartheta(0)}{2\pi}
\end{equation}
does not depend on the choice of $\eta$ in the eigenspace of $\nu$. If $\eta_1,\eta_2 \in \sigma(A_P)$ then $\eta_1\leq\eta_2 \Rightarrow \wind(\nu_1,P,\beta) \leq \wind(\nu_2,P,\beta)$. Moreover, if $\beta'$ is another homotopy class of $d\lambda$-symplectic trivializations and $\Psi'$ is in class $\beta'$ then
\begin{equation}\label{}
  \wind(\nu,P,\beta') = \wind(\nu,P,\beta) + m, \ \ \ \forall \nu \in\sigma(A_P)
\end{equation}
where $m$ is the Maslov number of the loop $\Psi'_t\circ (\Psi_t)^{-1}$.

We define $\wind^{\geq 0}(P,\beta)$ and $\wind^{<0}(P,\beta)$ to be the winding of the smallest non-negative and largest negative eigenvalues of $A_P$ with respect to $\beta$, respectively. In view of~\eqref{analytical_cz_def} we have
\begin{equation}\label{}
  \mu_{CZ}(P,\beta) = 2\wind^{<0}(P,\beta) + p
\end{equation}
where $p=0$ if $\wind^{\geq 0}(P,\beta)=\wind^{<0}(P,\beta)$ or $p=1$ if not. As a consequence of Corollary~\ref{extremal_winding_rotation} we get

\begin{lemma}\label{inequalities_windings}
Let $P =(x,T) \in \P(\lambda)$ and assume $P^k = (x,kT)$ is non-degenerate $\forall k\geq 1$. If we fix a homotopy class $\beta$ of $d\lambda$-symplectic trivializations of $(x_T)^*\xi$ then
\begin{itemize}
  \item $P$ is elliptic if, and only if, $\rho(P,\beta) = \alpha \not\in \Q$. In this case
      \begin{equation*}\label{}
        \begin{array}{ccc}
          \wind^{\geq 0}(P^k,\beta^k) = \lfloor k\alpha \rfloor +1 & \wind^{<0}(P^k,\beta^k) = \lfloor k\alpha \rfloor & \forall k\geq 1.
        \end{array}
      \end{equation*}
  \item $P$ is hyperbolic with positive Floquet multipliers if, and only if, $\rho(P,\beta) = l \in \Z$. In this case
      \begin{equation*}\label{}
        \begin{array}{cc}
          \wind^{\geq 0}(P^k,\beta^k) = kl = \wind^{<0}(P^k,\beta^k) & \forall k\geq 1.
        \end{array}
      \end{equation*}
  \item $P$ is hyperbolic with negative Floquet multipliers if, and only if, $\rho(P,\beta) = l +1/2$ for some $l \in \Z$. In this case
      \begin{equation*}\label{}
        \begin{aligned}
          & k \text{ is even} \Rightarrow \wind^{<0}(P^k,\beta^k) = \wind^{\geq 0}(P^k,\beta^k) = k(l+1/2)  \\
          & k \text{ is odd} \Rightarrow \left\{ \begin{aligned} & \wind^{<0}(P^k,\beta^k) = \lfloor k(l+1/2) \rfloor \\ & \wind^{\geq 0}(P^k,\beta^k) = \lfloor k(l+1/2) \rfloor + 1. \end{aligned} \right.
        \end{aligned}
      \end{equation*}
\end{itemize}
Here $\beta^k$ denotes the homotopy class of $d\lambda$-symplectic trivializations of $(x_{kT})^*\xi$ induced by $\beta$.
\end{lemma}

\begin{definition}[Martinet Tube]\label{martinet_tube_def}
Let $P = (x,T) \in \P(\lambda)$ and $T_{\rm min}$ be the underlying minimal positive period of $x$. A Martinet tube for $P$ is a pair $(U,\Phi)$ where $U$ is an open neighborhood of $x(\R)$ in $V$ and $\Phi: U \to \R/\Z\times B$ is a diffeomorphism ($B\subset \R^2$ is an open ball centered at the origin) satisfying
\begin{itemize}
  \item $\Phi^*(f(d\theta+xdy)) = \lambda$ where $(\theta,x,y)$ are the coordinates on $\R/\Z\times \R^2$ and the smooth positive function $f$ satisfies $f|_{\R/\Z\times 0}\equiv T_{\rm min}$ and $df|_{\R/\Z\times 0} \equiv 0$.
  \item $\Phi(x(T_{\rm min}t)) = (t,0,0)$.
\end{itemize}
\end{definition}

\begin{remark}
If $P=(x,T)$, $T_{\rm min}$ are as in the above definition and $\eta(t) \in \xi_{x(t)}$, $t\in \R/T_{\rm min}\Z$, is a smooth non-vanishing vector then there exists a Martinet tube $(U,\Phi)$ for $P$ such that $d\Phi_{x(t)} \cdot \eta(t) = \partial_x$ for every $t\in\R/T_{\rm min}\Z$.
\end{remark}

The precise asymptotic behavior of pseudo-holomorphic curves is studied by Hofer, Wysocki and Zehnder when $\lambda$ is non-degenerate. We will now summarize the main results of~\cite{props1}. Consider a non-degenerate contact form $\lambda$ for $\xi$, a closed connected Riemann surface $(S,j)$, a finite subset $\Gamma\subset S$ and a $\jhat \in \J(\lambda)$. Suppose $$ \util:(S,j) \to (W_\xi,\jhat) $$ is a non-constant finite-energy pseudo-holomorphic map.

\begin{theorem}[Hofer, Wysocki and Zehnder]\label{precise_asymptotics}
Let $(s,t)$ be positive holomorphic cylindrical coordinates at $z$ as in Remark~\ref{cyl_coord} if $z$ is a positive puncture, or negative holomorphic cylindrical coordinates at $z$ if it is a negative puncture, and let us write $\Psi_\lambda \circ \util(s,t) = (a(s,t),u(s,t)) \in \R\times V$. Then there exists $P = (x,T) \in \P(\lambda)$ and constants $r,a_0,t_0 \in \R$, $r>0$, such that $u(s,t) \to x(T(t+t_0))$ in $C^\infty$ as $|s|\to\infty$ and
\[
\lim_{|s|\to\infty} e^{r|s|} \left( \sup_t |D^\gamma[a(s,t)-Ts-a_0]|\right) = 0, \ \forall \gamma.
\]
Let $(U,\Phi)$ be a Martinet tube for $P$, so that one finds $s_0 \in \R$ such that $u(s,t) \in U$ when $|s|\geq |s_0|$, and write $\Phi \circ u(s,t) = (\theta(s,t),z(s,t)) \in \R \times \R^2$ (the universal covering of $\R/\Z\times\R^2$). Then
\[
\lim_{|s|\to\infty} e^{r|s|} \left( \sup_t |D^\gamma[\theta(s,t)-k(t+t_0)]|\right) = 0, \ \forall \gamma,
\]
where $k$ is the multiplicity of $P$. Either $(\tau\circ\util)^*d\lambda\equiv0$ for $|s|\gg1$ or the following holds. There exists an eigenvalue $\mu$ for $A_P$, an eigensection $\eta:\R/\Z\to (x_T)^*\xi$ for $\mu$, and functions $\alpha(s) \in \R$, $R(s,t) \in \R^2$ defined for $|s| \gg 1$ such that $\mu>0$ if $z$ is negative, $\mu<0$ if $z$ is positive, and if we represent $\eta(t) \simeq e(t) \in \R^2$ using the coordinates induced by $\Phi$ then, up to rotation of the cylindrical coordinates,
\[
z(s,t) = e^{\int_{s_0}^s \alpha(\tau)d\tau}(e(t) + R(s,t))
\]
for $|s|\gg1$, where $\alpha$ and $R$ satisfy
\[
\begin{aligned}
& \lim_{|s|\to\infty} \sup_t |D^\gamma R(s,t)| = 0, \ \forall \gamma \\
& \lim_{|s|\to\infty} |D^j [\alpha(s)-\mu]| = 0, \ \forall j.
\end{aligned}
\]
\end{theorem}

\begin{remark}
The same asymptotic behavior as described in Theorem~\ref{precise_asymptotics} holds near non-removable punctures of finite-energy curves in cobordisms and splitting corbordisms defined in Sections~\ref{cyl_ends_alm_cpx_str} and~\ref{splitting_alm_cpx_str}, respectively, assuming that the contact forms in the ends are non-degenerate.
\end{remark}

\begin{remark}\label{exact_nature}
By the exact nature of all the $2$-forms appearing in the integrands of the integrals involved in the energy of pseudo-holomorphic maps in cobordisms and in splitting cobordisms, we obtain the following statement: \\

\noindent {\it If $\lambda_- \prec \lambda \prec \lambda_+$ are positive contact forms for $\xi$ then there exists $C>0$ such that the following holds. For every $\jhat_+ \in \J(\lambda_+)$, $\jhat \in \J(\lambda)$, $\jhat_- \in \J(\lambda_-)$, $\bar J_1 \in \J(\jhat_-,\jhat)$, $\bar J_2 \in \J(\jhat,\jhat_+)$, $R>0$ and finite-energy $(\bar J_1 \circ_R \bar J_2)$-holomorphic map $\util$ we have $$ E(\util) \leq C \mathcal A_+(\util) $$ where $\mathcal A_+(\util)$ denotes the sum of the $\lambda_+$-actions of the closed $\lambda_+$-Reeb orbits which are the asymptotic limits of $\util$ at the positive punctures. An analogous statement holds for finite-energy $\bar J_1$-holomorphic maps.}
\end{remark}

\section{Contact homology in the complement of the Hopf link}\label{section_contact_homology}

We will now review the cylindrical contact chain complex for contact forms $h \lambda_0$, $h \in \F$, following~\cite{momin}. For completeness all necessary statements and proofs are included.

Before starting with our constructions we establish some notation. Let $f>0$ be a smooth function on $S^3$ and denote $\lambda = f\lambda_0$. If $P = (x,T) \in \P(\lambda)$ then we denote by $x_T:\R/\Z\to S^3$ the map $t\mapsto x(Tt)$, and $P^k := (x,kT)$, $\forall k\geq 1$. A homotopy class $\beta$ of $d\lambda$-symplectic trivializations of $(x_T)^*\xi_0$ induces a homotopy class of $d\lambda$-symplectic trivializations of $(x_{kT})^*\xi_0$ which is denoted by $\beta^k$ (the $k$-th iterate of $\beta$).

We will be dealing with various tight contact forms on $S^3$, and sometimes we need to indicate the dependence  on the contact form of the invariants $\rho$ and $\mu_{CZ}$ discussed in Section~\ref{cz_rotation_orbits}, and also of the spectral winding numbers described in Section~\ref{asymptotic_behavior_section}. When $P =(x,T) \in \P(\lambda)$ and the homotopy class $\beta$ of $d\lambda$-symplectic trivializations of $(x_T)^*\xi_0$ is given then we may write $\rho(P,\beta,\lambda)$, $\mu_{CZ}(P,\beta,\lambda)$, $\wind(\nu,P,\beta,\lambda)$, $\wind^{\geq 0}(P,\beta,\lambda)$ and $\wind^{<0}(P,\beta,\lambda)$ to stress the dependence on $\lambda$. The symplectic vector bundle $(\xi_0,d\lambda) \to S^3$ is trivial and we fix a global symplectic frame. For every $P = (x,T) \in \P(\lambda)$, the homotopy class of $d\lambda$-symplectic trivializations of $(x_T)^*\xi_0$ induced by this global frame is denoted by $\beta_P$. It does not depend on the particular choice of global frame. Note that $(\beta_P)^k = \beta_{P^k}$. We may write $\rho(P,\lambda)$, $\mu_{CZ}(P,\lambda)$, etc to denote the various invariants computed with respect to the global frame. When $f\in \F$ then $L_0$ and $L_1$ are closed Reeb orbits of $f\lambda_0$, and we denote $$ \theta_i(f) = \rho(L_i,f\lambda_0)-1  \ \ \ (i=0,1) $$ where the rotation number $\rho(L_i,f\lambda_0)$ is computed with respect to the global $d\lambda$-symplectic trivialization of $\xi_0$.

\subsection{The chain complex}\label{chain_complex_subsection}

To define cylindrical contact homology of $$ \lambda = h\lambda_0, \ h\in\F $$ up to action $T$ in the complement of $K_0$ we need to assume certain conditions:
\begin{itemize}
\item[(a)] Every closed Reeb orbit of $\lambda$ with action $\leq T$ is non-degenerate.
\item[(b)] There are no closed Reeb orbits of $\lambda$ in $S^3\setminus K_0$ with action $\leq T$ which are contractible in $S^3 \backslash K_0$.
\item[(c)] The transverse Floquet multipliers of the components $L_0,L_1$ of $K_0$, seen as prime closed Reeb orbits of $\lambda$, are of the form $e^{i2\pi\alpha}$ with $\alpha \not\in \Q$. In particular, every iterate $L_0^n,L_1^n$ is non-degenerate and elliptic.
\end{itemize}

We always identify
\begin{equation}\label{pi1_isomorphism}
\begin{array}{cc}
\pi_1(S^3\setminus K_0,{\rm pt}) \stackrel{\sim}{\to} \Z\times\Z, & [\gamma] \mapsto (p,q)
\end{array}
\end{equation}
where $$ p = \link(\gamma,L_0) \ \text{ and } \ q = \link(\gamma,L_1). $$ Fix a homotopy class of loops in $S^3\setminus K_0$ represented by a relatively prime pair $(p,q)$ of integers, i.e., there exists no integer $k\geq 2$ such that $(p/k,q/k) \in \Z \times \Z$. In particular, no closed loop in this homotopy class can be multiply covered. We also fix a number $T>0$.

Let $\mathcal{P}^{\leq T, (p,q)}(\lambda)$ be the set of closed $\lambda$-Reeb orbits contained in $S^3\setminus K_0$ which represent the homotopy class $(p,q)$ and have action $\leq T$. The field $\Z / 2 \Z$ will be denoted by $\mathbb{F}_2$. Consider, for each $k\in \Z$, the vector space $C_k^{\leq T, (p,q)}(\lambda)$ over $\mathbb{F}_2$ freely generated by closed orbits in $\mathcal{P}^{\leq T,(p,q)}(\lambda)$ of Conley-Zehnder index $k+1$:
\[
C_k^{\leq T, (p,q)}(\lambda) = \bigoplus_{\substack{P \in \mathcal{P}^{\leq T,(p,q)}(\lambda) \\ \mu_{CZ}(P) = k+1}} \mathbb{F}_2 \cdot q_P.
\]
The degree of the orbit $P$, or of the generator $q_P$, is defined as $|P| = |q_P| = \mu_{CZ}(P)-1$. We consider the direct sum over the degrees $k \in \Z$ as a graded vector space.

\begin{remark}
In general for SFT, one cannot use coefficients $\mathbb{F}_2$. But, since we only consider homotopy classes of loops which cannot contain multiply covered orbits, it is possible in this particular case. In fact, since $(p,q)$ is assumed to be a relatively prime pair of integers, all orbits in $\P^{\leq T,(p,q)}(\lambda)$ are simply covered and, consequently, SFT-good. In this way we do not need to consider orientations of moduli spaces of holomorphic curves.
\end{remark}

We turn these graded vector spaces into a chain complex as follows. Select a $d\lambda_0$-compatible complex structure $J: \xi_0 \to \xi_0$, and extend it to $\jhat \in \J(\lambda)$ on $W_{\xi_0}$ as explained in Section~\ref{cyl_alm_cpx_str}. Here $W_{\xi_0} \subset T^*S^3$ is the positive symplectization of $(S^3,\xi_0)$ equipped with its natural symplectic structure $\omega_{\xi_0}$ which is the restriction to $W_{\xi_0}$ of the canonical $2$-form. On $W_{\xi_0}$ there is a free $\R$-action $$ g_c: \theta \mapsto e^c\theta, \ c\in\R. $$ The projection onto the base point is denoted by $$ \tau : W_{\xi_0} \to S^3. $$

Denote by $\M^{\leq T,(p,q)}_{\jhat}(P,P')$ the space of equivalence classes of $\jhat$-holomorphic finite-energy maps $\util : \R\times \R/\Z \simeq S^2\setminus \{0,\infty\} \to W_{\xi_0}$ with one positive and one negative puncture, asymptotic at the positive puncture to $P \in \mathcal{P}^{\leq T, (p,q)}(\lambda)$ and at the negative puncture to $P' \in \mathcal{P}^{\leq T, (p,q)}(\lambda)$, with the additional property that the image of $\util$ does not intersect $\tau^{-1}(K_0)$, modulo holomorphic reparametrizations. Here we identify $\R\times \R/\Z\simeq S^2\setminus\{0,\infty\}$ via $(s,t)\simeq \est$, equip $\R\times \R/\Z$ with its standard complex structure, the positive puncture is $\infty$ and the negative puncture is $0$. Any two such cylinders $\util,\vtil$ are equivalent if there exists $(\Delta s,\Delta t) \in \R\times \R/\Z$ such that $\vtil(s,t) = \util(s+\Delta s,t+\Delta t)$. Note that we do not quotient out by the $\R$-action $\{g_c\}$ on the target manifold. Strictly speaking $\M^{\leq T,(p,q)}_{\jhat}(P,P')$ is not a set of maps, but we may write $\util \in \M^{\leq T,(p,q)}_{\jhat}(P,P')$ when a map $\util$ represents an element of this moduli space.

Since $(p,q)$ is a relatively prime pair of integers, every orbit in $\P^{\leq T,(p,q)}(\lambda)$ is simply covered and, consequently, results of~\cite{props1} imply that curves representing elements of $\M^{\leq T,(p,q)}_{\jhat}(P,P')$ must be somewhere injective. Consider the set $\J_{\rm reg}(\lambda) \subset \J(\lambda)$ of almost complex structures satisfying the following property: if $\jhat \in \J_{\rm reg}(\lambda)$ then all cylinders (representing elements) in $\M^{\leq T,(p,q)}_{\jhat}(P,P')$ are regular in the sense of Fredholm theory for all $P,P' \in \P^{\leq T,(p,q)}(\lambda)$. This is standard and means that, in the appropriate functional analytic set-up, the linearized Cauchy-Riemann operator at a cylinder representing an element of $\M^{\leq T,(p,q)}_{\jhat}(P,P')$ is a surjective Fredholm map whenever $P,P' \in \P^{\leq T,(p,q)}(\lambda)$; see~\cite{wendl_transv} for a nice description of the analytic set-up. The set $\J_{\rm reg}(\lambda)$ depends on $T$ and $(p,q)$, but we do not make this explicit in the notation. Results of~\cite{drag} show that $\J_{\rm reg}(\lambda)$ is a residual subset of $\J(\lambda)$. Consequently, the spaces $\M^{\leq T,(p,q)}_{\jhat}(P,P')$, for all $P,P' \in \P^{\leq T, (p,q)}(\lambda)$, have the structure of a finite dimensional manifold when $\widehat J \in \J_{\rm reg}(\lambda)$, with dimension ${\rm Ind}(\util) = \mu_{CZ}(P) - \mu_{CZ}(P')$ whenever this quantity is $\geq 0$. When this quantity is $>0$ then the $\R$-action $\{g_c\}$ on the target induces an $\R$-action on $\M^{\leq T,(p,q)}_{\jhat}(P,P')$ which is smooth and free. If ${\rm Ind}(\util) = 0$ and $\M^{\leq T,(p,q)}_{\jhat}(P,P') \neq \emptyset$ then $P=P'$, $\util$ is a trivial cylinder and the $\R$-action on the moduli space is trivial.

\begin{theorem}\label{compactness_index_0_differential}
If $\jhat \in \J_{\rm reg}(\lambda)$ and $P,P' \in \P^{\leq T,(p,q)}(\lambda)$ satisfy $\mu_{CZ}(P') = \mu_{CZ}(P)-1$ then the space $\M^{\leq T,(p,q)}_{\jhat}(P,P')/\R$ is finite.
\end{theorem}

See Section~\ref{proofs_thms_chain_complex} in the appendix for a proof. Therefore, it makes sense to define the following degree $-1$ map:
\begin{equation}\label{differential_defn}
  \begin{aligned}
    \partial(\lambda,J&)_*:\ C_*^{\leq T,(p,q)}(\lambda) \rightarrow C_{*-1}^{\leq T,(p,q)}(\lambda) \\
    & q_P \mapsto \sum_{\substack{P' \in \mathcal{P}^{\leq T,(p,q)}(\lambda) \\ |P'| = *-1}} \#_2  \left( \M^{\leq T,(p,q)}_{\jhat}(P,P')/\R \right) q_{P'}
  \end{aligned}
\end{equation}
on generators, where $\#_2$ denotes the number of elements in a set $(\mathrm{mod} \mbox{ } 2)$ as an element of $\mathbb{F}_2$.

\begin{theorem}\label{d2=0}
If $\jhat \in \J_{\rm reg}(\lambda)$ then $\partial_{k-1} \circ \partial_k = 0$, $\forall k\in \Z$.
\end{theorem}

The proof is also deferred to the appendix, see Section~\ref{proofs_thms_chain_complex}. As a consequence $$ (\oplus_{*\in\Z} C_*^{\leq T,(p,q)}(\lambda),\partial(\lambda,J)_*) $$ is a chain complex. Its homology is denoted by
\begin{equation}\label{homology_notation}
HC^{\leq T,(p,q)}_*(\lambda,J).
\end{equation}

\subsection{Chain maps}\label{chain_map_paragraph}

Let $T>0$ and $h_+,h_- \in \F$ be such that $\lambda_\pm = h_\pm\lambda_0$ satisfy conditions (a), (b) and (c) described in Section~\ref{chain_complex_subsection}. Let also $(p,q)$ be a pair of relatively prime integers. In this case we may choose $J_\pm \in \J_+(\xi_0)$ such that $\jhat_\pm \in \J_{\rm reg}(\lambda_\pm)$ and the chain complexes $(C^{\leq T,(p,q)}_*(\lambda_+),\partial_{(\lambda_+,J_+)})$ and $(C^{\leq T,(p,q)}_*(\lambda_-),\partial_{(\lambda_-,J_-)})$ are well-defined. There is a natural way to define a chain map between these chain complexes as long as $h_+ > h_-$ pointwise and the associated rotation numbers satisfy
\begin{equation}\label{crucial_chain_map}
  \begin{array}{cccc}
    \theta_i(h_{\pm}) \not\in \Q, & \theta_0(h_+) \geq \theta_0(h_-) & \text{and} & \theta_1(h_+) \geq \theta_1(h_-).
  \end{array}
\end{equation}

As is explained in Section~\ref{restricted_class_alm_cpx_str}, the space $\J(\jhat_-,\jhat_+:K_0)$ is non-empty and contractible. For any $\bar J \in \J(\jhat_-,\jhat_+:K_0)$, $P \in \P^{\leq T,(p,q)}(\lambda_+)$ and $P' \in \P^{\leq T,(p,q)}(\lambda_-)$ we consider the space $\M^{\leq T,(p,q)}_{\bar J}(P,P')$ of equivalence classes of finite-energy $\bar J$-holomorphic cylinders with image in $W_{\xi_0} \setminus \tau^{-1}(K_0)$ which are asymptotic to $P$ at the positive puncture and to $P'$ at the negative puncture, modulo holomorphic reparametrizations.

Let $\J_{\rm reg}(\jhat_-,\jhat_+:K_0) \subset \J(\jhat_-,\jhat_+:K_0)$ be the set of $\bar J$ for which the following holds: every element of $\M_{\bar J}^{\leq T,(p,q)}(P,P')$ is regular in the sense of Fredholm theory whenever $P \in \P^{\leq T,(p,q)}(\lambda_+)$ and $P' \in \P^{\leq T,(p,q)}(\lambda_-)$, see~\cite{wendl_transv}. As before, this set of regular almost complex structures depends on $(p,q)$ and $T$, but we do not make this explicit in the notation. Standard arguments~\cite{MR2200047,drag,mcdsal,momin} show that the set $\J_{\rm reg}(\jhat_-,\jhat_+:K_0)$ contains a residual subset of $\J(\jhat_-,\jhat_+:K_0)$. It is crucial here that $P,P'$ are simply covered, which is the case since the pair $(p,q)$ is relatively prime. Then $\M_{\bar J}^{\leq T,(p,q)}(P,P')$ becomes a smooth manifold of dimension $\mu_{CZ}(P) - \mu_{CZ}(P')$ since there are no orbifold points (every element is represented by a somewhere injective map).

\begin{theorem}\label{compactness_index_0_chain_map}
If $\bar J \in \J_{\rm reg}(\jhat_-,\jhat_+:K_0)$ and $P \in \P^{\leq T,(p,q)}(\lambda_+)$, $P' \in \P^{\leq T,(p,q)}(\lambda_-)$ satisfy $\mu_{CZ}(P') = \mu_{CZ}(P)$ then the space $\M^{\leq T,(p,q)}_{\bar J}(P,P')$ is finite.
\end{theorem}

The proof is found in Section~\ref{proofs_thms_chain_homotopy} in the appendix. After selecting $\jhat_\pm \in \J_{\rm reg}(\lambda_\pm)$, any $\bar{J} \in \J_{\mathrm{reg}}(\jhat_-,\jhat_+:K_0)$ can be used to define a chain map given by
\begin{equation}\label{chain_map_phi_j_bar}
\begin{aligned}
  \Phi(\bar{J})_* & : C_{*}^{\leq T, (p,q)}(\lambda_+) \rightarrow C_*^{\leq T, (p,q)}(\lambda_-) \\
  & q_P \mapsto \sum_{\substack{P' \in \mathcal{P}^{\leq T,(p,q)}(\lambda_-) \\ |P'| = *}} \left( \#_2 \M^{\leq T,(p,q)}_{\bar{J}}(P,P') \right) q_{P'}
\end{aligned}
\end{equation}
on generators, where again $\#_2$ denotes the number of elements in a set $(\mathrm{mod} \mbox{ } 2)$ as an element of $\mathbb{F}_2$. The number of elements in each such $\M^{\leq T,(p,q)}_{\bar J}(P,P')$ is finite by Theorem~\ref{compactness_index_0_chain_map} so that this map is well-defined. That $\Phi(\bar J)_*$ is a chain map is the content of the next statement. The proof is postponed to the appendix, see Section~\ref{proofs_thms_chain_homotopy}.

\begin{theorem}\label{chain_homotopy_thm}
$\Phi(\bar{J})_{*-1} \circ \partial(\lambda_+,J_+)_* - \partial(\lambda_-,J_-)_* \circ \Phi(\bar{J})_* = 0$
\end{theorem}

\subsection{Comparing chain maps}

We consider $h_\pm$ exactly as in Section~\ref{chain_map_paragraph}, together with regular choices $\jhat_\pm \in \J_{\rm reg}(\lambda_\pm)$ and regular choices $\bar J_0,\bar J_1 \in \J_{\rm reg}(\jhat_-,\jhat_+:K_0)$, so that we have chain maps $\Phi(\bar J_0)_*,\Phi(\bar J_1)_*$. Here we denoted $\lambda_\pm = h_\pm\lambda_0$. We would like to show that they induce the same map at the level of homology.

To this end, we consider the space $\widetilde \J(\bar J_0,\bar J_1 : K_0)$ of smooth homotopies
\[
  t\in [0,1] \mapsto \bar J_t \in \J(\jhat_-,\jhat_+:K_0)
\]
from $\bar J_0$ to $\bar J_1$. For orbits $P \in \P^{\leq T,(p,q)}(\lambda_+)$ and $P' \in \P^{\leq T,(p,q)}(\lambda_-)$ we set
\begin{equation}\label{}
  \M^{\leq T,(p,q)}_{\{\bar J_t\}}(P,P') = \{ (t,[\util]) \mid t\in[0,1] \text{ and } [\util] \in \M^{\leq T,(p,q)}_{\bar J_t}(P,P') \}
\end{equation}
where $\M^{\leq T,(p,q)}_{\bar J_t}(P,P')$ is as defined in Section~\ref{chain_map_paragraph}. Using standard arguments, in a similar way as it is done in~\cite[Section 3.2]{mcdsal}, one finds a residual set $$ \widetilde \J_{\rm reg}(\bar J_0,\bar J_1 : K_0) \subset \widetilde \J(\bar J_0,\bar J_1 : K_0) $$ such that if $\{\bar J_t\} \in \widetilde \J_{\rm reg}(\bar J_0,\bar J_1 : K_0)$ then $\M^{\leq T,(p,q)}_{\{\bar J_t\}}(P,P')$ is a smooth manifold of dimension $\mu_{CZ}(P) - \mu_{CZ}(P') + 1$, for every pair of orbits $P,P'$ as above. It is crucial here that for every $t$ all cylinders in $\M^{\leq T,(p,q)}_{\bar J_t}(P,P')$ are necessarily somewhere injective, which is true since the orbits $P \in \P^{\leq T,(p,q)}(\lambda_+)$ and $P' \in \P^{\leq T,(p,q)}(\lambda_-)$ are simply covered. Thus there are no orbifold points. As before, we may achieve regularity by a perturbation keeping the tangent space of $\tau^{-1}(K_0)$ complex invariant along the path of almost complex structures.

\begin{theorem}\label{degree_1_map_compactness}
Whenever $P \in \P^{\leq T,(p,q)}(\lambda_+)$ and $P' \in \P^{\leq T,(p,q)}(\lambda_-)$ are such that $\mu_{CZ}(P) = \mu_{CZ}(P')-1$ then the space $\M^{\leq T,(p,q)}_{\{\bar J_t\}}(P,P')$ is finite. Moreover, if $(t,[\util]) \in \M^{\leq T,(p,q)}_{\{\bar J_t\}}(P,P')$ then $t\neq 0,1$.
\end{theorem}

In the above statement we assume that $\{\bar J_t\} \in \widetilde \J_{\rm reg}(\bar J_0,\bar J_1 : K_0)$ and that $\bar J_0,\bar J_1 \in \J_{\rm reg}(\jhat_-,\jhat_+:K_0)$. See Section~\ref{proofs_thms_degree_1_map} in the appendix for a proof. Following a usual procedure, we define a degree $+1$ map
\begin{equation}\label{chain_maps_phis_bar}
\begin{aligned}
T(\{\bar J_t\})_* & : C_*^{\leq T, (p,q)}(\lambda_+) \rightarrow C_{*+1}^{\leq T, (p,q)}(\lambda_-) \\
& q_{P} \mapsto \sum_{\substack{P' \in \mathcal{P}^{\leq T,(p,q)}(\lambda_-) \\ |P'| = *+1}} \left( \#_2 \M^{\leq T,(p,q)}_{\{\bar J_t\}}(P,P') \right) q_{P'}
\end{aligned}
\end{equation}
The sum above is finite by Theorem~\ref{degree_1_map_compactness}.

\begin{theorem}\label{comparing_chain_maps_thm}
$$ \Phi(\bar J_1)_* - \Phi(\bar J_0)_* = T(\{\bar J_t\})_{*-1} \circ \partial(\lambda_+,J_+)_* - \partial(\lambda_-,J_-)_{*+1} \circ T(\{\bar J_t\})_*. $$
\end{theorem}

The proof is deferred to Section~\ref{proofs_thms_degree_1_map} in the appendix.

\section{Computing contact homology}\label{section_computing}

Our goal here is to compute contact homology in the complement of the Hopf link for special classes of contact forms. The main results in this section are Proposition~\ref{prop_morse_bott_approx} and Proposition~\ref{comp_j_iota}. We freely use the notation established in Section~\ref{section_contact_homology}.

\subsection{A class of model contact forms}\label{example-2}

Let $\theta_0,\theta_1 \in \R\setminus\Q$, and let $\gamma(t) = (x(t), y(t))$ for $t \in [0,1]$ be a smooth embedded curve in the first quadrant of $\R^2$ satisfying the following properties:
\begin{itemize}
\item $x(0) > 0, y(0) = 0$, and $y'(0) > 0$;
\item $x(1) = 0, y(1) > 0$, and $x'(1) < 0$;
\item $x  y' - x'  y > 0$ for all $t \in [0,1]$, equivalently, $\gamma$ and $\gamma'$ are never co-linear;
\item $x' y'' - x'' y' \neq 0$ for all $t \in [0,1]$;
\item $(y'(0),-x'(0)) \in \R^+(1,\theta_1)$;
\item $(y'(1),-x'(1)) \in \R^+(\theta_0,1)$.
\end{itemize}

It is always possible to find such a curve $\gamma$ if $(\theta_0,1) \not\in \R^+(1,\theta_1)$. We can define a star-shaped hypersurface $S_\gamma$ in $\C^2 \simeq \R^4$ associated to $\gamma$ by
\[
S_{\gamma} = \{(x_0,y_0,x_1,y_1) \in \R^4 \mid (r_0^2,r_1^2) \in \gamma([0,1]) \}
\]
where $(x_k,y_k) \simeq x_k + iy_k = r_ke^{i\phi_k}$ are polar coordinates ($k=0,1$). To see that $S_\gamma$ is a smooth hypersurface, consider a smooth function $F:\R^2\to\R$ such that $0\in\R$ is a regular value of $F$ and $\gamma([0,1]) = F^{-1}(0) \cap \{x\geq 0,y\geq 0\}$. Then $S_\gamma = H^{-1}(0)$, where $H:\R^4\to\R$ is defined by $H(x_0,y_0,x_1,y_1) = F(r_0^2,r_1^2)$. The third condition above guarantees that $S_\gamma$ is star-shaped and $\lambda_0$ restricts to a contact form on $S_\gamma$ inducing the contact structure $\xi_0=\ker\lambda_0|_{S_\gamma}$.

We parametrize leaves of the characteristic foliation of $S_\gamma$ as trajectories of the Reeb vector field $X_0$ determined by $\lambda_0$ on $S_\gamma$. Assuming that $F<0$ on the bounded component of $\{x\geq0,y\geq0\}\setminus\gamma([0,1])$ we obtain
\begin{align*}
X_0 = a_0(r_0^2,r_1^2) \partial_{\phi_0} + a_1(r_0^2,r_1^2) \partial_{\phi_1} \in \R^+ \left( y'(t) \partial_{\phi_0} - x'(t) \partial_{\phi_1} \right)
\end{align*}
where $t$ is uniquely determined by $(r_0^2,r_1^2) = \gamma(t)$, since $(y',-x')$ points to the unbounded component of $\{x\geq0,y\geq0\}\setminus\gamma([0,1])$. The sets $\bar L_0 = S_\gamma \cap (0\times\C)$, $\bar L_1 = S_\gamma \cap (\C\times0)$ are closed orbits and their iterates have Conley-Zehnder indices
\[
\mu_{CZ}(\bar L_0^k) = 2 \lfloor k(1 + \theta_0) \rfloor + 1 \qquad \mu_{CZ}(\bar L_1^k) = 2 \lfloor k(1 + \theta_1) \rfloor + 1.
\]
To see this note that the period of $\bar L_0$ as a prime periodic $X_0$-orbit is $2\pi/a_1(0,r_1^2)$. The transverse linearized flow of $X_0$ restricted to $\xi_0|_{\bar L_0}$ rotates $$ \frac{2\pi}{a_1(0,r_1^2)}a_0(0,r_1^2) = 2\pi\frac{y'(1)}{-x'(1)} = 2\pi\theta_0 $$ after the first period, measured with respect to the frame $\{\partial_{x_0},\partial_{y_0}\}$ of $\xi_0|_{\bar L_0}$. Thus it rotates exactly $2\pi(1+\theta_0)$ after the first period with respect to a global frame of $\xi_0$. This last claim follows from the fact that $\bar L_0$ has self-linking number $-1$, and can also be alternatively verified by explicitly writing down a global frame and comparing it with $\{\partial_{x_0},\partial_{y_0}\}$. A similar reasoning applies to $\bar L_1$.

To obtain the remaining orbits notice that each point $(r_0^2,r_1^2) \in \gamma((0,1))$ determines an invariant torus foliated by Reeb trajectories. These trajectories are all closed or all non-periodic when $a_0$ and $a_1$ are dependent or independent over $\Q$, respectively. The former happens precisely when the line determined by the corresponding normal $(y',-x')$ goes through points in the integer lattice.

Consider $$ f_{\theta_0,\theta_1}:S^3 \rightarrow (0,\infty) $$ determined by $\sqrt{f_{\theta_0,\theta_1}(z)} z \in S_{\gamma}, \forall z\in S^3$. The diffeomorphism $\Psi:S^3\to S_\gamma$, $\Psi(z) = \sqrt{f_{\theta_0,\theta_1}(z)} z$ satisfies $\Psi^*\lambda_0 = f_{\theta_0,\theta_1}\lambda_0$. Moreover, the components $L_0 = S^3 \cap (0\times \C)$ and $L_1 = S^3 \cap (\C\times 0)$ of the standard Hopf link are mapped onto $\bar L_0$, $\bar L_1$ respectively, which implies that $f_{\theta_0,\theta_1} \in \mathcal F$; see Remark~\ref{std_representation_hopf_link} for more details. Summarizing we have
\begin{itemize}
\item If $(p,q)$ is a relatively prime pair of integers satisfying~\eqref{non_resonance} then there is a unique torus foliated by prime closed orbits of the Reeb dynamics associated to $f_{\theta_0,\theta_1}\lambda_0|_{S^3}$, each closed orbit representing the homotopy class $(p,q) \in \Z\times \Z \simeq \pi_1(S^3\setminus (L_0\cup L_1),\rm{pt})$.
\end{itemize}
Uniqueness comes from strict concavity/convexity of $\gamma$, which is ensured by the fourth condition on $\gamma$.

\subsection{Perturbation of $f_{\theta_0,\theta_1}\lambda_0$ and computation of contact homology}

The forms $f_{\theta_0,\theta_1}\lambda_0$ defined above satisfy a weak non-degeneracy hypothesis.

\begin{definition}[Hofer, Wysocki and Zehnder~\cite{props4}, Bourgeois~\cite{Bourgeois_thesis}]\label{morse-bott_def}
Suppose that $\lambda$ is a contact form on a manifold $M$.  We say $\lambda$ is Morse-Bott non-degenerate if
\begin{enumerate}
  \item the action spectrum is discrete,
  \item for any given action value $T$, the set of points lying on closed orbits of action $T$ is a smooth embedded submanifold $N_T$ of $M$,
  \item the rank of $d\lambda|_{TN_T}$ is locally constant along $N_T$, and
  \item if $\phi_t(p)$ denotes the Reeb flow then $\ker(d\phi_T(x) - I) = T_x N_T$, $\forall x \in N_T$.
\end{enumerate}
\end{definition}

\begin{proposition}\label{prop_morse_bott_approx}
The forms $f_{\theta_0,\theta_1}\lambda_0$ are Morse-Bott non-degenerate contact forms on $S^3$ when $\theta_0,\theta_1 \not\in \Q$. Suppose $(p,q)$ is a relatively prime pair of integers satisfying~\eqref{non_resonance}, and denote by $N_{(p,q)}$ the unique 2-torus invariant under the Reeb flow of $f_{\theta_0,\theta_1}\lambda_0$ foliated by prime closed Reeb orbits in the homotopy class $(p,q)$. Let $T_{(p,q)}>0$ be their common prime period and let $S_{(p,q)}$ be the circle obtained by the quotient of $N_{(p,q)}$ by the Reeb flow. For any $S > T_{(p,q)}$ and $\delta>0$ we may find a contact form $f_S \lambda_0$ arbitrarily $C^\infty$-close to $f_{\theta_0,\theta_1}\lambda_0$ with the following properties:
\begin{itemize}
  \item there are no closed $(f_S\lambda_0)$-Reeb orbits of action at most $S$ in $S^3\setminus K_0$ which are contractible in $S^3 \backslash K_0$,
  \item the set $\mathcal{P}^{\leq S,(p,q)}(f_S\lambda_0)$ is in 1-1 correspondence with the critical points of a chosen perfect Morse function on $S_{(p,q)}$ and their Conley-Zehnder indices differ by $1$,
  \item $f_S \lambda_0$ is non-degenerate up to action $S$,
  \item the actions of the closed Reeb orbits in $\mathcal{P}^{\leq S,(p,q)}(f_S\lambda_0)$ lie in the interval $(T_{(p,q)}-\delta, T_{(p,q)}+\delta)$,
  \item $f_S$ agrees with $f_{\theta_0,\theta_1}$ near $K_0$.
\end{itemize}
Moreover, there exists a suitable $J_S \in \J_+(\xi_0)$ such that $\jhat_S$ is regular with respect to homotopy class $(p,q)$ and action bound $S$, as explained in Section~\ref{chain_complex_subsection}, and
\[
HC_*^{\leq S,(p,q)}(f_S\lambda_0,J_S) \cong H_{*-s}(S^1;\mathbb{F}_2),
\]
for some $s\in\Z$.
\end{proposition}

The remaining paragraphs in this section are devoted to proving Proposition~\ref{prop_morse_bott_approx}. The idea of perturbing a Morse-Bott non-degenerate contact form is originally due to Bourgeois~\cite{Bourgeois_thesis}. In the proof below we provide all the details for this perturbation in our particular case. As the reader will notice, we resolve the analytical difficulties that arise in the computation of contact homology by a new and independent argument that relies on the intersection theory for punctured pseudo-holomorphic curves developed by Siefring~\cite{Sie}.

\subsubsection{Verifying the Morse-Bott property and perturbing $f_{\theta_0,\theta_1}\lambda_0$}\label{paragraph_perturb}

We work directly on the star-shaped hypersurface $S_\gamma = \{ F(r_0^2,r_1^2) = 0 \}$, where $F : \R^2 \to \R$ is the smooth function associated to the special curve $\gamma$ described in Section~\ref{example-2}. The Hopf link is represented by $\bar L_0 = S_\gamma \cap (0 \times \C)$ and $\bar L_1 = S_\gamma \cap (\C\times 0)$, and we write $\bar K_0 = \bar L_0 \cup \bar L_1$. First we need to show that $\lambda_0|_{S_\gamma}$ is Morse-Bott non-degenerate on $S_\gamma$.

Let us assume $\gamma$ is strictly convex, the case when $\gamma$ is strictly concave is analogous. It is convenient to introduce the function
\[
\vartheta(r_0,r_1) := \text{argument of the vector } (D_1F(r_0^2,r_1^2),D_2F(r_0^2,r_1^2)).
\]
This is a well-defined smooth function on $S_\gamma\setminus \bar K_0$ assuming all values in an interval $(a,b) \subset (-\pi/2,\pi)$ in view of the defining properties of the curve $\gamma$. Strict convexity of $\gamma$ shows that $\vartheta$ is a global unambiguously defined smooth parameter on the embedded arc $\gamma \setminus \{\text{end points}\}$, and that we can choose $\vartheta$ to be strictly increasing when $\gamma$ is prescribed counter-clockwise. One finds that $d\vartheta$, $d\phi_0$ and $d\phi_1$ are pointwise linearly independent. Thus we get coordinates $(\vartheta,\phi_0,\phi_1)$ which define a diffeomorphism $S_\gamma \setminus \bar K_0 \simeq (a,b) \times \R/2\pi\Z\times \R/2\pi\Z$ and we can write
\[
\lambda_0|_{S_\gamma} = h_0(\vartheta) d\phi_0 + h_1(\vartheta) d\phi_1
\]
for suitable functions uniquely determined by $h_0(\vartheta) = \frac{1}{2}r_0^2$, $h_1(\vartheta) = \frac{1}{2}r_1^2$. The argument of the vector $(h_0,h_1)$ also varies monotonically with $\vartheta$, and the denominator in the following expression for Reeb vector field associated to $\lambda_0|_{S_\gamma}$
\[
X_0 = \frac{h_1' \partial_{\phi_0} - h_0' \partial_{\phi_1}}{h_0h_1' - h_1h_0'}
\]
is strictly positive. In the following we write $\lambda_0$ instead of $\lambda_0|_{S_\gamma}$ for simplicity and denote by $\varphi_t$ the flow of $X_0$.

Note that the action spectrum is discrete. Indeed, the closed orbits are either iterates of the components of $\bar K_0$ or lie in invariant tori determined by values of $\vartheta \in (a,b)$ such that $h_1'(\vartheta)$ and $-h_0'(\vartheta)$ are dependent over $\Q$. The periods of orbits in such a torus are $\{kT(\vartheta) \mid k=1,2,\dots \}$ where $$ T(\vartheta) = \min \left\{ t>0 \text{ such that } \frac{t(h_1'(\vartheta),-h_0'(\vartheta))}{h_0(\vartheta)h_1'(\vartheta) - h_1(\vartheta)h_0'(\vartheta)} \in 2\pi\Z \times 2\pi\Z \right\}. $$ Given any $M>0$ there are only finitely many values $\vartheta_1,\dots,\vartheta_N$ such that $T(\vartheta_i)$ is defined and satisfies $T(\vartheta_i)\leq M$. Moreover, for each $i$ there are only finitely many positive integers $k_i^1,\dots,k_i^{J_i}$ satisfying $k_i^jT(\vartheta_i)\leq M$. Hence the intersection of the action spectrum with $[0,M]$ is finite, for all $M>0$. We proved that $\lambda_0$ satisfies condition (1) in Definition~\ref{morse-bott_def}.

For each value $T$ in the action spectrum, the set $N_T$ of fixed points of $\varphi_T$ is a submanifold of $S_\gamma$ consisting of a finite collection of tori together with at most two circles (corresponding to $\bar L_0$ or $\bar L_1$). The Reeb flow induces a smooth action of the circle $\R/T\Z$ on $N_T$. In our simple situation the quotient $S_T$ of $N_T$ by this action is a finite collection of circles together with at most two additional isolated points. Another particular feature of our model forms is that if $k>1$ then $N_{T}$ is a collection of components of $N_{kT}$ and, consequently, $S_{T}$ is a collection of components of $S_{kT}$.

In view of strict convexity/concavity of $\gamma$ there is one, and only one, invariant torus foliated by periodic orbits representing the homotopy class $(i,j) \in \Z\times \Z \simeq \pi_1(S_\gamma\setminus \bar K_0,{\rm pt})$ for each relatively prime pair of integers $(i,j)$ satisfying $(\theta_0,1) < (i,j) < (1,\theta_1)$ or $(1,\theta_1) < (i,j) < (\theta_0,1)$. All such tori are singled out by fixing a certain value of $\vartheta$. The torus corresponding to $(p,q)$ will be denoted $\mathbb T_{(p,q)}$ and the period of the prime closed Reeb orbits there by $T_{(p,q)}$, so that $\mathbb T_{(p,q)}$ is a component of $N_{T_{(p,q)}}$.

We fix $S>T_{(p,q)}+1$ not in the action spectrum, and choose inductively for $T<S$ in the action spectrum a smooth function $\bar g_T : N_T \to \R$ invariant by the Reeb flow such that $\bar g_{kT}|_{N_T} \equiv \bar g_T$ whenever $kT<S$ and $k\geq1$, inducing perfect Morse functions $g_T$ on the orbit spaces $S_T$. These induce a smooth function on $\bar g_S: \cup_{T<S}N_T \to \R$ which is constant along Reeb trajectories. If $S$ is larger than the actions of $\bar L_0$ and $\bar L_1$ then $\bar K_0 = \bar L_0 \cup \bar L_1$ is a pair of components of $\cup_{T<S}N_T$ and we may assume without loss of generality that $\bar g_S$ vanishes on $\bar K_0$. To be more concrete, consider an invariant torus $\mathbb T = \{\vartheta = \vartheta^*\}$ foliated by closed orbits of period $T<S$, for some value $\vartheta = \vartheta_* \in (a,b)$. Assume that $\bar g_{T'}$ has been chosen for all values $0<T'<T$ in the action spectrum, satisfying the above compatibility conditions. If $\mathbb T \subset N_{T'}$ for some action value $T'<T$ then the function $\bar g_T$ is already there, so we assume $\mathbb T \subset N_T\setminus (\cup_{T'<T} N_{T'})$. Consider $(x,y)$ defined on the universal covering by
\begin{equation}\label{coord_xy}
\begin{pmatrix} x \\ y \end{pmatrix} = \begin{pmatrix} h_0'(\vartheta^*) & h_1'(\vartheta^*) \\ h_0(\vartheta^*) & h_1(\vartheta^*) \end{pmatrix} \begin{pmatrix} \phi_0 \\ \phi_1 \end{pmatrix}.
\end{equation}
Denoting $d = h_1h_0' - h_0h_1'$ and
\[
\begin{array}{ccc} \Delta_1(\vartheta) = \frac{h_0(\vartheta)h_1(\vartheta^*)-h_1(\vartheta)h_0(\vartheta^*)}{d(\vartheta^*)} & & \Delta_2(\vartheta) = \frac{h_1(\vartheta)h_0'(\vartheta^*)-h_0(\vartheta)h_1'(\vartheta^*)}{d(\vartheta^*)} \end{array}
\]
we obtain
\begin{equation}\label{mat_ident}
\begin{pmatrix} \Delta_1'(\vartheta^*) & \Delta_1(\vartheta^*) \\ \Delta_2'(\vartheta^*) & \Delta_2(\vartheta^*) \end{pmatrix} = \begin{pmatrix} 1 & 0 \\ 0 & 1 \end{pmatrix}.
\end{equation}
A brief calculation gives $$ \begin{array}{ccc} \lambda_0 = \Delta_1 dx + \Delta_2 dy & \text{and} & X_0 = \frac{1}{\Delta}(-\Delta_2'\partial_x + \Delta_1'\partial_y), \end{array} $$ where $\Delta := \Delta_1'\Delta_2 - \Delta_1\Delta_2'$. In particular $X_0|_{\mathbb T} = \partial_y$, so that invariant functions on $\mathbb T$ depend only on $x$. The variable $x$ is periodic and, denoting the period of $x$ by $L$, we set $\bar g_T|_{\mathbb T} = \cos (2\pi x/L)$. Repeating this construction for the (finitely many) tori foliated by Reeb orbits of period $T$ not contained in $\cup_{T'<T}N_{T'}$, and setting $\bar g_T = 0$ on $\bar K_0 \cap N_T$, we obtain the desired function on $\cup_{T'\leq T}N_{T'}$. Doing this for all action values $T<S$ we get the desired function on $\cup_{T<S}N_T$.

We will use the special coordinates $x,y$ defined in~\eqref{coord_xy} adapted to some invariant torus $\mathbb T = \{\vartheta = \vartheta^*\} \subset N_T$, for some $T<S$, to study the linearized Reeb flow $d\varphi_t$ on $\mathbb T$. If $V = A \partial_\vartheta + B \partial_x + C \partial_y \in T(S_{\gamma} \setminus \bar K_0)|_{\mathbb T}$ then $V(t) = d\varphi_t \cdot V = A(t) \partial_\vartheta + B(t) \partial_x + C(t) \partial_y$ satisfies
\[
\begin{array}{ccc} \dot A(t) = 0, & \dot B(t) = -\Delta_2''A, & \dot C(t) = 0. \end{array}
\]
We have $\Delta_2''(\vartheta^*) = (h_1''(\vartheta^*)h_0'(\vartheta^*) - h_0''(\vartheta^*)h_1'(\vartheta^*))/d(\vartheta^*) \neq0$ by strict convexity/concavity of $\gamma$, and this implies that $\ker (d\varphi_T|_{\mathbb T} - I) = T\mathbb T$, that is, condition (4) in Definition~\ref{morse-bott_def} holds at $\mathbb T$. Conditions (2) and (3) are easy to check using the above constructions. This concludes the proof that $\lambda_0|_{S_\gamma}$ is Morse-Bott non-degenerate.

We will extend the function $\bar g_T$ to a small neighborhood of $\mathbb T$. More precisely, if $I$ is a small neighborhood of $\vartheta^*$ and $\beta : I \to [0,1]$ is a smooth compactly supported function such that $\beta\equiv1$ near $\vartheta^*$ then we define
\begin{equation}\label{function_g}
\bar g_T(\vartheta,x) = \beta(\vartheta)\cos(2\pi x/L).
\end{equation}
Let $\vartheta_1<\dots<\vartheta_n$ be so that
\begin{equation}\label{tori}
(\cup_{T<S}N_T) \setminus \bar K_0 = \bigcup_{i=1}^n \mathbb T_i, \ \text{ with } \mathbb T_i := \{\vartheta = \vartheta_i\}.
\end{equation}
We can repeat the above constructions near each $\mathbb T_i$ to get an extension of $\bar g_S$ to a small neighborhood of $\cup_i \mathbb T_i$. Then $\bar g_S$ can be further extended to $S_\gamma$ by zero outside of this neighborhood of $\cup_i\mathbb T_i$.

We consider, as in~\cite{sftcomp}, the 1-form
\begin{equation}\label{lambda_epsilon}
  \lambda_\epsilon := f_\epsilon \lambda_0, \ \text{ where } \ f_\epsilon = 1 + \epsilon \bar g_S,
\end{equation}
for small values $\epsilon>0$. Note that $f_\epsilon-1$ is $C^\infty$-small and supported near $\cup_i \mathbb T_i$, where the $\mathbb T_i$ are defined in~\eqref{tori}. The Reeb vector field of $\lambda_\epsilon$ will be denoted $X_\epsilon$.

Fixing a torus $\mathbb T = \{ \vartheta = \vartheta^* \}$ among the $\mathbb T_i$, we will analyze the flow of $X_\epsilon$ near $\mathbb T$. Any prime closed Reeb orbit $P$ in $\mathbb T$ has period $T<S$ and satisfies $\link(P,\bar L_0)=r$, $\link(P,\bar L_1)=s$ for certain $r,s\in\Z$. Recall that according to our constructions, after introducing coordinates $x,y$ as in~\eqref{coord_xy}, $\bar g_S$ takes the form $\bar g_S = \beta(\vartheta)\cos(2\pi x/L)$ near $\mathbb T$, where $\beta:\R\to[0,1]$ has support on a small interval $I$ centered at $\vartheta^*$ and equals to $1$ near $\vartheta^*$. Here $L$ is the period of $x$. The vector field $X_\epsilon = X^\vartheta_\epsilon \partial_\vartheta + X^x_\epsilon \partial_x + X^y_\epsilon \partial_y$ is determined by the functions
\begin{equation}\label{comp_Reeb_epsilon}
\begin{array}{ccc}
X^\vartheta_\epsilon = \partial_x\left( -\frac{1}{f_\epsilon} \right) \frac{\Delta_2}{\Delta},
& X^x_\epsilon = -\frac{\partial_\vartheta(f_\epsilon\Delta_2)}{f_\epsilon^2\Delta},
& X^y_\epsilon = \frac{\partial_\vartheta(f_\epsilon\Delta_1)}{f_\epsilon^2\Delta}.
\end{array}
\end{equation}

We wish to understand the dynamics of $X_\epsilon$ on the neighborhood $\mathcal O = \{\vartheta \in I\}$ of $\mathbb T$. Note that the vector field $X_\epsilon$ does not depend on $y$, $f_\epsilon^2\Delta X^\vartheta_\epsilon = (\partial_xf_\epsilon) \Delta_2$ and $f_\epsilon^2\Delta X^x_\epsilon = -\partial_\vartheta(f_\epsilon\Delta_2)$. Thus, periodic orbits of $X_\epsilon$ on $\mathcal O$ project to periodic orbits and rest points of the vector field $Z = Z^\vartheta \partial_\vartheta + Z^x \partial_x = (\partial_xf_\epsilon) \Delta_2 \partial_\vartheta - \partial_\vartheta(f_\epsilon\Delta_2) \partial_x$ on the annulus $A = \{(\vartheta,x) \mid \vartheta \in I, \ x\in\R/L\Z\}$. More explicitly we have
\begin{equation}\label{comp_Z}
\begin{aligned} & Z^\vartheta = -\frac{\epsilon2\pi}{L}\beta(\vartheta)\Delta_2(\vartheta) \sin(2\pi x/L), \\ & Z^x = -\epsilon\beta'(\vartheta)\Delta_2(\vartheta) \cos(2\pi x/L) - (1+\epsilon\beta(\vartheta)\cos(2\pi x/L))\Delta_2'(\vartheta). \end{aligned}
\end{equation}
In view of~\eqref{mat_ident}, of $\Delta_2''(\vartheta^*)\neq0$ and of the fact that $I$ is small, we have $\Delta_2 \neq 0$ on $I$, $\Delta_2'$ does not vanish on $I\setminus \{\vartheta^*\}$ and has different signs on both components of $I\setminus \{\vartheta^*\}$. We assume $\Delta_2''(\vartheta^*)>0$, so that $\Delta'_2>0$ on $\{\vartheta\in I \mid \vartheta>\vartheta^*\}$ and $\Delta'_2<0$ on $\{\vartheta\in I \mid \vartheta<\vartheta^*\}$, the other case is analogous. It follows from~\eqref{comp_Z} that if $\epsilon$ is small enough then the zeros of $Z$ are $P_{\rm max} = (\vartheta^*,0)$ and $P_{\rm min} = (\vartheta^*,L/2)$. These rest points correspond to the only periodic $\lambda_0$-Reeb orbits on $\mathbb T$ which survive as periodic $\lambda_\epsilon$-orbits: they are the maximum and the minimum of $g_T$ on $S_T$.

We claim that for $\epsilon$ small enough the only periodic orbits of $X_\epsilon$ inside $\mathcal O$ with $\lambda_\epsilon$-action $\leq S$ correspond to the rest points $P_{\rm max},P_{\rm min}$. If not, we find $\epsilon_n \to 0^+$ and a sequence of periodic $X_{\epsilon_n}$-trajectories $\gamma_n$ in $\mathcal O$ with $\lambda_\epsilon$-action $\leq S$ different from them. Up to a subsequence, there exists a periodic $\lambda_0$-Reeb orbit $\gamma$ in $\mathcal O$ with action $\leq S$ such that $\gamma_n \to \gamma$. But the only such orbits correspond to the points in the circle $\Lambda = \{(\vartheta^*,x) \mid x\in \R/L\Z\}$. Thus the projections $\Gamma_n$ of $\gamma_n$ to the annulus $A$ are periodic orbits of $Z$ converging to a point $P_* \in \Lambda$. It must be the case that $P_* = P_{\rm min}$. In fact, $Z^\vartheta$ has a definite sign near any point in $\Lambda \setminus \{P_{\rm max},P_{\rm min}\}$, which implies $P_* \in \{P_{\rm max},P_{\rm min}\}$. It is easy to check, using the above formulas and the assumption $\Delta_2''(\vartheta^*)>0$, that the characteristic equation of $DZ(P_{\rm max})$ looks like $t^2 - \epsilon_n k^2 = 0$, while that of $DZ(P_{\rm min})$ is of the form $t^2 + \epsilon_n k^2 = 0$, for certain values $k\neq0$. Thus $P_{\rm min}$ is elliptic while $P_{\rm max}$ is hyperbolic, and we cannot have $\Gamma_n \to P_{\rm max}$ since, otherwise, $\Gamma_n$ would bound a disk containing no singularities or containing $P_{\rm max}$ as the only singularity, contradicting $\chi(\mathbb D)=1$. This implies $\Gamma_n \to P_{\rm min}$. To obtain the desired contradiction, note that the eigenvalues of $DZ(P_{\rm min})$ are purely imaginary and proportional to $\sqrt{\epsilon_n}$ in absolute value. Thus the orbits $\Gamma_n$ take very long time to close up when $n$ is large, contradicting the fact that they have uniformly bounded period (which can be estimated in terms of $S$).

\begin{remark}
The characteristic equations of $DZ(P_{\rm max})$, $DZ(P_{\rm min})$ calculated above show that the corresponding periodic $\lambda_\epsilon$-Reeb orbit are non-degenerate. The orbits corresponding to $P_{\rm min}$, $P_{\rm max}$ clearly link $r$ times with $\bar L_0$ and $s$ times with $\bar L_1$.
\end{remark}

\begin{remark}
The above calculations assumed $\Delta_2''(\vartheta^*)>0$, and the case $\Delta_2''(\vartheta^*)<0$ is treated analogously. In the latter case $P_{\rm min}$ is hyperbolic and $P_{\rm max}$ is elliptic.
\end{remark}

Our arguments so far have proved that for each $X_0$-invariant torus $\mathbb T_i\subset \cup_{T'<S}N_{T'}$ there exists $\epsilon_i>0$ and a small neighborhood $\mathcal O_i$ of $\mathbb T_i$ such that $\lambda_\epsilon = \lambda_0$ on $S_\gamma \setminus (\cup_i\mathcal O_i)$ and, moreover, such that if $\epsilon<\epsilon_i$ then there are precisely two prime closed $\lambda_\epsilon$-Reeb orbits in $\mathcal O_i$ with action $\leq S$. These are reparametrizations of two closed $\lambda_0$-Reeb orbits in $\mathbb T_i$, are non-degenerate and have action close to $T_i$, where $T_i$ is the period of the prime closed $\lambda_0$-Reeb orbits foliating $\mathbb T_i$. One of them is hyperbolic and the other is elliptic, both have transverse Floquet multipliers close to $1$. Let $k_i$ be such that $k_iT_i$ is the maximum value in the action spectrum of $\lambda_0$ of the form $kT_i$ which is smaller than $S$, with $k\geq 1$. Since the Floquet multipliers of the two surviving orbits are close to $1$, all their interates up to the $k_i$-th iterate are also non-degenerate, and are precisely the closed $\lambda_\epsilon$-Reeb orbits in $\mathcal O_i$ with period $<S$ since their prime periods are close to $T_i$. Since $X_\epsilon=X_0$ on $S_\gamma \setminus (\cup_i\mathcal O_i)$, any closed $X_\epsilon$-orbit not contained in $\cup_i\mathcal O_i \cup \bar L_0 \cup \bar L_1$ has period larger than $S$.

Therefore, taking $\epsilon>0$ small enough there exists a diffeomorphism $\Phi:S^3\to S_\gamma$ such that $\Phi^*\lambda_\epsilon = f_S\lambda_0|_{S^3}$ with $f_S:S^3\to(0,+\infty)$ close to $f_{\theta_0,\theta_1}$ in $C^\infty$, satisfying the requirements of Proposition~\ref{prop_morse_bott_approx}.

\subsubsection{Existence, uniqueness and regularity of finite-energy cylinders}

We keep using the constructions made above. Let $\vartheta^*$ be determined by $\mathbb T_{(p,q)} = \{\vartheta=\vartheta^*\}$, where $\mathbb T_{(p,q)}$ is the unique $X_0$-invariant torus foliated by prime closed $\lambda_0$-Reeb orbits in homotopy class $(p,q)$. In the following we denote $T = T_{(p,q)}$ the prime period of the closed $\lambda_0$-Reeb orbits foliating $\mathbb T_{(p,q)}$, for simplicity. According to our constructions, there exists a small open interval $I$ centered at $\vartheta^*$ such that on the neighborhood $\mathcal O = \{\vartheta\in I\}$ we have
\begin{equation}\label{f_eps_near_torus}
\begin{array}{ccc} \lambda_\epsilon = f_\epsilon\lambda_0, & \mbox{with} & f_\epsilon = 1+\epsilon \cos(2\pi x/L). \end{array}
\end{equation}
Here $I$ is small enough such that $\beta|_I\equiv1$. We make use of the parameters $x,y$ adapted to $\mathbb T_{(p,q)}$ defined in~\eqref{coord_xy}. It follows from~\eqref{comp_Reeb_epsilon} that the Reeb vector field $X_\epsilon = X^\vartheta_\epsilon \partial_\vartheta + X^x_\epsilon \partial_x + X^y_\epsilon \partial_y$ is given by
\begin{equation*}
\begin{array}{ccc}
X^\vartheta_\epsilon = \left( -\frac{1}{f_\epsilon} \right)' \frac{\Delta_2}{\Delta},
& X^x_\epsilon = \left( -\frac{1}{f_\epsilon} \right)\frac{\Delta_2'}{\Delta},
& X^y_\epsilon = \left( \frac{1}{f_\epsilon} \right)\frac{\Delta_1'}{\Delta}.
\end{array}
\end{equation*}
where $'$ denotes differentiation with respect to $x$ when applied to functions of $x$, or differentiation with respect to $\vartheta$ when applied to functions of $\vartheta$. The critical points of $-\frac{1}{f_\epsilon}$ are $x=0$ (maximum) and $x=L/2$ (minimum). It was proved in Section~\ref{paragraph_perturb} above that the only closed $\lambda_\epsilon$-Reeb orbits contained in $S_\gamma\setminus \bar K_0$ representing the homotopy class $(p,q)$ with $\lambda_\epsilon$-action $<S$ are $P_{\rm max} := (x_{\rm max},T_{\rm max}=T(1+\epsilon))$, $P_{\rm min} := (x_{\rm min},T_{\rm min}=T(1-\epsilon))$ where 
\[
\begin{array}{ccc}
x_{\rm max}(t) \simeq (\vartheta^*,0,t(1+\epsilon)^{-1}) & \text{and} & x_{\rm min}(t) \simeq (\vartheta^*,L/2,t(1-\epsilon)^{-1}).
\end{array}
\]
The vectors $e^1_\epsilon = \partial_\vartheta$, $e^2_\epsilon = (\Delta_2\partial_x - \Delta_1\partial_y)/f_\epsilon$ define a frame of $\xi_0 = \ker\lambda_0$ near $\mathbb T_{(p,q)}$ which is $d\lambda_\epsilon$-symplectic on $\mathbb T_{(p,q)}$. We represent the linearized flows along $P_{\rm max}$ and $P_{\rm min}$ in the frame $\{e^1_\epsilon,e^2_\epsilon\}$ as symplectic paths of matrices $\varphi_+(t)$, $\varphi_-(t)$, respectively. They satisfy
\[
\dot\varphi_\pm(t) = \begin{pmatrix} 0 & \mp c_\epsilon \\ -\Delta_2''(\vartheta^*) & 0 \end{pmatrix} \varphi_\pm(t)
\]
where $c_\epsilon \to 0^+$ as $\epsilon\to0$. If $\Delta_2''(\vartheta^*)>0$ then $\mu(\varphi_+) = 0$, $\mu(\varphi_-) = -1$, and if $\Delta_2''(\vartheta^*)<0$ then $\mu(\varphi_+) = 1$, $\mu(\varphi_-) = 0$. In any case $\mu_{CZ}(P_{\rm max}) - \mu_{CZ}(P_{\rm min}) = 1$.

We consider $J_\epsilon \in \J_+(\xi_0)$ satisfying $J_\epsilon \cdot e^1_\epsilon = e^2_\epsilon$ near $\mathbb T_{(p,q)}$, and define an $\R$-invariant almost complex structure $\jtil_\epsilon$ on $\R\times S_\gamma$ by $$ \begin{array}{ccc} \jtil_\epsilon \cdot \partial_a = X_\epsilon, & & \jtil_\epsilon|_{\xi_0} \equiv J_\epsilon, \end{array} $$ where above we see $\xi_0$ as a $\R$-invariant subbundle of $T(\R\times S_\gamma)$.

For each $\epsilon>0$ small we consider the following elliptic problem.
\begin{equation*}
{\rm (PDE_\epsilon)} \ \ \ \ \ \ \ \left\{ \begin{aligned} & \util = (a,u) : \R\times\R/\Z \to \R\times S_\gamma \ \mbox{ is smooth and satisfies} \\ & \bar\partial_{\jtil_\epsilon}(\util)=0, \ 0<E(\util)<\infty, \\ & \util \mbox{ is asymptotic to } P_{\rm max} \mbox{ at } \{+\infty\}\times \R/\Z, \\ & \util \mbox{ is asymptotic to } P_{\rm min} \mbox{ at } \{-\infty\}\times \R/\Z. \end{aligned} \right.
\end{equation*}
We call two solutions $\util=(a,u),\vtil=(b,v)$ of ${\rm (PDE_\epsilon)}$ equivalent if there are constants $c,\Delta s,\Delta t$ such that $v(s,t) = u(s+\Delta s,t+\Delta t)$ and $b(s,t) = a(s+\Delta s,t+\Delta t) + c$. The set of equivalence classes of solutions will be denoted by $\M_\epsilon$. The conclusion of the proof of Proposition~\ref{prop_morse_bott_approx} is a consequence of the statement below.

\begin{lemma}\label{lemma_cyl_p,q}
If $\epsilon$ is small enough then $\M_\epsilon$ has exactly two elements, and every solution $\util=(a,u)$ of ${\rm (PDE_\epsilon)}$ satisfies
\begin{itemize}
\item $u(\R\times \R/\Z) \subset \mathbb T_{(p,q)}$,
\item the linearized Cauchy-Riemann equation at $\util$ is a surjective Fredholm operator.
\end{itemize}
\end{lemma}

\begin{proof}
The existence of two elements in $\M_\epsilon$ is proved by explicitly exhibiting solutions of (PDE$_\epsilon$) using the two Morse trajectories of the function $-\frac{1}{f_\epsilon}$ on the circle $\R/L\Z$.

Up to a harmless abuse of notation, the symplectization of $\mathcal O = \{\vartheta\in I\}$ is equipped with coordinates $(a,\vartheta,x,y)$. Let $x_\epsilon(s)$ solve
\begin{equation}\label{flow_eqn}
\dot x_\epsilon = \left( -\frac{T}{f_\epsilon} \right)' \circ x_\epsilon, \ \ x_\epsilon(0) = \frac{L}{4}
\end{equation}
and consider the parametrized cylinder $\util_\epsilon(s,t) = (a_\epsilon(s),\vartheta^*,x_\epsilon(s),Tt)$, where $$ a_\epsilon(s) = T \int_0^s f_\epsilon\circ x_\epsilon(r) dr. $$
Plugging formulas
\[
\begin{aligned}
& \partial_x = f_\epsilon \Delta_1 X_\epsilon - f_\epsilon\Delta_1X^\vartheta_\epsilon e^1_\epsilon + f_\epsilon^2 X^y_\epsilon e^2_\epsilon \\
& \partial_y = f_\epsilon \Delta_2 X_\epsilon - f_\epsilon\Delta_2X^\vartheta_\epsilon e^1_\epsilon - f_\epsilon^2X^x_\epsilon e^2_\epsilon
\end{aligned}
\]
we get
\begin{equation}\label{form_J_eps}
\jtil_\epsilon = \begin{pmatrix} 0 & 0 & -f_\epsilon\Delta_1 & -f_\epsilon\Delta_2 \\ X^\vartheta_\epsilon & 0 & -f_\epsilon^2X^y_\epsilon & f_\epsilon^2X^x_\epsilon \\ X^x_\epsilon & \Delta_2/f_\epsilon & -X^\vartheta_\epsilon\Delta_1\Delta_2 & -X^\vartheta_\epsilon\Delta_2^2 \\ X^y_\epsilon & -\Delta_1/f_\epsilon & X^\vartheta_\epsilon\Delta_1^2 & X^\vartheta_\epsilon \Delta_1\Delta_2 \end{pmatrix}.
\end{equation}
Note that all coefficients are functions of $x$ and $\vartheta$ only. Thus
\begin{equation}\label{J_at_u}
\jtil_\epsilon(\util_\epsilon(s,t)) = \begin{pmatrix} 0 & 0 & 0 & -f_\epsilon \\ \left(-\frac{1}{f_\epsilon}\right)' & 0 & -f_\epsilon & 0 \\ 0 & 1/f_\epsilon & 0 & -\left(-\frac{1}{f_\epsilon}\right)' \\ 1/f_\epsilon & 0 & 0 & 0 \end{pmatrix},
\end{equation}
where all coefficients are evaluated at $x_\epsilon(s)$.

We claim that the map $\util_\epsilon$ solves ${\rm (PDE_\epsilon)}$. Indeed, note that $$ \partial_s \util_\epsilon = \begin{pmatrix} Tf_\epsilon\circ x_\epsilon \\ 0 \\ \dot x_\epsilon \\ 0 \end{pmatrix} = \begin{pmatrix} Tf_\epsilon\circ x_\epsilon \\ 0 \\ -T(\frac{1}{f_\epsilon})'\circ x_\epsilon \\ 0 \end{pmatrix} $$ and, using~\eqref{J_at_u}, we get $$ \jtil_\epsilon(\util_\epsilon) \cdot \partial_t\util_\epsilon = T \jtil_\epsilon(\util_\epsilon) \cdot \partial_y = T \begin{pmatrix} -f_\epsilon\circ x_\epsilon \\ 0 \\ (\frac{1}{f_\epsilon})'\circ x_\epsilon \\ 0 \end{pmatrix}. $$ Adding these two equations we get $\bar\partial_{\jtil_\epsilon}(\util_\epsilon)=0$. One checks $E(\util_\epsilon)<\infty$ easily.

Thus $\util_\epsilon$ is one of the cylinders in the statement of Lemma~\ref{lemma_cyl_p,q}. The other cylinder can be obtained by considering a solution of $\dot x_\epsilon = (-T/f_\epsilon)'\circ x_\epsilon$ satisfying $x_\epsilon(0)=3L/4$.

To prove the uniqueness of cylinders we use Siefring's intersection theory from \cite{Sie}. Let us denote by $\util_1:=\util_{\epsilon,1}=(a_1,u_1)$ and $\util_2:=\util_{\epsilon,2}=(a_2,u_2)$ the finite energy cylinders found above. We fix $\epsilon>0$ small and omit the dependence on $\epsilon$ in the notation for simplicity. Recall that both $\util_1$ and $\util_2$ solve ${\rm PDE_\epsilon}$ and are asymptotic to $P_{\max}=(x_{\max},T_{\max})$ and $P_{\min}=(x_{\min},T_{\min})$ as $s \to +\infty$ and $s\to -\infty$, respectively, where $T_{\max}$ and $T_{\min}$ are their prime periods. Let $x_{T_{\max}}(t):=x_{\max}(tT_{\max})$ and $x_{T_{\min}}(t):=x_{\min}(tT_{\min})$.

One can easily check that   $$\left\{ \begin{aligned} & \util_i \mbox{ and } u_i \mbox{ are embeddings for } i=1,2, \\ &  {\rm Image}(u_i) \cap P_{\max} = {\rm Image}(u_i) \cap P_{\min} = \emptyset  \mbox { for } i=1,2, \\ &  {\rm Image}(u_1) \cap {\rm Image}(u_2)= \emptyset. \end{aligned} \right.$$

Recall that according to our conventions established in the beginning of Section~\ref{section_contact_homology}, whenever $P = (x,T)$ is a closed $\lambda_\epsilon$-Reeb orbit we denote by $\beta_P$ the homotopy class of $d\lambda_\epsilon$-symplectic trivializations of $(x_T)^*\xi_0$ induced by a global $d\lambda_\epsilon$-symplectic trivialization of $\xi_0$. In the following we will denote this homotopy class by $\beta$, for simplicity, without a direct reference to the orbit $P$.

Let $A_{P_{\max}}$ and $A_{P_{\min}}$ be the respective asymptotic operators defined in Section~\ref{asymptotic_behavior_section} associated to $P_{\max}$ and $P_{\min}$. As we saw, we have well-defined winding numbers $\wind^{\geq 0}(P_{\max},\beta),$ $\wind^{<0}(P_{\max},\beta),$ $\wind^{\geq 0}(P_{\min},\beta)$ and $\wind^{<0}(P_{\min},\beta)$ associated to $A_{P_{\max}}$ and $A_{P_{\min}}$.

For each $\util=(a,u)$ representing an element in $\M_\epsilon$ one can find eigensections $\eta_+:\R / \Z \to (x_{T_{\max}})^*\xi_0$ and $\eta_-:\R / \Z \to (x_{T_{\min}})^*\xi_0$ of  $A_{P_{\max}}$ and $A_{P_{\min}}$, respectively, with $A_{P_{\max}}\eta_+ = \mu_+ \eta_+$ and  $A_{P_{\min}}\eta_- = \mu_- \eta_-$, $\mu_+<0$, $\mu_->0$, and a diffeomorphism $\psi : \R \times \R/\Z \to \R \times \R/\Z$ such that \begin{equation}\label{sie1}  \begin{aligned} u\circ \psi(s,t)= \exp_{x_{T_{\max}}(t)} \{e^{\mu_+ s}(\eta_+(t)+r_+(s,t))\} \mbox { for } s\gg 0, \\ u \circ \psi(s,t)= \exp_{x_{T_{\min}}(t)}\{e^{\mu_-s}(\eta_-(t)+r_-(s,t))\} \mbox { for } s\ll 0, \end{aligned}\end{equation}  where $|r_{\pm} (s,t)| \to 0$ as $s \to \pm \infty$ uniformly in $t$.  Here, $\exp$ denotes the exponential map of the Riemannian metric $g_\epsilon$ on $S_\gamma$ given by $g_\epsilon = \lambda_\epsilon \otimes \lambda_\epsilon + d\lambda_\epsilon(\cdot ,J_\epsilon \cdot )$. Let 
\[
\begin{aligned} 
U_+(s,t):= & e^{\mu_+ s}(\eta_+(t)+r_+(s,t)) \ \text{defined for} \ s\gg0, \\ U_-(s,t):= & e^{\mu_- s}(\eta_-(t)+r_-(s,t)) \ \text{defined for} \ s\ll0.
\end{aligned}
\]
The triple $(U_+,U_-,\psi)$ is called an asymptotic representative of $\util$.  For a proof of the existence of an asymptotic representative, see~\cite{Sie0}. Here, differently from~\cite{Sie0,Sie}, we represent both ends simultaneously. We may use the notation $\eta_+(\util), \mu_+(\util)$ etc to emphasize their dependence on $\util.$

In view of the asymptotic behavior at the ends, we observe that   \begin{equation}\label{desig} \begin{aligned} \wind(\eta_+,\beta) \leq &  \wind^{<0}(P_{\max},\beta) \\ \wind(\eta_-,\beta) \geq & \wind^{\geq 0}(P_{\min},\beta). \end{aligned} \end{equation} We claim that 
\begin{equation}\label{eq_prop1}
\begin{aligned} \wind(\eta_+,\beta) = & \wind^{<0}(P_{\max},\beta) \\ \wind(\eta_-,\beta)= & \wind^{\geq 0}(P_{\min},\beta). \end{aligned}
\end{equation}

To see this recall that $\mu_{CZ}(P_{\max})= \mu_{CZ}(P_{\min})+1$. From the definition of the Conley-Zehnder index, this implies that 
\begin{equation} \label{cz1}
\begin{aligned}
2\wind^{<0}(P_{\max},\beta)+p_+=&\mu_{CZ}(P_{\max}) \\ 
= & \mu_{CZ}(P_{\min})+1 \\ = & 2\wind^{<0}(P_{\min},\beta)+p_-+1,
\end{aligned} 
\end{equation} 
where $p_+=\wind^{\geq 0}(P_{\max},\beta)-\wind^{<0}(P_{\max},\beta)$ and $p_- = \wind^{\geq 0}(P_{\min},\beta)-\wind^{<0}(P_{\min},\beta).$ Since $p_+\in \{0,1\}$ and $p_-\in \{0,1\}$, we conclude by parity reasons in \eqref{cz1} that \begin{equation}\label{cz2} p_++p_-=1.\end{equation} From \cite[Proposition 5.6]{props2}, we have \begin{equation}\label{proper2} 0 \leq \wind_\pi(\util) = \wind(\eta_+,\beta)- \wind(\eta_-,\beta) \end{equation} where the (non-negative integer-valued) invariant $\wind_\pi(\util)$ was introduced in~\cite{props2}. It follows from \eqref{desig},~\eqref{cz1}-\eqref{proper2} that  $$\begin{aligned}
0 \leq & \wind(\eta_+,\beta)- \wind(\eta_-,\beta) \\ \leq & \wind^{< 0}(P_{\max},\beta)- \wind^{\geq 0}(P_{\min},\beta) \\ = & \wind^{< 0}(P_{\max},\beta)- \wind^{< 0}(P_{\min},\beta)-p_-\\ = & \frac{1-p_--p_+}{2}=  0.\end{aligned} $$ We conclude that $$\begin{aligned} \wind(\eta_+,\beta)= & \wind(\eta_-,\beta),\\ \wind^{< 0}(P_{\max},\beta)=& \wind^{\geq 0}(P_{\min},\beta).\end{aligned}$$ Now observe again from \eqref{desig} that  $$\wind(\eta_+,\beta)\leq \wind^{< 0}(P_{\max},\beta)=\wind^{\geq 0}(P_{\min},\beta)\leq \wind(\eta_-,\beta),$$ which proves~\eqref{eq_prop1}.

Now we claim that
\begin{equation}\label{eq_prop2}
{\rm Image}(u) \cap P_{\max} =  {\rm Image}(u) \cap P_{\min} = \emptyset.
\end{equation}
To prove this first note that, since $\eta_\pm(t)$ never vanishes the map $u \circ \psi(s,t)$ does not intersect $P_{\rm max},P_{\rm min}$ when $|s|$ is large enough. We will construct a homotopy between $u$ and $u_1$ so that no intersections with $P_{\rm max}$ and $P_{\rm min}$ are created or destroyed near the ends of the domain. This implies that the algebraic intersection numbers of $u$ and $u_1$ with both $P_{\max}$ and $P_{\min}$ coincide. Since all intersections count positively and $u_1$ does not intersect $P_{\max}$ and $P_{\min}$, the claim follows.

The homotopy will be constructed in two steps. Let $(U_+,U_-,\psi)$, $(U_{1+},U_{1-},$
$\psi_1)$ be asymptotic representatives of $\util=(a,u)$ and $\util_1=(a_1,u_1)$, respectively, with eigensections $\eta_+,\eta_-,\eta_{1+}$ and $\eta_{1-}$.   We will denote $u \circ \psi$ and $u_1 \circ \psi_1$ simply by $u$ and $u_1$, respectively. By~\eqref{eq_prop1}, we know that $\wind(\eta_+,\beta)=\wind(\eta_{1+},\beta)$ and $\wind(\eta_-,\beta)=\wind(\eta_{1-},\beta)$. From the properties of the asymptotic operator $A_{P_{\max}}$ explained in~\cite{props2}, we have three possibilities (see Lemma 3.5 in~\cite{props2}):

\begin{itemize}\item[(i)] $\eta_+(t)=c\eta_{1+}(t)$ for a positive constant $c$ for all $t$. \item[(ii)] $\eta_+(t)$ and $\eta_{1+}(t)$ are linearly independent for all $t$. \item[(iii)] $\eta_+(t)=c\eta_{1+}(t)$ for a negative constant $c$ for all $t$. \end{itemize}

Cases (i) and (ii) are treated similarly. Given $M>0$ large, choose a smooth function $\gamma_M:\R \times [0,1] \to [0,1]$ satisfying $\gamma_M(s,\mu)=\mu$ if $s>M$ and $\gamma_M(s,\mu)=0$ if $s< M-1$. Define the homotopy $H:[0,1]\times \R \times \R/\Z \to S_\gamma$  by $$H(\mu,s,t) = \exp_{x_{T_{\max}}(t)}\{ (1-\gamma_M(s,\mu))U_+(s,t)+\gamma_M(s,\mu)U_{1+}(s,t)\}$$ for $s\geq M-1$ and $H(\mu,s,t)=u(s,t)$ for $s<M-1$. To see that no intersection with $P_{\rm max}$ is created or destroyed for $s> M$, if $M$ is sufficiently large, note that for $s > M$ we have $$\begin{aligned} & (1-\gamma_M(s,\mu))U_+(s,t)+\gamma_M(s,\mu)U_{1+}(s,t) \\ & =(1-\mu)e^{\mu_+ s}(\eta_+(t)+r_+(s,t)) + \mu e^{\mu_{1+}s}(\eta_{1+}(t)+r_{1+}(s,t)), \end{aligned}$$ which never vanishes since, for each large $s$, $\eta_+(t)+r_+(s,t)$ and $\eta_{1+}(t)+r_{1+}(s,t)$ are never a negative multiple of each other for all $t$. New intersections with $P_{\rm min}$ do not appear since this homotopy is supported near $P_{\max}$.

In case (iii), given $\varepsilon>0$ small and $M>0$ large, let $H : [0,1]\times \R \times \R/\Z \to S_\gamma$ be the homotopy defined by $$H(\mu,s,t)=\exp_{x_{T_{\max}}(t)} R(\varepsilon \gamma_M(s,\mu))U_+(s,t), $$ for $s\geq M-1$  and $H(\mu,s,t)=u(s,t)$ for $s<M-1$. Here, $\varepsilon>0$ is small and $R(\theta)$ denotes the rotation by an angle $\theta$ in the fiber coordinates given by an {\it a priori} choice of trivialization in class $\beta$. Clearly no intersection with $P_{\max}$ is created or destroyed if $s>M$ and the same holds for intersections with $P_{\min}$. After performing this first homotopy, we proceed as in case (ii) in order to construct a second homotopy to $u_1$ near $P_{\max}$.

Now we proceed in the same way to construct a homotopy supported near $P_{\min}$ so that we end up with a map $\bar u: \R \times \R/\Z \to S_\gamma$ which coincides with $u_1$ for $|s|>M\gg1$  and the algebraic intersection numbers of $\bar u$ with both $P_{\max}$ and $P_{\min}$ coincide with those of $u$.

Now, choose a point $p \in S_{\gamma}$ not contained in the images of $\bar u$ and $u_1$ and consider a diffeomorphism $\Psi: S_\gamma \setminus \{p\} \to \R^3$. Define the homotopy $H_1:[0,1] \times \R\times\R/\Z \to S_\gamma$ between $\bar u$ and $u_1$ by $$H_1(\mu,s,t)=\Psi^{-1} ((1-\mu)\Psi \circ \bar u(s,t) + \mu \Psi \circ u_1(s,t)).$$ Note that this homotopy is supported in $\{|s|\leq M\}$ and, therefore the algebraic intersection numbers of $\bar u$ and $u_1$ with both $P_{\max}$ and $P_{\min}$ coincide.  We conclude that $u$ does not intersect either $P_{\max}$ or $P_{\min}$ by positivity of intersections, and~\eqref{eq_prop2} is proved.

In~\cite{Sie}, a generalized intersection number $[\util] * [\vtil]\in \Z$ is defined for two pseudo-holomorphic curves $\util=(a,u)$ and $\vtil=(b,v)$. It counts  the actual algebraic intersection number between $\util$ and $\vtil$ plus the asymptotic intersection number, which is computed by carefully analyzing their asymptotic behavior at the punctures. An application of~\eqref{eq_prop1},~\eqref{eq_prop2} and Corollary 5.9 from~\cite{Sie} (see conditions (1) and (3)) gives \\

\noindent {Claim A.} Let $\util=(a,u),\vtil=(b,v)$ represent classes in $\M_\epsilon$. Then $[\tilde u] * [\tilde v]=0$. \\

\begin{definition}[Siefring~\cite{Sie}]
Let $\util,\vtil$ represent distinct elements in $\M_\epsilon$ and $(U_+(\util),U_-(\util),\psi(\util))$, $(U_+(\vtil),U_-(\vtil),\psi(\vtil))$ be their asymptotic representatives with eigensections $\eta_+(\util),\eta_-(\util),$ $\eta_+(\vtil)$ and $\eta_-(\vtil)$. We say that $\util$ and  $\vtil$ approach $P_{\rm max}$ in the same direction if $\eta_+(\util)=c \eta_+(\vtil)$ for a positive constant $c$.  Similarly, we say that $\util$ and $\vtil$ approach $P_{\rm min}$ in the same direction if $\eta_-(\util)=c \eta_-(\vtil)$ for a positive constant $c$.
\end{definition}

Now we finally prove that  $\M_\epsilon$ has exactly two elements, namely the equivalence classes of $\util_1$ and $\util_2$. Assume indirectly the existence of a third element in $\M_\epsilon$ represented by $\util_3$. Since $\mu_{CZ}(P_{\max})=\mu_{CZ}(P_{\min})+1$, either $\mu_{CZ}(P_{\max})$ or $\mu_{CZ}(P_{\min})$ is even. Assume without loss of generality that $\mu_{CZ}(P_{\max})$ is even. Let $(U_{i+},U_{i-},\psi_i)$  be asymptotic representatives of $\util_i$, $i=1,2,3$, with respective eigensections $\eta_{i+}$. From~\eqref{eq_prop1} we see that $\wind(\eta_{i+},\beta)= \wind(\eta_{j+},\beta)$, $\forall i,j\in\{1,2,3\}$. Since $\mu_{CZ}(P_{\max})$ is even, we must have $\eta_{i+}(t)= c_{ij}\eta_{j+}(t)$, $\forall t$ for non-vanishing constants $c_{ij}$. Here it was used that the negative extremal eigenvalue of the asymptotic operator at an even hyperbolic orbit has $1$-dimensional eigenspace, and is the only negative eigenvalue with that given winding number. It follows that there exist $i_0 \neq j_0\in \{1,2,3\}$ so that $c_{i_0j_0}>0$. Theorem~2.5 from~\cite{Sie} implies that $[\util_{i_0}] * [\util_{j_0}]>0.$ However, this contradicts Claim A and proves uniqueness of cylinders.

To handle regularity we use Theorem~1 from~\cite{wendl_transv}. The (unparametrized) Fredholm index with no asymptotic constraints of the solutions $\util_\epsilon$ constructed above is ${\rm Ind}(\util_\epsilon) = \mu_{CZ}(P_{\rm max})-\mu_{CZ}(P_{\rm min}) = 1$. We identify $\R\times\R/\Z \simeq \C P^1\setminus \Gamma$ via $(s,t) \simeq [e^{2\pi(s+it)}:1]$ where $\Gamma = \{[0:1],[1:0]\}$, and see $\util_\epsilon$ as a pseudo-holomorphic map defined in $\C P^1\setminus \Gamma$. Since $\emptyset = \partial \C P^1$, Remark~1.2 from~\cite{wendl_transv} tells us that $\util_\epsilon$ is regular if
\begin{equation}\label{ineq_fredholm_index}
1 = {\rm Ind}(\util_\epsilon) > -\chi(\C P^1) + \#\Gamma_0 + 2Z(d\util_\epsilon).
\end{equation}
Here $\Gamma_0 \subset \Gamma$ is the set of punctures where the Conley-Zehnder index of the corresponding asymptotic orbit is even and $Z(d\util_\epsilon)$ is the sum of the order of the critical points of $\util_\epsilon$. Note from the definition of $\util_\epsilon$ that $d\util_\epsilon(z)\neq 0$, $\forall z$. Now, since $\mu_{CZ}(P_{\rm max})-\mu_{CZ}(P_{\rm min})=1$ we have $\#\Gamma_0=1$ and, consequently, the right hand side of~\eqref{ineq_fredholm_index} is equal to $-1$. The proof of Lemma~\ref{lemma_cyl_p,q} is now finished. 
\end{proof}

Proposition~\ref{prop_morse_bott_approx} follows immediately from Lemma~\ref{lemma_cyl_p,q}.

\subsection{A non-trivial chain map}\label{inclusions}

We choose $0<c<1$, $T>0$ and consider a contact form $\lambda = h\lambda_0$ with $h\in\F$, satisfying the conditions:
\begin{itemize}
\item All closed $\lambda$-Reeb orbits with action $\leq T/c$ are non-degenerate;
\item There is no closed $\lambda$-Reeb orbit in $S^3\setminus K_0$ with action $\leq T/c$ which is contractible in $S^3\setminus K_0$;
\item The transverse Floquet multipliers of $L_0,L_1$ seen as prime closed $\lambda$-Reeb orbits are of the form $e^{i2\pi\alpha}$ with $\alpha \not\in \Q$.
\end{itemize}

Let $(p,q)$ be a relatively prime pair of integers, which represent an element in $\pi_1(S^3\setminus K_0,{\rm pt}) \simeq \Z\times\Z$ via the isomorphism~\eqref{pi1_isomorphism}. Consider $\jhat_+ \in \J(\lambda)$ induced by some $d\lambda_0$-compatible complex structure $J : \xi_0 \to \xi_0$ as explained in Section~\ref{cyl_alm_cpx_str}. We assume that $\jhat_+$ is regular with respect to the homotopy class $(p,q)$ and action bound $\leq T/c$, see Section~\ref{chain_complex_subsection} for more details. It follows that the almost complex structure $\jhat_- \in \J(c\lambda)$ induced by $J$ and $c\lambda$ is also regular with respect to the homotopy class $(p,q)$ and action bound $\leq T$. In fact, consider the diffeomorphisms 
\[
\begin{array}{cc}
\varphi : \R\times S^3 \to \R\times S^3, & \varphi(a,x) = (\frac{1}{c}a,x)
\end{array}
\]
and 
\[
K = (\Psi_{\lambda})^{-1}  \circ \varphi \circ \Psi_{\lambda} : W_{\xi_0} \to W_{\xi_0}
\]
where $\Psi_\lambda$ is the diffeomorphism~\eqref{coordinates_symplectization}. Then $K^*\jhat_+ = \jhat_-$, so that finite-energy $\jhat_-$-holomorphic cylinders are precisely of the form $K^{-1}\circ \util$, where $\util$ is some finite-energy $\jhat_+$-cylinder. This observation also shows that the obvious identification defined by
\begin{equation}\label{map_j_*}
  \begin{aligned}
    j_* : C^{\leq T/c,(p,q)}_*(\lambda) &\simeq C^{\leq T,(p,q)}_*(c\lambda) \\
    q_{(x(t),T')} &\simeq q_{(x(t/c),cT')}
  \end{aligned}
\end{equation}
where $(x(t),T') \in \P^{\leq T/c,(p,q)}(\lambda)$ and $(x(t/c),cT') \in \P^{\leq T,(p,q)}(c\lambda)$, is a chain map that induces an isomorphism at the homology level. In fact, there is a 1-1 correspondence between the relevant moduli spaces used to define the differentials $\partial_{(c\lambda,J)}$ and $\partial_{(\lambda,J)}$, proving that $\partial_{(c\lambda,J)} \circ j_* = j_* \circ \partial_{(\lambda,J)}$.

By Theorem~\ref{chain_homotopy_thm} there is a well-defined chain map $$ \Phi(\bar J)_* : C^{\leq T,(p,q)}_*(\lambda) \to C^{\leq T,(p,q)}_*(c\lambda) $$ for any given $\bar J \in \J_{\rm reg}(\jhat_-,\jhat_+ : K_0)$ (with respect to action bound $\leq T/c$ and homotopy class $(p,q)$). Consider the inclusion map
\begin{equation}\label{map_iota_*}
\begin{aligned}
\iota_* : C^{\leq T,(p,q)}_*(\lambda) &\hookrightarrow C^{\leq T/c,(p,q)}_*(\lambda) \\
q_{(x,T')} &\mapsto  q_{(x,T')}.
\end{aligned}
\end{equation}

\begin{lemma}\label{j_iota_relation_chain}
The chain maps $\Phi(\bar J)_*$ and $j_* \circ \iota_*$ are chain homotopic.
\end{lemma}

\begin{proof}
We claim that $j_*\circ \iota_* = \Phi(\bar J')_*$, for some $\bar J' \in \J_{\rm reg}(\jhat_-,\jhat_+:K_0)$ regular with respect to action bound $\leq T/c$ and homotopy class $(p,q)$. To see this consider a function $g:\R\to\R$ such that $g \equiv 1/c$ near $(-\infty,\ln c]$, $g\equiv 1$ near $[0,+\infty)$ and $g' \leq 0$, and define a $d(e^a\lambda)$-compatible almost complex structure $\jtil$ on $\R\times S^3$ by
\[
\begin{array}{cccc}
\jtil \cdot \partial_a = g X_\lambda, & \jtil \cdot X_\lambda = -\frac{1}{g}\partial_a  & \text{ and } & \jtil|_{\xi_0} \equiv J.
\end{array}
\]
Recalling the map $\Psi_\lambda$ from~\eqref{coordinates_symplectization} we note that $\bar J' := (\Psi_\lambda)^*\jtil \in \J(\jhat_-,\jhat_+:K_0)$.

Let us consider a positive diffeomorphism $G:\R\to\R$ solution of the initial value problem
\[
\begin{array}{cc}
G'(a) = \frac{1}{g(G(a))}, & G(0) = 0.
\end{array}
\]
This can be used to define diffeomorphisms
\[
\begin{array}{cc}
F: \R\times S^3 \to \R\times S^3, & F(a,x) = (G(a),x)
\end{array}
\]
and
\begin{equation}\label{}
\begin{array}{cc}
H: W_{\xi_0} \to W_{\xi_0}, & H = (\Psi_{\lambda})^{-1} \circ F \circ \Psi_{\lambda}.
\end{array}
\end{equation}
One checks that $H^*\bar J' = \jhat_+$. In fact, $F^*\jtil \equiv J$ on $\xi_0$ and
\[
\begin{aligned}
(F^*\jtil)|_{(a,x)} \cdot \partial_a &= dF^{-1}|_{(G(a),x)} \cdot \jtil|_{(G(a),x)} \cdot dF|_{(a,x)} \cdot \partial_a \\
&= G'(a) \ dF^{-1}|_{(G(a),x)} \cdot \jtil|_{(G(a),x)} \cdot \partial_a \\
&= G'(a) g(G(a)) \ dF^{-1}|_{(G(a),x)} \cdot X_\lambda|_x \\
&= X_\lambda|_x
\end{aligned}
\]
which gives the desired conclusion since $F^*\jtil$ is an almost complex structure. This proves $\bar J' \in \J_{\rm reg}(\jhat_-,\jhat_+:K_0)$ since $\jhat_+ \in \J_{\rm reg}(\lambda)$ is regular with respect to action bound $\leq T/c$ and homotopy class $(p,q)$. Consider an orbit $P \in \P^{\leq T,(p,q)}(\lambda)$. Then
\begin{equation}\label{map_phi_j'}
  \Phi(\bar J')_*(q_P) = \sum_{\substack{P' \in \P^{\leq T,(p,q)}(c\lambda) \\ \mu_{CZ}(P) = \mu_{CZ}(P')}} \left(\#_2 \M^{\leq T,(p,q)}_{\bar J'}(P,P') \right) q_{P'}.
\end{equation}
Recall the set $\M^{\leq T,(p,q)}_{\jhat_+}(P,P'')$ of finite-energy $\jhat_+$-holomorphic cylinders with image in $\tau^{-1}(S^3\setminus K_0)$ asymptotic to $P,P'' \in \P^{\leq T,(p,q)}(\lambda)$ at the positive and negative punctures, respectively, modulo holomorphic reparametrizations. This set was defined in Section~\ref{chain_complex_subsection} and we do not quotient out by the $\R$-action on the target. $\M^{\leq T,(p,q)}_{\jhat_+}(P,P'')$ is a smooth manifold of dimension $\mu_{CZ}(P) - \mu_{CZ}(P'')$ since $\jhat_+ \in \J_{\rm reg}(\lambda)$ and $(p,q)$ is relatively prime. The biholomorphism $H$ induces a 1-1 correspondence between moduli spaces $$ \M^{\leq T,(p,q)}_{\bar J'}(P,P') \ \text{ and } \ \M^{\leq T,(p,q)}_{\jhat_+}(P,P'') $$ where $j_*(q_{P''}) = q_{P'}$. However, $\M^{\leq T,(p,q)}_{\jhat_+}(P,P'')$ is empty when $P'' \neq P$, or consists of a single (trivial) cylinder when $P'' = P$. We conclude that the right side of~\eqref{map_phi_j'} is equal to $j_*(q_{P}) = j_* \circ \iota_*(q_{P})$ and, in particular, that $j_* \circ \iota_* = \Phi(\bar J')_*$. The lemma now follows from Theorem~\ref{comparing_chain_maps_thm}.
\end{proof}

We will now apply the above discussion to our model contact forms. Let us choose $\theta_0,\theta_1 \not\in \Q$, and let $f_{\theta_0,\theta_1}\lambda_0$ be the model contact forms discussed in Section~\ref{example-2}.  If $(p,q)$ is a relatively prime pair of integers satisfying~\eqref{non_resonance}, then the closed $(f_{\theta_0,\theta_1}\lambda_0)$-Reeb orbits in $S^3\setminus K_0$ representing the homotopy class $(p,q)$ are necessarily prime orbits and have the same period, which we denote by $T_{(p,q)}>0$.

Select $0<c<1$, $T>T_{(p,q)}$ and $S>T/c$. By Proposition~\ref{prop_morse_bott_approx}, there is $f_S \in \F$ arbitrarily $C^\infty$-close to $f_{\theta_0,\theta_1}$ and some $d\lambda_0$-compatible complex structure $J_S : \xi_0 \to \xi_0$ such that the homology $HC^{\leq S,(p,q)}_*(f_S\lambda_0,J_S)$ is well-defined and 
\[
\begin{aligned}
HC^{\leq T,(p,q)}_*(f_S\lambda_0,J_S) = HC^{\leq T/c,(p,q)}_*(f_S\lambda_0,J_S) & = HC^{\leq S,(p,q)}_*(f_S\lambda_0,J_S) \\
& \cong H_{*-s}(S^1;\mathbb{F}_2)
\end{aligned}
\]
for some $s\in\Z$. It is also clear that the homology of these complexes are in fact generated by the same closed Reeb orbits, and the differentials count the same cylinders. In particular, we have shown that the inclusion map $\iota_*$ defined in~\eqref{map_iota_*}
\[
(C^{\leq T,(p,q)}_*(f_S\lambda_0),\partial_{(f_S\lambda_0,J_S)}) \overset{\iota_*}{\hookrightarrow} (C^{\leq T/c,(p,q)}_*(f_S\lambda_0),\partial_{(f_S\lambda_0,J_S)})
\]
is non-trivial at the level of homology. Since
\[
j_* : (C^{\leq T/c,(p,q)}_*(f_S\lambda_0),\partial_{(f_S\lambda_0,J_S)}) \to (C^{\leq T,(p,q)}_*(cf_S\lambda_0),\partial_{(cf_S\lambda_0,J_S)})
\]
is an isomorphism of chain complexes, and therefore an isomorphism at the homology level, we obtain the following statement.

\begin{proposition}\label{comp_j_iota}
Choosing $T$, $c$, $S$, $f_S$ and $J_S$ as above, the map $$ j_* \circ \iota_* : HC^{\leq T,(p,q)}_*(f_S\lambda_0,J_S) \to HC^{\leq T,(p,q)}_*(cf_S\lambda_0,J_S) $$ is non-zero.
\end{proposition}

\section{Proof of main theorem in the non-degenerate case}\label{proof_non_deg_section}

In this section we prove Theorem~\ref{thm-1} assuming that the tight contact form on $S^3$ as in the statement is non-degenerate. We aim to prove

\begin{proposition}\label{prop-non-degenerate_case}
Consider a sequence $f_n \in \F$ such that $\lambda_n = f_n\lambda_0$ is non-degenerate for each $n$, and assume that there are uniform bounds
\[
0 < m < \inf_{x,n} f_n(x) < \sup_{x,n} f_n(x) < M.
\]
Suppose $(p,q)$ is a relatively prime pair of integers, and also that there are numbers $\theta_0,\theta_1$ satisfying
\begin{equation}\label{non_resonance_prop}
\begin{array}{ccc}
(\theta_0,1) < (p,q) < (1,\theta_1) & \text{or} & (1,\theta_1) < (p,q) < (\theta_0,1)
\end{array}
\end{equation}
and
\[
\lim_{n\to\infty} \theta_i(f_n) = \theta_i \ \ \ (i=0,1).
\]
Then there is a $T > 0$ independent of $n$ such that for each $n$ sufficiently large there is a simple closed $\lambda_n$-Reeb orbit $P_n \subset S^3\setminus K_0$ of period less than $T$ satisfying $\link(P_n, L_0) = p$ and $\link(P_n,L_1) = q$.
\end{proposition}

Here $\theta_i(f_n) = \rho(L_i,\lambda_n) - 1$, where $\rho(L_i,\lambda_n)$ is the transverse rotation number of $L_i$ seen as a prime periodic orbit of the Reeb flow associated to the contact form $\lambda_n$ computed with respect to a global positive trivialization of $\xi_0$, see~\eqref{eq_defn_rot_number} in Section~\ref{cz_rotation_orbits}.

Note that for each $n$ the link $K_0$ consists of a pair of closed orbits for the Reeb flow of $\lambda_n$ since $f_n\in \F$, but we do not assume that these orbits are elliptic or that contact homology in the complement of $K_0$ discussed in Section~\ref{section_contact_homology} is well-defined for the contact forms $\lambda_n$. Our argument combines several constructions, such as chain maps, stretching-the-neck, 
SFT compactness, and asymptotic analysis to deduce existence of the desired periodic orbit for this general type of contact form.

Theorem~\ref{thm-1} in the non-degenerate case follows from Proposition~\ref{prop-non-degenerate_case} by considering a constant sequence.

\subsection{Computations with homotopy classes}\label{computing_homotopy}

Let $h\in \F$ and assume that $h\lambda_0$ is a non-degenerate contact form. Then we may view $L_0 = (x_0,T_0)$ and $L_1 = (x_1,T_1)$ as prime closed orbits of the flow associated to the Reeb vector field $X_{h\lambda_0}$.

Let us fix $k\geq 1$ and $i\in\{0,1\}$, and suppose that $\nu$ is a non-zero eigenvalue of the asymptotic operator $A_{L_i^k}$ associated to the contact form $h\lambda_0$, some $J$ and the orbit $L_i^k = (x_i,kT_i)$. If $t\in \R/\Z \mapsto \eta(t)\in \xi_0|_{x_i(kT_it)}$ is a non-zero section in the eigenspace of $\nu$ and $\epsilon>0$ is small enough then $$ t\in \R/\Z \mapsto \eta_\epsilon(t) := \exp_{x_i(kT_it)}(\epsilon\eta(t)) $$ is a closed loop in $S^3\setminus K_0$ and its homotopy class in $S^3\setminus K_0$ does not depend on $\epsilon$.

\begin{lemma}\label{estimating_homotopy_lemma}
Suppose $m = \link(\eta_\epsilon(t),L_0)$ and $n = \link(\eta_\epsilon(t),L_1)$.
\begin{itemize}
  \item If $i=0$ then $n=k>0$ and
  \[
    \begin{array}{ccc}
      \nu > 0 \Rightarrow \frac{m}{n} \geq \theta_0(h) &  & \nu<0 \Rightarrow \frac{m}{n} \leq \theta_0(h).
    \end{array}
  \]
  \item If $i=1$ then $m=k>0$ and
  \[
    \begin{array}{ccc}
      \nu > 0 \Rightarrow \frac{n}{m} \geq \theta_1(h) &  & \nu<0 \Rightarrow \frac{n}{m} \leq \theta_1(h).
    \end{array}
  \]
\end{itemize}
\end{lemma}

\begin{proof}
We only prove the lemma for $i=0$, the case $i=1$ is analogous. Note that $\pi_1(S^3\setminus K_0,{\rm pt}) \simeq \Z\times \Z$, where an explicit isomorphism is given by $$ [\gamma] \simeq (\link(\gamma,L_0),\link(\gamma,L_1)). $$ Thus, since $\eta_\epsilon$ is $C^\infty$-close to $L_0^k$ we get $n = \link(\eta_\epsilon,L_1) = \link(L_0^k,L_1) = k\geq 1$. The orbit $L_0$ is unknotted and spans an embedded disk $D_0\subset S^3$, and we let the orientation of $L_0$ by the Reeb vector field induce an orientation on $D_0$. Choosing non-vanishing sections $W$ of $(\xi_0 \cap TD_0)|_{L_0}$ and $Z$ of $\xi_0|_{D_0}$ we have $\wind(Z|_{L_0},W) = \sl(L_0) = -1$, where the winding is computed seeing $Z|_{L_0}$ and $W$ as sections of the (oriented by $d(h\lambda_0)$) vector bundle $({x_0}_{T_0})^*\xi_0 \to \R/\Z$, see Remark~\ref{winding_numbers}. Here $\sl(L_0)$ denotes the self-linking number of $L_0$. Thus, if we denote by $\beta_{\rm disk}$ the homotopy class of $d(h\lambda_0)$-symplectic frames of $({x_0}_{T_0})^*\xi_0$ induced by a frame containing $W$ we have
\begin{equation}\label{}
  \rho(L_0,\beta_{\rm disk}) = \rho(L_0,\beta_{L_0}) - 1 = \theta_0(h).
\end{equation}
Now we compute
\[
  \begin{array}{c}
    \nu>0 \Rightarrow m = \link(\eta_\epsilon,L_0) = \wind(\nu,L_0^n,(\beta_{\rm disk})^n) \geq \wind^{\geq 0}(L_0^n,(\beta_{\rm disk})^n), \\
    \nu<0 \Rightarrow m = \link(\eta_\epsilon,L_0) = \wind(\nu,L_0^n,(\beta_{\rm disk})^n) \leq \wind^{< 0}(L_0^n,(\beta_{\rm disk})^n).
  \end{array}
\]
Using Lemma~\ref{inequalities_windings}, there are three possibilities:
\begin{itemize}
  \item $L_0$ is elliptic, $\theta_0(h) \not\in \Q$ and $$ \begin{aligned} & \wind^{\geq 0}(L_0^n,(\beta_{\rm disk})^n) = \lfloor n\theta_0(h) \rfloor + 1 > n\theta_0(h), \\ & \wind^{<0}(L_0^n,(\beta_{\rm disk})^n) = \lfloor n\theta_0(h) \rfloor < n\theta_0(h). \end{aligned} $$
  \item $L_0$ is hyperbolic with positive Floquet multipliers, $\theta_0(h) \in \Z$ and $$ \wind^{\geq 0}(L_0^n,(\beta_{\rm disk})^n) = \wind^{<0}(L_0^n,(\beta_{\rm disk})^n) = n\theta_0(h) $$
  \item $L_0$ is hyperbolic with negative Floquet multipliers, $\theta_0(h) \in \frac{1}{2}\Z$ and $$ \begin{aligned} n \text{ is even} \Rightarrow \wind^{\geq 0}(L_0^n,\beta_{\text{disk}}^n) = \wind^{<0}(L_0^n,\beta_{\text{disk}}^n) = n\theta_0(h) \\ n \text{ is odd} \Rightarrow \left\{ \begin{aligned} & \wind^{\geq 0}(L_0^n,\beta_{\text{disk}}^n) = \lfloor n\theta_0(h) \rfloor + 1 > n\theta_0(h), \\ & \wind^{<0}(L_0^n,\beta_{\text{disk}}^n) = \lfloor n\theta_0(h) \rfloor < n\theta_0(h). \end{aligned} \right. \end{aligned} $$
\end{itemize}
In any case $\nu>0 \Rightarrow m \geq n\theta_0(h)$ and $\nu<0 \Rightarrow m \leq n\theta_0(h)$.
\end{proof}

\subsection{An existence lemma}\label{existence_lemma_section}

Let us fix $f^+,f \in \F$ and $0<c< 1$ such that for every $x\in S^3$ we have $cf^+(x)<f(x)<f^+(x)$. We denote $\lambda^+ = f^+\lambda_0$, $\lambda = f\lambda_0$ and $\lambda^- = c\lambda^+ = cf^+\lambda_0$. Let $\theta_0,\theta_1,\vartheta_0,\vartheta_1$ be defined by
\begin{equation}\label{rotation_model}
  \begin{array}{cc}
    \rho(L_0,\beta_{L_0},\lambda^+) = \rho(L_0,\beta_{L_0},\lambda^-) = 1 + \vartheta_0, & \rho(L_0,\beta_{L_0},\lambda) = 1 + \theta_0, \\
    \rho(L_1,\beta_{L_1},\lambda^+) = \rho(L_1,\beta_{L_1},\lambda^-) = 1 + \vartheta_1, & \rho(L_1,\beta_{L_1},\lambda) = 1 + \theta_1,
  \end{array}
\end{equation}
where we follow the notation established in the beginning of Section~\ref{section_contact_homology}. Here we are considering $L_0,L_1$ as prime closed orbits of the Reeb flows of $\lambda^\pm,\lambda$.

Let $\jhat_\pm \in\J(\lambda_\pm)$, $\jhat \in \J(\lambda)$ be cylindrical almost-complex structures on the symplectization $W_{\xi_0}$ of $(S^3,\xi_0)$, and $J_1 \in \J(\jhat_-,\jhat : K_0)$, $J_2 \in \J(\jhat,\jhat_+ : K_0)$ be special almost-complex structures described in Section~\ref{restricted_class_alm_cpx_str}. Then, for each $R>0$, we consider the splitting almost-complex structure $\bar J_R = J_1\circ_R J_2$ as explained in Section~\ref{splitting_alm_cpx_str}. We denote by $\tau : W_{\xi_0} \to S^3$ the projection onto the base point.

\begin{lemma}\label{existence_lemma}
Suppose $\lambda^+$, $\lambda$, $\lambda^-$ as defined above are non-degenerate contact forms, and let $\theta_0,\vartheta_0,\theta_1,\vartheta_1$ be defined by~\eqref{rotation_model}. Let $R_n \to +\infty$ and $\util_n : \R\times \R/\Z \to W_{\xi_0}$ be finite-energy $\bar J_{R_n}$-holomorphic cylinders satisfying $$ \tau\circ \util_n(\R\times \R/\Z) \cap K_0 = \emptyset, \ \ \forall n, $$ with uniformly bounded energies as defined in Section~\ref{fe_curves_splitting_cob}. Identifying $\R\times \R/\Z \simeq \C P^1\setminus \{[0:1],[1:0]\}$ via $(s,t) \simeq [\est:1]$ assume that $[0:1]$ is a negative puncture and $[1:0]$ is a positive puncture of $\util_n$, $\forall n$. We assume also that all $\util_n$ are asymptotic to fixed orbits $\bar P_+ \in \P(\lambda^+)$, $\bar P_- \in \P(\lambda^-)$ at the positive and negative punctures, respectively, which lie in $S^3\setminus K_0$, and define $p,q \in \Z$ by
\[
  \begin{array}{ccc}
    \link(t\mapsto \tau\circ \util_n(s,t),L_0) = p & \text{and} & \link(t\mapsto \tau\circ \util_n(s,t),L_1) = q
  \end{array}
\]
for every $s$ and $n$. If $(p,q)$ is a relatively prime pair of integers and both conditions A) and B) below are satisfied
\begin{equation}\label{crack_hyp}
\begin{aligned}
& \text{A)} \hspace{.5cm} (q\theta_0 - p)(q\vartheta_0 - p) > 0 \ \text{ or } \ q \leq 0 \\
& \text{B)} \hspace{.5cm} (p\theta_1 - q)(p\vartheta_1 - q) > 0 \ \text{ or } \ p \leq 0
\end{aligned}
\end{equation}
then $\exists P\in\P(\lambda)$ in $S^3\setminus K_0$ such that $\link(P,L_0) = p$ and $\link(P,L_1) = q$.
\end{lemma}

We now turn to the proof of Lemma~\ref{existence_lemma}. The possible limiting behavior of a sequence $\{\util_n\}$ as in the statement is described by the SFT Compactness Theorem from~\cite{sftcomp}. Loosely speaking, it asserts that a space of (equivalence classes of) pseudo-holomorphic maps with \emph{a priori} energy and genus bounds can be compactified by the addition of so-called holomorphic buildings. However, since we deal with cylinders, the possible limiting holomorphic buildings are of a very simple nature, allowing us to avoid introducing all the necessary definitions for precisely stating the SFT Compactness Theorem.

Let us summarize the conclusions we need. Let $R_n$ and $\util_n$ be as in the statement of Lemma~\ref{existence_lemma}. There exists a subsequence $\util_{n_j}$, a collection $\{\Gamma^1,\dots,\Gamma^m\}$ of finite subsets of $\R\times \R/\Z$, a corresponding collection $\{\vtil^1,\dots,\vtil^m\}$ of smooth maps
\begin{equation}\label{}
  \vtil^i :(\R\times \R/\Z) \setminus \Gamma^i \to W_{\xi_0}
\end{equation}
and numbers $1\leq k'<k''\leq m$ ($\Rightarrow m\geq 2$) satisfying the following properties.
\begin{itemize}
\item[(a)] $\vtil^1,\dots,\vtil^{k'-1}$ are $\jhat_+$-holomorphic.
\item[(b)] $\vtil^{k'}$ is $J_2$-holomorphic, $\vtil^{k''}$ is $J_1$-holomorphic.
\item[(c)] $\vtil^{k'+1},\dots,\vtil^{k''-1}$ are $\jhat$-holomorphic.
\item[(d)] $\vtil^{k''+1},\dots,\vtil^m$ are $\jhat_-$-holomorphic.
\item[(e)] $0<E(\vtil^i)<\infty$ for every $i$. All $\vtil^i$ have a positive puncture at the end $\{+\infty\}\times \R/\Z$ of the cylinder, and a negative puncture at $\{-\infty\}\times \R/\Z$. 
\item[(f)] There are Reeb orbits $\bar P_1,\dots\bar P_{k'-1} \in \P(\lambda^+)$, $\bar P_{k'},\dots,\bar P_{k''-1} \in \P(\lambda)$ and $\bar P_{k''},\dots,\bar P_{m-1} \in \P(\lambda^-)$ such that $\bar P_i$ is the asymptotic limit of $\vtil^i$ at $\{-\infty\}\times \R/\Z$ and also the asymptotic limit of $\vtil^{i+1}$ at $\{+\infty\}\times \R/\Z$, for $1\leq i\leq m-1$. $\vtil^1$ is asymptotic to $\bar P_+$ at $\{+\infty\}\times \R/\Z$ and $\vtil^m$ is asymptotic to $\bar P_-$ at $\{-\infty\}\times \R/\Z$.
\item[(g)] For each $i$ there are sequences $\{s^i_j\},\{c^i_j\}$ of real numbers such that the maps $$ (s,t) \mapsto g_{c^i_j} \circ \util_{n_j}(s+s^i_j,t) $$ converge to $\vtil^i$ in $C^\infty_{{\rm loc}}((\R\times \R/\Z) \setminus \Gamma^i)$, as $j\to +\infty$. Here $g_c(\theta) = e^c\theta$ is the $\R$-action on $W_{\xi_0}$.
\end{itemize}

For simplicity of notation we set $\dot Z_i = (\R\times \R/\Z) \setminus \Gamma^i$ and $C_i = \vtil^i(\dot Z^i) \subset W_{\xi_0}$. Note that $\tau^{-1}(K_0)$ is an embedded surface and its tangent space is invariant by all almost-complex structures $\jhat_\pm$, $\jhat$, $J_1$, $J_2$ and $\bar J_R$.

We will now show that
\begin{equation}\label{Pi_not_K0}
  \bar P_i \cap K_0 = \emptyset, \ \forall i.
\end{equation}
Arguing indirectly, assume that $\bar P_i \subset K_0$ for some $i$, and let
\[
  i_0 = \min \{ i\in \{1,\dots,m-1 \} \mid \bar P_i \subset K_0 \}.
\]
For each $1\leq i\leq i_0$ note that $C_i \not\subset \tau^{-1}(K_0)$, and consider the set
\begin{equation}\label{set_Di}
  D_i = \{ (z,x) \in \dot Z_i \times \tau^{-1}(K_0) \mid \vtil^i(z) = x \} \ \ \ (1\leq i\leq i_0).
\end{equation}
Clearly $D_i$ is closed in $\dot Z_i\times \tau^{-1}(K_0)$. If $D_i$ accumulates in a point of $\dot Z_i\times \tau^{-1}(K_0)$ then one could use Carleman's similarity principle to conclude that $C_i \subset \tau^{-1}(K_0)$. This would imply $\bar P_{i-1} \subset K_0$, a contradiction to the definition of $i_0$. Thus $D_i$ is discrete and if $D_i \neq\emptyset$ then we get isolated intersections of the pseudo-holomorphic map $\vtil^i$ with the embedded surface $\tau^{-1}(K_0)$. By positivity and stability of intersections we get intersections of the image of the maps $\util_{n_j}$ with $\tau^{-1}(K_0)$, for $j$ large, contradicting our hypotheses. We showed $C_i \cap \tau^{-1}(K_0) = \emptyset$ for all $1\leq i\leq i_0$.

Either $\bar P_{i_0} \subset L_0$ or $\bar P_{i_0} \subset L_1$. We assume $\bar P_{i_0} \subset L_0$, the other case is entirely analogous. Thus $\bar P_{i_0} = L_0^m$ for some $m\geq1$. Since $\bar P_{i_0}$ is the asymptotic limit of $\vtil_{i_0}$ at the negative end $\{-\infty\} \times \R/\Z$, $\bar P_{i_0}$ can be approximated in $C^\infty$ by curves of the form $t\mapsto \tau\circ \util_{n_j}(s_j,t)$ with suitable values $s_j$. In particular, $\bar P_{i_0}$ is homotopic to $\bar P_+$ in $S^3\setminus L_1$ and $m = \link(\bar P_{i_0},L_1) = \link(\bar P_+,L_1) = q$, which implies $q\geq 1$. In view of~\eqref{crack_hyp} we can assume $(q\theta_0-p)(q\vartheta_0-p)>0$.

Let us set
\[
  i_1 = \max \{ i\in \{1,\dots,m-1\} \mid \bar P_i = L_0^q \} \geq i_0.
\]
Then $\vtil^{i_0}$ is asymptotic to $L_0^q$ at $\{-\infty\}\times \R/\Z$ and $\vtil^{i_1+1}$ is asymptotic to $L_0^q$ at $\{+\infty\}\times \R/\Z$.

We claim that $C_{i_1+1}$ is not contained in $\tau^{-1}(K_0)$. In fact, if $C_{i_1+1} \subset \tau^{-1}(K_0)$ then $C_{i_1+1} \subset \tau^{-1}(L_0)$ and $\vtil^{i_1+1}$ is asymptotic to $L_0^r$ at $\{-\infty\}\times \R/\Z$, for some $r\neq q$. Here we used the definition of $i_1$. If $j$ is large enough and $s\gg 1$ then $t\mapsto \tau\circ \util_{n_j}(s+s^{i_1+1}_j,t)$ is close to $L_0^q$, and if $s\ll -1$ then $t\mapsto \tau\circ \util_{n_j}(s+s^{i_1+1}_j,t)$ is close to $L_0^r$. However, the image of the cylinders $\tau\circ \util_{n_j}$ are contained in $S^3\setminus K_0 \subset S^3\setminus L_1$, showing that $L_0^q$ is homotopic to $L_0^r$ in $S^3\setminus L_1$. This is a contradiction to $q\neq r$.

We set $\alpha_0 = \lambda^+$ if $i_0<k'$, $\alpha_0 = \lambda$ if $k'\leq i_0<k''$ or $\alpha_0 = \lambda^-$ if $k'' \leq i_0$. Also, we set $\alpha_1 = \lambda^+$ if $i_1<k'$, $\alpha_1 = \lambda$ if $k'\leq i_1<k''$ or $\alpha_1 = \lambda^-$ if $k'' \leq i_1$. Since $f,f^+ \in \F$, the Reeb vector fields $X_{\alpha_0}$ and $X_{\alpha_1}$ are pointwise positive multiples of the Reeb vector field $X_{\lambda_0}$ of $\lambda_0$ on $L_0$, where $\lambda_0$ is the standard Liouville form on $S^3$. $L_0$ is a periodic trajectory of $X_{\alpha_0}$ or of $X_{\alpha_1}$, and we write $L_0 = (x_{\alpha_0},T_{\alpha_0})$ or $L_0 = (x_{\alpha_1},T_{\alpha_1})$ depending on whether we see it as a $\alpha_0$-Reeb orbit or as a $\alpha_1$-Reeb orbit ($T_{\alpha_0}$ and $T_{\alpha_1}$ are minimal periods). For simplicity we denote $\gamma_0(t) = x_{\alpha_0}(T_{\alpha_0}t)$ and $\gamma_1(t) = x_{\alpha_1}(T_{\alpha_1}t)$.

Let $(U_0,\Phi_0)$ and $(U_1,\Phi_1)$ be a Martinet tubes for the contact forms $\alpha_0$ and $\alpha_1$ at $L_0$, respectively, as explained in Definition~\ref{martinet_tube_def}. That is, for $l=0,1$, $U_l$ is a tubular neighborhood of $L_0$ and $\Phi_l: U_l \to \R/\Z\times B$ is a diffeomorphism, where $B \subset \R^2$ is a small ball centered at the origin, such that $\Phi_l(\gamma_l(t)) = (t,0,0)$ and $(\Phi_l)_*\alpha_l = F_l(d\theta+xdy)$. Here $F_l: \R/\Z\times B\to \R^+$ satisfies $F_l|_{\R/\Z\times \{(0,0)\}} \equiv T_{\alpha_l}$ and $dF_l|_{\R/\Z\times \{(0,0)\}} \equiv 0$, and the usual coordinates on $\R/\Z\times \R^2$ are denoted by $(\theta,x,y)$. For $l=0,1$ we have sections $Y_l(t) = d\Phi_l^{-1} \cdot \partial_x|_{(t,0,0)}$ of the bundle $\gamma_l^*\xi_0$, and we assume $\Phi_0,\Phi_1$ were chosen so that the loops $t\mapsto \exp(\epsilon Y_l(t))$ ($\epsilon>0$ small) have linking number $0$ with $L_0$. Then $Y_l$ can be completed to a $d\alpha_l$-symplectic frame of $\gamma_l^*\xi_0$ in certain homotopy classes $\beta_l$ such that
\begin{equation}\label{}
  \begin{array}{ccc}
    \theta_0 = \rho(L_0,\beta_0,\lambda) & \text{and} & \vartheta_0 = \rho(L_0,\beta_1,\lambda^\pm).
  \end{array}
\end{equation}
Here we used that $L_0$ has self-linking number $-1$.

Since $\vtil^{i_0}$ is asymptotic to $L_0^q$ at its negative end $\{-\infty\}\times\R/\Z$, and $\vtil^{i_1+1}$ is asymptotic to $L_0^q$ at its positive end $\{+\infty\}\times\R/\Z$, there exists $s_0\ll -1$ such that $\tau\circ \vtil^{i_0}(s,t) \in U_0$ when $s\leq s_0$ and $\tau\circ \vtil^{i_1+1}(s,t) \in U_1$ when $s\geq -s_0$.

The behavior as $s\to -\infty$ of the functions $$ (\theta^0(s,t),x^0(s,t),y^0(s,t)) = \Phi_0 \circ \tau\circ \vtil^{i_0}(s,t)  \ \ \ (\text{defined for } s\leq s_0) $$ is well understood in view of Theorem~\ref{precise_asymptotics}. The function $\theta^0(s,t)$ satisfies
\begin{equation}
  \theta^0(s,t) \to qt + t_0 \ \text{ uniformly in } t \text{ as } s\to -\infty, \ \text{ for some } t_0.
\end{equation}
To describe the behavior of $x^0, y^0$ let us consider the asymptotic operator $A_0$ associated to the orbit $L_0^q$ of the contact form $\alpha_0$, as explained in Section~\ref{asymptotic_behavior_section}. There is an eigenvalue $\nu^0 \in \sigma(A_0) \cap (0,+\infty)$ of $A_0$ and a section $$ t\in \R/\Z \mapsto \eta^0(t) \in \xi_0|_{\gamma_0(qt+t_0)} $$ in the eigenspace of $\nu^0$ such that the following holds. If $w^0(t) : \R/\Z\to \R^2\setminus 0$ is the representation of $\eta^0(t)$ in the local frame $\{\partial_x , \partial_y\}$ induced by $\Phi_0$ then, perhaps after making $s_0$ more negative, we can write
\begin{equation}\label{xy_coordinates_i}
  (x^0(s,t),y^0(s,t)) = e^{\int_{s_0}^s\alpha^0(r)dr} (w^0(t)+R(s,t))  \ \ \forall s\leq s_0
\end{equation}
where $|R(s,t)| \to 0$ uniformly in $t$ as $s\to -\infty$ and $\alpha^0(r) \to \nu^0$ as $r\to-\infty$. The behavior of the functions
\[
  (\theta^{1}(s,t),x^{1}(s,t),y^1(s,t)) = \Phi_1 \circ \tau\circ \vtil^{i_1+1}(s,t) \ \ \ (\text{defined for } s\geq -s_0)
\]
is entirely analogous. More precisely, let $A_1$ be the asymptotic operator associated to $L_0^q$ seen as an $\alpha_1$-Reeb orbit. Then
\begin{equation}
  \theta^1(s,t) \to qt + t'_0 \ \text{ uniformly in } t \text{ as } s\to +\infty, \ \text{ for some } t'_0,
\end{equation}
and we find an eigenvector $\nu^1 \in \sigma(A_1) \cap (-\infty,0)$ and an eigensection $$ t \in \R/\Z \mapsto \eta^1(t) \in \xi_0|_{\gamma_1(qt+t_0')} $$ for $\nu^1$ such that the following holds. If $w^1: \R/\Z \to \R^2\setminus 0$ is the representation of $\eta^1(t)$ in the local frame $\{\partial_x , \partial_y\}$ induced by $\Phi_1$ then
\begin{equation}\label{xy_coordinates_i+1}
  (x^1(s,t),y^1(s,t)) = e^{\int_{-s_0}^s\alpha^1(r)dr} (w^1(t)+ \tilde R(s,t))  \ \ \forall s\geq -s_0
\end{equation}
where $|\tilde R(s,t)| \to0$ uniformly in $t$ as $s\to +\infty$ and $\alpha^1(r) \to \nu^1$ as $r\to+\infty$.

Now we consider, as we did in Section~\ref{computing_homotopy}, the curves
\[
  \begin{array}{ccc}
    \eta^0_\epsilon(t) = \exp_{\gamma_0(qt+t_0)}(\epsilon \eta^0(t)) & \text{and} & \eta^1_\epsilon(t) = \exp_{\gamma_1(qt+t_0')}(\epsilon \eta^1(t))
  \end{array}
\]
for $\epsilon>0$ small and set $p_0 = \link(\eta^0_\epsilon,L_0)$, $p_1 = \link(\eta^1_\epsilon,L_0)$. Clearly $q = \link(\eta^l_\epsilon,L_1)$ for $l=0,1$ since $\eta^0_\epsilon$, $\eta^1_\epsilon$ are loops close to $L_0^q$. Lemma~\ref{estimating_homotopy_lemma} implies
\begin{equation}\label{estimates_evectors}
  \begin{array}{cc}
    i_0<k' \text{ or } k''\leq i_0 \Rightarrow p_0 \geq q\vartheta_0, & k'\leq i_0<k'' \Rightarrow p_0\geq q\theta_0
  \end{array}
\end{equation}
and
\begin{equation}\label{estimates_evectors_2}
  \begin{array}{cc}
    i_1<k' \text{ or } k'' \leq i_1 \Rightarrow p_1 \leq q\vartheta_0, & k'\leq i_1<k'' \Rightarrow p_1\leq q\theta_0.
  \end{array}
\end{equation}
We use~\eqref{xy_coordinates_i} and~\eqref{xy_coordinates_i+1} to find numbers $s' \ll -1$ and $s'' \gg 1$ such that the curve
\[
  t\mapsto \tau\circ \vtil^{i_0}(s',t) = \Phi_0^{-1}(\theta^0(s',t),x^0(s',t),y^0(s',t))
\]
is homotopic to $\eta^0_\epsilon(t)$ in $U_0 \setminus L_0 \subset S^3\setminus K_0$, and the curve
\[
  t\mapsto \tau\circ \vtil^{i_1+1}(s'',t) = \Phi_1^{-1}(\theta^1(s'',t),x^1(s'',t),y^1(s'',t))
\]
is homotopic to $\eta^1_\epsilon(t)$ in $U_1 \setminus L_0 \subset S^3\setminus K_0$. In view of item (g) described above there are sequences $s'_j,s''_j \in \R$ such that
\[
  \begin{array}{ccc}
    \tau\circ \util_{n_j}(s'_j,t) \to \tau\circ \vtil^{i_0}(s',t) & \text{and} & \tau\circ \util_{n_j}(s''_j,t) \to \tau\circ \vtil^{i_1+1}(s'',t)
  \end{array}
\]
in $C^\infty(\R/\Z,S^3)$ as $j \to \infty$. Taking $j$ large enough
\begin{equation}\label{}
  \begin{aligned}
    & p = \link(t\mapsto \tau\circ \util_{n_j}(s'_j,t),L_0) = \link(\tau\circ \vtil^{i_0}(s',t),L_0) = p_0 \\
    & p = \link(t\mapsto \tau\circ \util_{n_j}(s''_j,t),L_0) = \link(\tau\circ \vtil^{i_1+1}(s'',t),L_0) = p_1.
  \end{aligned}
\end{equation}
Then~\eqref{estimates_evectors} implies
\begin{equation}\label{boolean_1}
\begin{array}{ccc}
p\geq q\theta_0 & \text{or} & p\geq q\vartheta_0
\end{array}
\end{equation}
and~\eqref{estimates_evectors_2} implies
\begin{equation}\label{boolean_2}
\begin{array}{ccc}
p\leq q\theta_0 & \text{or}&  p\leq q\vartheta_0.
\end{array}
\end{equation}
Putting together~\eqref{boolean_1} with~\eqref{boolean_2} we have
\begin{equation}\label{boolean_3}
\begin{aligned}
& p\geq q\theta_0 \ \text{and} \ p\leq q\theta_0, \ \ \ \text{or} \\
& p\geq q\theta_0 \ \text{and} \ p\leq q\vartheta_0, \ \ \ \text{or} \\
& p\geq q\vartheta_0 \ \text{and} \ p\leq q\theta_0, \ \ \ \text{or} \\
& p\geq q\vartheta_0 \ \text{and} \ p\leq q\vartheta_0.
\end{aligned}
\end{equation}
In all four cases we get a contradiction to $(q\theta_0-p)(q\vartheta_0-p)>0$. As remarked before, the argument assuming $\bar P_{i_0} \subset L_1$ is analogous. Thus~\eqref{Pi_not_K0} is proved. At this point the condition~\eqref{crack_hyp} plays its role.

We showed that every $\bar P_i$ lies in $S^3\setminus K_0$ and it follows that they are homotopic to $\bar P_+$ away from $K_0$. Consequently, $\bar P_{k'} \in \P(\lambda)$ is the desired orbit satisfying $\link(\bar P_{k'},L_0) = p$ and $\link(\bar P_{k'},L_1) = q$. The proof of Lemma~\ref{existence_lemma} is complete.

\subsection{Proof of Proposition~\ref{prop-non-degenerate_case}}

First we consider the case of a constant sequence: $f_n = f \ \forall n$, for some $f \in \mathcal{F}$ such that $f\lambda_0$ is non-degenerate. Let $\theta_0 = \theta_0(f),\ \theta_1 = \theta_1(f)$ be the associated rotation numbers. Select a model $f_{\theta'_0,\theta'_1}$ with $\theta'_0,\theta'_1 \not\in \Q$, as described in Section~\ref{example-2}. We choose $\theta_0'$ and $\theta_1'$ close enough to $\theta_0$ and $\theta_1$, respectively, in such a way that
\begin{itemize}
  \item if $(\theta_0,1) < (p,q) < (1,\theta_1)$ then $(\theta'_0,1) < (p,q) < (1,\theta'_1)$,
  \item if $(1,\theta_1) < (p,q) < (\theta_0,1)$ then $(1,\theta'_1) < (p,q) < (\theta'_0,1)$.
\end{itemize}
By rescaling $f_{\theta'_0,\theta'_1}$ if necessary, we may assume that $f_{\theta'_0,\theta'_1} > f$ pointwise. There exists a small constant $0<c<1$ such that $f > cf_{\theta'_0,\theta'_1}$ pointwise as well. Using Proposition~\ref{prop_morse_bott_approx} we may find $f_+ \in \F$ arbitrarily close to $f_{\theta'_0,\theta'_1}$, a suitable $J_+ \in \J_+(\xi_0)$ and some $T>0$ such that $\jhat_+ \in \J_{\rm reg}(f_+\lambda_0)$, the chain complexes
\[
  \begin{array}{ccc}
    (C^{\leq T/c,(p,q)}_*(f_+\lambda_0),\partial_{(f_+\lambda_0,J_+)}) & \text{ and } & (C^{\leq T,(p,q)}_*(f_+\lambda_0),\partial_{(f_+\lambda_0,J_+)})
  \end{array}
\]
are well-defined and their homologies equal the homology of $S^1$ over $\mathbb{F}_2$, up to a common shift in degree. These chain complexes are generated by the same orbits and their differentials count the same holomorphic cylinders. We may assume that $cf_+ < f < f_+$ holds pointwise as well. Recall that $f_+$ coincides with $f_{\theta'_0,\theta'_1}$ near $K_0$, so that $\theta_0(f_+) = \theta'_0$ and $\theta_1(f_+) = \theta'_1$.

Consider the almost complex structure $\jhat_- \in \J(cf_+\lambda_0)$ induced by $J_+$ and the contact form $cf_+\lambda_0$. Then, as explained in Section~\ref{inclusions}, we have $\jhat_- \in \J_{\rm reg}(cf_+\lambda_0)$ which can be used to define the differential of the chain complex $$ (C^{\leq T,(p,q)}_*(cf_+\lambda_0),\partial_{(cf_+\lambda_0,J_+)}). $$ By Proposition~\ref{comp_j_iota} the map
\[
  (C_{*}^{\leq T, (p,q)}(f_+\lambda_0),\partial_{(f_+\lambda_0,J_+)}) \overset{j_* \circ \iota_*}{\longrightarrow} (C_*^{\leq T, (p,q)}(cf_+\lambda_0),\partial_{(cf_+\lambda_0,J_+)})
\]
is non-trivial in homology. Here $j_*$ is the map~\eqref{map_j_*} and $\iota_*$ is the map~\eqref{map_iota_*}.

Let us select $\jhat \in \J(f\lambda_0)$ and consider almost complex structures $\bar J_R = \bar J_1 \circ_R \bar J_2$, where $\bar J_2 \in \J(\jhat,\jhat_+:K_0)$, $\bar J_1 \in \J(\jhat_-,\jhat:K_0)$ and $R>0$. As explained in Section~\ref{splitting_alm_cpx_str}, the almost complex structure $\bar J_R$ is biholomorphic to some $J'_R\in \J(\jhat_-,\jhat_+:K_0)$. We claim that there is a finite-energy $J'_R$-holomorphic cylinder asymptotic to orbits in $\P^{\leq T,(p,q)}(f_+\lambda_0)$ and $\P^{\leq T,(p,q)}(cf_+\lambda_0)$ at the positive and negative punctures, respectively, which do not intersect $\tau^{-1}(K_0)$. Arguing indirectly, if there are none we conclude that $J'_R \in \J_{\rm reg}(\jhat_-,\jhat_+:K_0)$ and, therefore, the map $\Phi(J'_R)_*$ as in~\eqref{chain_map_phi_j_bar} is well-defined and is equal to zero. Lemma~\ref{j_iota_relation_chain} implies that $\Phi(J'_R)_*$ is chain-homotopic to $j_*\circ\iota_*$ and, thus, non-trivial by Proposition~\ref{comp_j_iota}. This contradiction proves our claim.

We found, for every $R>0$, finite-energy $\bar J_R$-holomorphic cylinders asymptotic to Reeb orbits in $\P^{\leq T,(p,q)}(f_+\lambda_0)$ and $\P^{\leq T,(p,q)}(cf_+\lambda_0)$ at the positive and negative punctures, respectively, not intersecting $\tau^{-1}(K_0)$. Since there are only two orbits in each set $\P^{\leq T,(p,q)}(f_+\lambda_0)$ and $\P^{\leq T,(p,q)}(cf_+\lambda_0)$, there is a sequence of $\bar J_{R_n}$-cylinders with the same asymptotic limits for every $n$, where $R_n \to +\infty$. This implies that the energies of these cylinders are uniformly bounded in $n$, see Remark~\ref{exact_nature}. If $\theta_0'$ and $\theta_1'$ are chosen sufficiently close to $\theta_0$ and $\theta_1$, respectively, then conditions A) and B) in~\eqref{crack_hyp} are both satisfied (replacing $\vartheta_i$ by $\theta_i'$). Applying Lemma~\ref{existence_lemma} to these cylinders we find an element of $\P^{\leq T, (p,q)}(f \lambda_0)$, which proves the first assertion of Proposition~\ref{prop-non-degenerate_case}.

To prove the result for families $f_n$ as in the statement, observe that we may select a single model $f_{\theta'_0,\theta'_1}$, with $\theta'_0,\theta'_1 \not\in\Q$, and a constant $c > 0$ such that:
\begin{itemize}
\item if $(\theta_0,1) < (p,q) < (1,\theta_1)$ then $(\theta'_0,1) < (p,q) < (1,\theta'_1)$,
\item if $(1,\theta_1) < (p,q) < (\theta_0,1)$ then $(1,\theta'_1) < (p,q) < (\theta'_0,1)$,
\item $\inf_x f_{\theta'_0,\theta'_1}(x) > M$, $\sup_x cf_{\theta_0',\theta_1'}(x) < m$.
\end{itemize}
The assumptions on $f_n$ guarantee that $c f_{\theta'_0,\theta'_1} < f_n < f_{\theta'_0,\theta'_1}$ pointwise, for each $n$. Since $\theta_i(f_n) \to \theta_i$ we have for large $n$
\begin{itemize}
\item if $(\theta_0,1) < (p,q) < (1,\theta_1)$ then $(\theta_0(f_n),1) < (p,q) < (1,\theta_1(f_n))$,
\item if $(1,\theta_1) < (p,q) < (\theta_0,1)$ then $(1,\theta_1(f_n)) < (p,q) < (\theta_0(f_n),1)$.
\end{itemize}

We may assume that $\theta_0'$ and $\theta_1'$ are chosen close enough to $\theta_0$ and $\theta_1$, respectively, in such a way that both conditions A) and B) in~\eqref{crack_hyp} are satisfied, replacing $(\vartheta_i,\theta_i)$ in~\eqref{crack_hyp} by $(\theta_i',\theta_i(f_n))$ for $n$ large. Applying the above argument to each form $f_n\lambda_0$ with these specific choices of $f_{\theta'_0,\theta'_1}$ and $c f_{\theta'_0,\theta'_1}$, we obtain, for all $n$ sufficiently large, an orbit in $\mathcal{P}^{\leq T, (p,q)}(f_n \lambda_0)$, where $T$ is some upper bound on the action independent of $n$ large. In fact, $T$ could be any number larger than the action of an orbit in the $(p,q)$-orbit torus for $f_{\theta'_0,\theta'_1}$ and is therefore independent of $n$.

\section{Passing to the degenerate case}\label{deg_section}

\subsection{Non-degenerate approximations}

\begin{lemma}\label{lemm1}
Let $f \in \mathcal{F}$.  There is a sequence $f_n \in \mathcal{F}$ such that $f_n \lambda_0$ is non-degenerate for each $n$, $f_n \rightarrow f$ in $C^{\infty}$. In particular, for $i = 0,1$ we have $\theta_i(f_n) \rightarrow \theta_i(f)$ as $n \rightarrow \infty$, where $\theta_i(f_n) = \rho(f_n\lambda_0,L_i)-1$.
\end{lemma}

\begin{proof}
It is possible to find $f' \in \F$ arbitrarily $C^\infty$-close to $f$ such that $L_0$, $L_1$ are non-degenerate prime Reeb orbits of $f'\lambda_0$; see~\cite[Lemma 6.8]{convex}. Now there exists $f''$ $C^\infty$-close to $f'$ such that $f''\lambda_0$ is non-degenerate, but at this step $f''$ need not belong to $\F$. However, the orbits $L_0$, $L_1$ get perturbed to closed $f''\lambda_0$-Reeb orbits $L'_0$, $L'_1$ in a way that $L'_i$ is $C^\infty$-close to $L_i$, $i=0,1$. Here we used that $L_i$ were non-degenerate orbits of $f'\lambda_0$. We take a $C^\infty$-small contact isotopy $\{\varphi_t\}_{t\in[0,1]}$ of $(S^3,\xi_0)$ satisfying $\varphi_0=id$, $\varphi_1(L_i) = L'_i$, $i=0,1$. Then $\varphi_1^*(f''\lambda_0) = f'''\lambda_0$ for some $f''' \in \F$ which is $C^\infty$-close to the original $f\lambda_0$. Thus $f'''\lambda_0$ is non-degenerate and $\theta_i(f''') \sim \theta_i(f)$, $i=0,1$.
\end{proof}

Let us select a sequence $f_n$ as in Lemma~\ref{lemm1}.  By Proposition~\ref{prop-non-degenerate_case}, for each pair $(p,q)\in\Z\times\Z$ relatively prime satisfying
\[
(\theta_{0}(f_n),1) < (p,q) < (1,\theta_{1}(f_n)), \mbox{ or } (1,\theta_{1}(f_n),1) < (p,q) < (\theta_{0}(f_n),1)
\]
there is a $(p,q)$-closed Reeb orbit for $f_n \lambda_0$, for large $n$, which we shall denote $P_n(p,q)$. Indeed, since $\theta_i(f_n) \rightarrow \theta_i(f)$, if $(p,q)$ satisfies one of the above inequalities for $\theta_0(f), \theta_1(f)$ then for all $n$ large enough the same inequality holds for $f_n$, and therefore the orbit $P_n(p,q)$ is obtained by Proposition~\ref{prop-non-degenerate_case}.  Moreover, since $f_n \rightarrow f$ it is clear that there are constants $m, M$ such that
\[
0 < m < \inf_{n,x} f_n(x) \leq \sup_{n,x} f_n(x) < M.
\]
Therefore the second assertion of Proposition \ref{prop-non-degenerate_case} applies to the sequence $f_n$ and guarantees that we may assume a uniform bound
\[
\int_{P_n(p,q)} f_n\lambda_0 \leq T
\]
for some $T$ independent of $n$ and all $n$ large. Using this period bound, the Arzela-Ascoli theorem guarantees that there exists a subsequence $n_i$ such that $P_{n_i}(p,q) \rightarrow P(p,q)$ in $C^{\infty}(S^1, S^3)$, where $P(p,q)$ is a closed Reeb orbit for $f \lambda_0$.  If $P(p,q)$ does not have image contained in $K_0$, it is clear that it is in the homotopy class $(p,q)$: for, $P_n(p,q)$ is $C^{\infty}$-close to $P(p,q)$, which implies that for all large $n$ the homotopy classes of $P_n(p,q)$, $P(p,q)$ in $S^3 \backslash K_0$ must be the same.  However, at this point it is conceivable that $P(p,q)$ has image in $K_0$; we show next that this cannot be the case.

\subsection{Non-collapsing}

Let us suppose that the sequence of orbits $P_n = P_n(p,q)$ converges to $L_0^q$, $q>0$  (otherwise, the sequence must converge to $L_1^p$ and the argument is analogous). This fact together with~\eqref{non_resonance} implies that $(\theta_0(f),1) \not\in \R(p,q)$, which also implies that $\theta_0(f)\neq p/q$.

Let $(U,\Phi)$ be a Martinet tube for $L_0$, so that we have special coordinates $(\theta,x,y) \in \R/\Z\times B$ on $U$ with respect to which $f_n\lambda_0 \simeq g_n(d\theta+xdy)$ and $f\lambda_0 \simeq g(d\theta+xdy)$. Moreover on $\R/\Z\times 0$ we have $\xi_0 \simeq 0\times \R^2$. In particular $\sigma = \{\partial_x,\partial_y\}$ is a conformal $d(f_n\lambda_0)$-symplectic frame of $\xi_0|_{L_0}$, for every $n$. We may assume, without loss of generality, that given an embedded disk $D_0$ spanning $L_0$ then $\partial_x$ is an outward pointing tangent vector of $D_0$ along $L_0$. Denoting by $\beta_{\rm disk}$ the homotopy class of trivializations of $\xi_0|_{L_0}$ induced by $\sigma$ then $\rho(f_n\lambda_0,L_0,\beta_{\rm disk}) = \theta_0(f_n)$ since $L_0$ has self-linking number $-1$.

Let $\phi_t$ denote the Reeb flow of $f\lambda_0$. Given $N>0$, we may find a smaller neighborhood $U(N)$ so that for $w \in U(N)$, the trajectory $\phi_t(w)$ for $0 \leq t \leq N$ lies in $U$. Let $\phi_{n,t}$ denote the Reeb flow for $f_n \lambda_0$.  For $N$ fixed, if $n$ is sufficiently large then $\phi_{n,t}(w)\in U$ for $w \in U(N)$ and $0 \leq t \leq N$.  We have $\phi_{n} \rightarrow \phi$ in $C^{\infty}([0,N] \times U(N),U)$.

Denote by $T$ the prime $f\lambda_0$-period of $L_0$, and by $T_n$ the $f_n\lambda_0$-period of $P_n$. Recall that $P_n \rightarrow P = L_0^q$. Choose $N > qT + 1$, say, and note that by hypothesis for $n$ large we have $P_n \subset U(N)$. Let $(0,w_n)$ be in the intersection of $P_n$ with the disc $0 \times B$ with respect to the coordinate system above. Note that $w_n \rightarrow 0$.  After passing to a subsequence, we may suppose that $w_n / \|w_n\| \rightarrow h \neq 0 \in \R^2$.

We claim that $d\phi_{qT}(0,0)(0,h) = (0,h)$ and that the winding of the vector $d\phi_t(0,0)(0,h) \in 0\times (\R^2\setminus 0)$, $t\in[0,qT]$, is exactly $p$ around the origin in the plane $0\times \R^2$. This leads to $\theta_0(f)=p/q$, which is a contradiction since $\theta_0(f) \neq p/q$.

In fact, let $\Pi: \R/\Z \times \R^2 \to \R^2$ be the projection on the second factor inducing the tangent projection $\Pi_*$. We use the Euclidean norm in all linear spaces whenever it is necessary.

We have $(0,w_n) = \phi_{n,T_n}(0,w_n)$ for each $n$, where $T_n \rightarrow q T$. Note the following estimates:

\begin{itemize}
\item $\| \phi_{n,t}(0,w_n) - \phi_{n,t}(0,0)-d\phi_{n,t}(0,0) \cdot (0,w_n) \| \leq C_n \| w_n \|^2$,  $\forall t\in[0,N]$, where $C_n$ is the sup-norm of the second derivative of $\phi_{n,t},t\in[0,N],$ in $U(N)$. For $n$ large, $C_n$ can be bounded by twice the sup-norm of the second derivative of $\phi_t,t\in[0,N],$ in $U(N)$ (i.e. $C_n$ is uniformly bounded by an absolute constant $K_1$ for $n$ large). Projecting onto $\R^2$ and using that $\phi_{n,t}(0,0) \in \R/\Z \times \{0\}$ for all $t$, we get
\begin{equation}\label{collapsing-estimate1}
\|\Pi(\phi_{n,t}(0,w_n)) - \Pi_*d\phi_{n,t}(0,0) \cdot (0,w_n)\| \leq K_1 \|w_n\|^2
\end{equation}
for all $t\in[0,N]$ and $n$ large.
\item We have
\begin{equation}\label{collapsing-estimate2}
\| \Pi_*d\phi_{n,t}(0,0)\cdot (0,w_n) - \Pi_*d\phi_{t}(0,0) \cdot (0,w_n) \| \leq \epsilon^{(1)}_n \| w_n \|, \forall t\in[0,N],
\end{equation}
where $\epsilon^{(1)}_n \downarrow 0$ can be taken to be the supremum of the sup-norm of the difference $d \phi_{n,t}(0,0) - d\phi_t(0,0)$ in $t\in[0,N]$.
\item We have
\begin{equation}\label{collapsing-estimate3}
\left\| \Pi_*d\phi_{t}(0,0)\cdot (0,w_n) - \{\|w_n\| \Pi_*d\phi_{t}(0,0) \cdot (0,h)\} \right\| \leq \epsilon^{(2)}_n \| w_n \|,
\end{equation}
for some $\epsilon^{(2)}_n \downarrow 0$.  To see this, let $\|(w_n/\|w_n\|) - h\| = \epsilon^{(3)}_n \rightarrow 0$, so $\|w_n - (\|w_n\| h) \| = \|w_n\| \cdot \epsilon^{(3)}_n$ and we can take the constant $\epsilon^{(2)}_n = \epsilon^{(3)}_n \cdot K_2 \downarrow 0$ where $K_2$ is the supremum of the sup-norm of $d\phi_t(0,0)$ in $t\in[0,N]$ .
\item Putting \eqref{collapsing-estimate1}, \eqref{collapsing-estimate2}, \eqref{collapsing-estimate3} together and dividing by $\|w_n\| \neq 0$, we have for $t\in[0,N]$ and $n$ large
\begin{equation}\label{desig1}
\left\| \frac{\Pi(\phi_{n,t}(0,w_n))}{\|w_n\|} -  \Pi_*d\phi_t(0,0) (0,h) \right\| \leq \left( K_1 \| w_n \| + \epsilon^{(1)}_n + \epsilon^{(2)}_n \right) \downarrow 0.
\end{equation}
\end{itemize}
Since $\phi_{n,T_n}(0,w_n)=(0,w_n)$ we find for $t=T_n$
\[
\left\| \left(0,\frac{w_n}{\|w_n\|}\right) - d\phi_{T_n}(0,0) (0,h) \right\| \to 0 \mbox{ as } n\to \infty,
\]
which implies, since $T_n \to qT$, that $d\phi_{qT}(0,h) = (0,h)$.

Again from~\eqref{desig1} and taking $K_1 \| w_n \| + \epsilon^{(1)}_n + \epsilon^{(2)}_n$ less than the infimum of $\|d\phi_t(0,0) (0,h)\|$ over $t\in[0,N]$, we necessarily have for $n$ large that
\[
  \begin{aligned}
    \mathrm{wind}_{t\in[0,T_n]} \left( \Pi(\phi_{n,t}(0,w_n)) \right) &= \mathrm{wind}_{t\in[0,T_n]}\left( \frac{\Pi ( \phi_{n,t}(0,w_n))}{\|w_n\|} \right) \\
    &= \mathrm{wind}_{t\in[0,qT]}(\Pi_*d\phi_t(0,0) (0,h)).
  \end{aligned}
\]
Here, obviously, we are computing the winding numbers with respect to the standard basis of $\R^2$.

Now since $P_n$ is a Reeb orbit for $f_n\lambda_0$ in the homotopy class $(p,q)$ with respect to the Hopf Link $K_0 = L_0 \cup L_1$, we also have that
\[
  \mathrm{wind}_{t\in[0,T_n]}(\Pi ( \phi_{n,t}(0,w_n) )) = p.
\]
Here we make use of our particular choice of Martinet tube. We conclude that $\theta_0(f) = \mathrm{wind}_{t\in[0,qT]}(\Pi_*d\phi_t(0,0) (0,h))/q = p/q$ which, as mentioned before, gives a contradiction. It follows that the orbits $P_n$ cannot converge to $L_0^q$.

Arguing similarly, they cannot converge to $L_1^p$ either, and thus the sequence $P_n$ has a limit $P$ in $S^3 \backslash K_0$ which is an orbit for $f\lambda_0$ satifying $\link(P,L_0)=p$, $\link(P,L_1)=q$. This completes the proof of Theorem~\ref{thm-1}.

\section{The $SO(3)$-case}\label{sec-SO3}

Our goal in this section is to prove Theorem~\ref{teot} and its corollaries.

\subsection{Geometric set-up}

Let $g_0$ be the Euclidean metric on $\mathbb{R}^3$ and consider the associated unit tangent bundle of the $2$-sphere $$ T^1S^2 := \{ (x,v)\in \mathbb{R}^3 \times \mathbb{R}^3 \mid g_0(x,x)=g_0(v,v)=1 \mbox{, } g_0(x,v)=0\}\simeq SO(3), $$ with bundle projection $\Pi:T^1S^2 \to S^2$, $\Pi(x,v) = x$. There exists a double covering map $$ D:S^3 \stackrel{2:1}{\longrightarrow} T^1S^2, $$ explicitly given by
\begin{equation}
\left[ \begin{array}{c} q_0 \\ p_0 \\ q_1 \\ p_1 \end{array} \right] \mapsto
\left[\begin{array}{cc} (q_0^2-p_0^2+q_1^2-p_1^2) & -2(q_0p_0+q_1p_1) \\ 2(-q_0p_0+q_1p_1) & -(q_0^2-p_0^2-q_1^2 +p_1^2) \\ 2(q_0p_1+p_0q_1) & 2(q_0q_1-p_0p_1) \end{array} \right]
= \left [x \ v \right ],
\end{equation}
with group of deck transformations  generated by the antipodal map $$ A(z) = -z, \ \ z\in S^3. $$ Here $z = (q_0+ip_0,q_1+ip_1) \in \C^2 \simeq \R^4$. Later we will make use of polar coordinates $r_0e^{i\phi_0} = q_0+ip_0 \neq 0$ and $r_1e^{i\phi_1} = q_1+ip_1 \neq 0$.

Recall the standard contact form $\lambda_0 = \frac{1}{2} \sum_{j=0,1} q_jdp_j - p_jdq_j$ on $S^3$ and the standard contact structure $\xi_0 = \ker\lambda_0$, and observe that \begin{equation}\label{antipoda} A^*\lambda_0 = \lambda_0.\end{equation} The contact form $\bar\lambda_0$ on $T^1S^2$ given by $\bar\lambda_0|_{(x,v)} \cdot \zeta = g_0(v,d\Pi\cdot\zeta)$ $\forall \zeta \in T_{(x,v)}T^1S^2$ satisfies $D^*\bar\lambda_0 = 4\lambda_0$. Let $\bar \xi_0:=\ker \bar \lambda_0$ be its (tight) contact structure and $X_{\bar\lambda_0}$ its Reeb vector field. The flow of $\dot x = X_{\bar\lambda_0} \circ x$ is the geodesic flow of $(S^2,g_0)$ on $T^1S^2$. Its orbits are closed (the unit vectors tangent to the great circles) and correspond to the projections of Hopf fibers under $D$. There exists a natural orientation on $T^1S^2$ induced by $\bar \lambda_0 \wedge d \bar \lambda_0>0$.

As before, $K_0 := L_0\cup L_1\subset S^3$ is the Hopf link
\begin{equation}\label{hopflink}
\begin{aligned}
L_0:=\{q_0=p_0=0,\ q_1^2 + p_1^2=1\}, \\
L_1:=\{q_1=p_1=0,\ q_0^2+p_0^2=1\}.
\end{aligned}
\end{equation}
Its components doubly cover $$ l_i:=D(L_i),\ i=0,1, $$ which are the velocity vectors of the geodesics running the equator of $S^2$ in opposite directions. Let
\begin{equation}\label{hopfl}
l:=l_0 \cup l_1
\end{equation}
be the link in $T^1S^2$ which we also call a Hopf link.

\begin{definition}\label{remarkiso}
We call any transverse link $\bar l=\bar l_0 \cup \bar l_1\subset (T^1S^2,\bar\xi_0)$ a Hopf link if $\bar l$ is transversally isotopic to the link $l$ defined in \eqref{hopfl}. This means that there exists an isotopy $\rho_t:S^1\sqcup S^1 \to T^1S^2,t\in[0,1],$ such that $\Img (\rho_0) = \bar l,\ \Img (\rho_1) = l$, $\rho_t$ is an embedding positively transverse to the contact structure for all $t$. It is a theorem that the isotopy $\rho_t,\ t\in[0,1],$ can be extended to a contact isotopy  $\psi_t,\ t\in[0,1],$ of $(T^1S^2,\bar \xi_0)$, i.e., $\psi_{t*} \bar \xi_0 = \bar \xi_0$ and $\psi_t \circ \rho_0 = \rho_t$ for all $t$, see Theorem 2.6.12 in~\cite{geiges}. The Hopf link $l$ is said to be in normal position.
\end{definition}

For each $c\in S^1$, let $u_{1,c},\ u_{0,c}:\mathbb{C} \to S^3$ be the embeddings, defined in polar coordinates $(R,\theta)\in[0,\infty) \times \mathbb{R} / 2\pi \mathbb{Z}$ by
\begin{equation}\label{ucvc}
\begin{array}{ccc} u_{1,c}(R,\theta)=\frac{1}{\sqrt{1+R^2}}(R e^{i\theta},c), & & u_{0,c}(R,\theta )=\frac{1}{\sqrt{1+R^2}}(c,R e^{i\theta}). \end{array}
\end{equation}
Note that  $u_{i,c}$ is transverse to the Reeb vector field $X_{\lambda_0}$ of $\lambda_0$ and satisfies
\begin{equation}\label{asympt}
u_{i,c}(R,\cdot) \to L_i  \mbox{ as } R \to \infty.
\end{equation}
It is clear that the family $\{u_{i,c};\ c\in S^1\}$  smoothly foliates $S^3 \setminus L_i$ for $i=0,1$. Each $u_{i,c}$  has an orientation induced by $d\lambda_0$. Notice that $u_{i,c}(\C) \cap L_j = u_{i,c}(0)$ for all $c$ and $i\neq j$ and that $(c,R,\theta) \simeq u_{i,c}(R,\theta)$ is a diffeomorphism $S^3\setminus K_0 \simeq S^1 \times (0,\infty) \times \R/2\pi\Z$ for each $i=0,1$.

The solutions of $\dot x = X_{\lambda_0} \circ x$ in coordinates $(c,R,\theta)$ are given by
\begin{equation}\label{flowR0}
\begin{array}{ccc} c(t) =c(0)e^{2it}, & R(t) =R(0), & \theta(t) = \theta(0) + 2t. \end{array}
\end{equation}

Let $\bar u_{i,c}:\mathbb{C} \to T^1S^2,\ c\in S^1,$ be defined by $$ \begin{aligned} \bar u_{i,c} := D \circ u_{i,c}. \end{aligned} $$ for $i=0,1$. Now since $u_{i,c}$ does not have antipodal points in their image, we see that $\bar u_{i,c}$ is an embedding. From \eqref{asympt}, we have  \begin{equation}\label{asympt2} \bar u_{i,c}(R,\cdot) \to l_i
\mbox{ as }R \to \infty,\end{equation} the convergence being as double covering maps.

Notice from \eqref{ucvc} that
$$
\begin{array}{ccc}
\bar u_{i,c}(\mathbb{C})=\bar u_{i,c'}(\mathbb{C}) \Leftrightarrow  c'\in\{c,-c\} & \text{and} & \bar u_{i,c}(\mathbb{C}) \cap \bar u_{i,c'}(\mathbb{C})=\emptyset \Leftrightarrow  c'\not\in \{c,-c\},
\end{array}
$$
where the identification under $A$ is given by
\begin{equation}\label{identi}
(c,R,\theta)\sim (-c,R,\theta + \pi)
\end{equation}
in the coordinates $(c,R,\theta)$. Observe that $\bar u_{i,c}(\mathbb{C}) \cap l_j = \bar u_{i,c}(0)$ for all $c$ and $i\neq j$. It follows that $T^1S^2 \setminus l \simeq \bar S^1 \times (0,\infty) \times \R/\pi\Z$ where $\bar S^1 := S^1 / \{c\sim -c\} \simeq S^1$, and we find
\begin{equation}\label{isopi}
\pi_1(T^1S^2 \setminus l) \simeq \mathbb{Z} \oplus \mathbb{Z}
\end{equation}
where the choice of a base point is omitted.

Let $\bar \alpha : [0,1] \to T^1S^2 \setminus l$ be a closed curve and $\alpha : [0,1] \to S^3 \setminus K_0$ be a lift. In polar coordinates we have $\alpha(t) = (r_0(t)e^{i\phi_0(t)},r_1(t)e^{i\phi_1(t)})$ with continuous arguments $\phi_0,\phi_1 : [0,1] \to \R$, and non-vanishing continuous functions $r_0(t),r_1(t)$. We will denote
\begin{equation}\label{windtot}
  \wind_0(\bar \alpha) = \frac{\phi_0(1)-\phi_0(0)}{2\pi} \ \ \ \text{ and } \ \ \ \wind_1(\bar \alpha) = \frac{\phi_1(1)-\phi_1(0)}{2\pi}
\end{equation}
which are half-integers independent of the choice of lifts. It is always the case that
$$
  \wind_0(\bar \alpha)+\wind_1(\bar \alpha)\in \mathbb{Z}.
$$
The pair of half-integers $\wind_0,\wind_1$ determine a homotopy class in the complement of any Hopf link in $T^1S^2$.

We choose $[\bar a_0]$ and $[\bar a_1]$ as generators of $\pi_1(T^1S^2\setminus l)$  where
$$
\begin{array}{ccc}
\bar a_0(t) = D \circ a_0(t), & & \bar a_1(t) = D \circ a_1(t),
\end{array}
$$
and
$$
\begin{array}{ccc}
a_0(t) =\frac{1}{\sqrt{2}}(e^{\pi t i},e^{\pi t i}), & & a_1(t) =  \frac{1}{\sqrt{2}}(e^{\pi t i},e^{-\pi t i}),
\end{array}
$$
for $t\in[0,1]$. We have
\begin{equation}\label{wind1}
\wind_0(\bar a_0)=\wind_0(\bar a_1)=\wind_1(\bar a_0)=1/2 \mbox{ and } \wind_1(\bar a_1)=-1/2.
\end{equation}

Any homotopy class $[\bar \alpha] \in \pi_1(T^1S^2\setminus l)$ is uniquely determined by the half integers $\wind_0(\bar \alpha)$ and $\wind_1(\bar \alpha)$. The isomorphism~\eqref{isopi} can be chosen as
\begin{equation}\label{eq-pi1_isomorphism}
  [\bar \alpha]\mapsto (\wind_0(\bar \alpha)+\wind_1(\bar \alpha),\wind_0(\bar \alpha)-\wind_1(\bar \alpha)).
\end{equation}

The bundle $TTS^2$ decomposes as the direct sum of vertical and horizontal sub-bundles $TTS^2 = VTS^2 \oplus HTS^2$. Here $VTS^2= \ker d \Pi$ where $\Pi:TS^2 \to S^2$ is the projection, and $HTS^2 = \ker K$ where $K:TTS^2 \to TS^2$ is the connection map of $g_0$. Given $v \in T_xS^2$ there are associated isomorphisms
\begin{equation}\label{isos}
  \begin{aligned} i_V : T_xS^2 \to V_vTS^2, \\ i_H : T_xS^2 \to H_vTS^2, \end{aligned}
\end{equation}
and $$ T_vT^1S^2 = i_H(T_xS^2) \oplus i_V(\mathbb{R}v^{\perp})  = \ker \bar \lambda_0|_{(x,v)} \oplus \ker d\bar \lambda_0|_{(x,v)}, $$ where $v^{\perp}\in T_x^1S^2$ is normal to $v$ and $\{v,v^\bot\}$ is positive. In fact $$ \xi_0|_{(x,v)} = \ker \bar \lambda_0|_{(x,v)} = \R i_V(v^{\perp}) \oplus \R i_H(v^{\perp}) \mbox{ and } \ker d\bar \lambda_0|_{(x,v)} = \R i_H(v). $$ The vectors
\begin{equation}\label{vvperp}
\{ v^{\perp,V}:= i_V(v^{\perp}),  v^{\perp,H}:=i_H(v^{\perp})\}
\end{equation}
induce a global symplectic trivialization $\beta$ of the $d\bar \lambda_0$-symplectic vector bundle $\bar\xi_0 \to T^1S^2$ since $d\bar \lambda_0(v^{\perp,V},v^{\perp,H})=1$.

Now let $\lambda = f\bar \lambda_0$ be a contact form on $T^1S^2$ inducing the tight contact structure $\bar \xi_0$. Consider the Reeb flow associated to $\lambda$ and let $P=(x,T)$ be a closed Reeb orbit with prime period $T>0$. Multiplying the vectors $v^{\perp,V}$ and $v^{\perp,H}$ in \eqref{vvperp} by $1/\sqrt{f}$ we find a global $d\lambda$-symplectic trivialization of $\ker \lambda=\ker \bar \lambda_0 = \bar \xi_0$. With respect to this global frame we define the transverse rotation number of $P$ by
\begin{equation}\label{transverse_rotation_so3}
\rho(P) = T \lim_{t\to \infty}\frac{\theta(t)}{2\pi t}
\end{equation}
for any solution $0\neq v(t) \simeq r(t)e^{i\theta(t)}  \in \bar \xi_0|_{x(t)}$ of the linearized Reeb flow over $P$.

\subsection{Tight Reeb flows on $T^1S^2$}\label{sec-t1s2}

Here we prove Theorem~\ref{teot}, and for that we will use the same model as in the $S^3$ case. Let $\gamma:[0,1]\to [0,+\infty)\times [0,+\infty)$ be a curve associated to the numbers $\eta_0$ and $\eta_1$ and satisfying the conditions explained in Section~\ref{example-2}. With this curve we have the star-shaped hypersurface $S_\gamma \subset \mathbb{C}^2$ defined in polar coordinates by $$ S_\gamma := \{ (r_0,\phi_0,r_1,\phi_1) : (r_0^2,r_1^2) \in \Img(\gamma) \}. $$ Let $f_{\eta_0,\eta_1} \lambda_0$ be the contact form on $S^3$ associated to $S_\gamma$, as explained in Section~\ref{example-2}, and let $\bar \lambda_{\eta_0,\eta_1}$ be the tight contact form on $T^1S^2$ induced by the double covering map $D$. This form is well defined since $f_{\eta_0,\eta_1} \circ A = f_{\eta_0,\eta_1}$. It is clear that the Reeb flow of $f_{\eta_0,\eta_1} \lambda_0$ admits the Hopf link $K_0 = L_0 \cup L_1 \subset S^3$ as closed Reeb orbits, where $L_i$, $i=0,1$, is defined in~\eqref{hopflink}. Their transverse rotation numbers are $\rho(L_i) = 1 + \eta_i$, $i=0,1$. Thus the flow of $\bar \lambda_{\eta_0,\eta_1}$ on $T^1S^2$ admits the Hopf link $l = l_0 \cup l_1$ in normal position as closed Reeb orbits. Their transverse rotation numbers are equal to $$ \rho(l_i) = \rho(L_i)/2. $$

\begin{lemma}\label{lem-s1family}
The conclusions of Theorem~\ref{teot} hold for the Reeb flow of $\bar \lambda_{\eta_0,\eta_1}$ on $T^1S^2$. Moreover, for each pair of relatively prime integers $(p,q)$ satisfying~\eqref{non-resonance-eta-i}, the closed Reeb orbits satisfying $\wind_0=p$, $\wind_1=q$ in case $p+q$ is odd, or $\wind_0=p/2$, $\wind_1=q/2$ in case $p+q$ is even, appear as an $S^1$-family filling an embedded $2$-torus in $T^1S^2\setminus l$.
\end{lemma}

\begin{proof}
There is a unique point $z_0\in \Img (\gamma)$ such that the vector $(p,q)$ is normal to $\gamma$ at $z_0$. Using polar coordinates, solutions corresponding to $z_0$ satisfy $\dot r_i=0,\ i=0,1$, $\dot \phi_0 = mp$ and $\dot \phi_1 = mq$ for some $m > 0$. The values of $r_i,\ i=0,1,$ are determined by $z_0$. Assuming $\phi_0(0)=\phi_1(0)=0$ we find $\phi_0(t) = mpt$ and $\phi_1(t)=mqt$.

If $p+q$ is even then both $p$ and $q$ are odd since $(p,q)$ is relatively prime. A period of the corresponding Reeb orbit is $\pi/m$. Moreover
\begin{equation}\label{eqpsi} \begin{aligned} \phi_0(\pi/m)=p \pi, \\ \phi_1(\pi/m)=q \pi,\end{aligned} \end{equation} and this corresponds to a non-contractible closed orbit $\gamma_{p,q}$ on $T^1S^2$. Since $(p,q)$ is a relatively prime pair, this orbit is simple. From \eqref{eqpsi}, we have $\wind_0(\gamma_{p,q})=p/2$ and $\wind_1(\gamma_{p,q})=q/2$. Varying the initial condition $\phi_1(0)$ we find the $S^1$-family of such orbits filling a $2$-torus in $T^1S^2$.

Let us consider the case $p+q$ is odd. A period of the corresponding Reeb orbit is $2\pi/m$ since
\begin{equation}\label{eqpsi2}
\begin{aligned} \phi_0(2 \pi/m)=p2 \pi, \\ \phi_1(2 \pi/m)=q2 \pi,
\end{aligned}
\end{equation}
which obviously corresponds to a contractible closed orbit $\gamma_{p,q}$ on $T^1S^2$. Again, since $(p,q)$ is a relatively prime pair this orbit is simple, $\wind_0(\gamma_{p,q})=p$ and $\wind_1(\gamma_{p,q})=q$. We also find an $S^1$-family of such orbits varying the initial condition $\phi_1(0)$.
\end{proof}

\begin{proof}[Proof of Theorem~\ref{teot}]
The case $p+q$ odd follows almost immediately from Theorem~\ref{thm-1}. Consider the pulled back Reeb flow on $S^3 \backslash K_0$, corresponding to $\tilde f \lambda_0 = D^*(f \bar \lambda_0)$, it follows that $\tilde f \in \mathcal{F}$ on $S^3 \backslash K_0$; the numbers $\theta_i$ in Theorem~\ref{thm-1} coincide with $\eta_i$ for $i = 0,1$. Theorem~\ref{thm-1} implies that there is a simple closed orbit $\gamma_{p,q}$ satisfying $\link(\gamma_{p,q},L_0) = p$ and $\link(\gamma_{p,q},L_1) = q$.  Since $p+q$ is odd and $(p,q)$ is relatively prime it follows that one of $p,q$ is even and the other is odd and that $\bar \gamma_{p,q} = D \circ \gamma_{p,q}$ is a simple closed orbit in $T^1 S^2 \backslash l$ with
\[
\begin{array}{ccc}
\wind_0(\bar \gamma_{p,q}) = \link(\gamma_{p,q},L_0) = p, & & \wind_1(\bar \gamma_{p,q}) = \link(\gamma_{p,q},L_1) = q,
\end{array}
\]
as claimed.

In the case $p+q$ even, one can argue the same way, but the orbit $\bar \gamma_{p,q} = D \circ \gamma_{p,q}$ obtained may be simple or it may be a double cover of another simple orbit $\bar \gamma_{p,q}'$ with
\[
\begin{array}{ccc}
\wind_0(\bar \gamma_{p,q}') = p/2, & & \wind_1(\bar \gamma_{p,q}') = q/2.
\end{array}
\]
We wish to show that we can indeed find a simple orbit $\bar \gamma_{p,q}'$ as such, that is, in the homotopy class
\[
  [\bar a] =  \left( \frac{p+q}{2}, \frac{p-q}{2} \right)
\]
under the isomorphism~\eqref{eq-pi1_isomorphism}. Notice that loops in this homotopy class in $T^1 S^2 \setminus l$ are not contractible in $T^1 S^2$ since $\wind_0(\bar a)$ and $\wind_1(\bar a)$ are half-integers. Thus, if $\bar a \in [ \bar a]$ then $D (D^{-1}(\bar a))$ is a double-cover of $\bar a$.

To find the desired simple closed orbits, one can follow the argument that proves Theorem~\ref{thm-1} but working directly on the manifold $T^1S^2$. Let $h\bar\lambda_0$, $h>0$, be a contact form such that the associated Reeb flow is tangent to $l$, $\rho(l_i) \not\in \Q$, all closed Reeb orbits with action up to some number $T>0$ are non-degenerate, and no closed Reeb orbit with action $\leq T$ in $T^1S^2\setminus l$ is contractible in $T^1S^2\setminus l$. A pair $(m,n) \in \Z$ represents a homotopy class in $\pi_1(T^1S^2\setminus l)$ under the isomorphism~\eqref{eq-pi1_isomorphism}, and we denote by $\P^{\leq T,(m,n)}(h\bar\lambda_0)$ the set of closed Reeb orbits in $T^1S^2\setminus l$ representing this homotopy class with action $\leq T$. As explained in Section~\ref{cyl_alm_cpx_str}, we consider the symplectization $W_{\bar\xi_0} \subset T^*T^1S^2$ with projection $\tau:W_{\bar\xi_0}\to T^1S^2$ onto the base point. Every $P \in \P^{\leq T,(m,n)}(h\bar\lambda_0)$ has a well-defined Conley-Zehnder index $\mu_{CZ}(P)$ which is computed using the global trivialization~\eqref{vvperp}, with associated degree $|P| = \mu_{CZ}(P)-1$. The vector space $C_k^{\leq T,(m,n)}(h\bar\lambda_0)$ is freely generated, with coefficients in $\mathbb F_2$, by the elements of $\P^{\leq T,(m,n)}(h\bar\lambda_0)$ with degree $k$, and on the graded vector space $\bigoplus_k C_k^{\leq T,(m,n)}(h\bar\lambda_0)$ we have a differential which is defined by counting finite-energy $\jhat$-holomorphic cylinders in $\tau^{-1}(T^1S^2\setminus l) \subset W_{\bar \xi_0}$, asymptotic to orbits in $\P^{\leq T,(m,n)}(h\bar\lambda_0)$ with Fredholm index $1$. Here the almost complex structure $\jhat \in \J(h\bar\lambda_0)$ is induced by some $d\bar\lambda_0$-compatible complex structure $J:\bar\xi_0\to\bar\xi_0$, see Section~\ref{cyl_alm_cpx_str}, and is assumed to be Fredholm regular for the homotopy class $(m,n)$ and action bound $T$ in an analogous fashion as was explained in Section~\ref{chain_complex_subsection}. The associated homology is denoted by $HC_*^{\leq T,(m,n)}(h\bar\lambda_0,J)$. Analogous versions of Theorems~\ref{compactness_index_0_differential}-\ref{comparing_chain_maps_thm}, of Lemma~\ref{j_iota_relation_chain} and of Propositions~\ref{prop_morse_bott_approx},~\ref{comp_j_iota} can be proved similarly as before.

Suppose $f \bar \lambda_0$ is a non-degenerate contact form with Reeb vector field tangent to the Hopf link $l$, with associated numbers $\eta_i(f) = 2\rho(l_i)-1$, $i = 0,1$. Let $(p,q)\in\Z\times\Z$ be a relatively prime pair and assume that~\eqref{non-resonance-eta-i} holds. Following Section~\ref{section_computing}, we can choose numbers $\eta_0',\eta_1' \not\in \Q$ close to $\eta_1(f),\eta_2(f)$ and construct a contact form $h_+ \bar \lambda_0$ as a small perturbation of $\bar \lambda_{\eta'_0,\eta'_1}$, find a number $T>0$ and a suitable $d\bar\lambda_0$-compatible complex structure $J$ on $\bar\xi_0$ with the properties:
\begin{itemize}
\item the Reeb flow of $h_+ \bar \lambda_0$ is tangent to $l$, each $l_i$ is an elliptic orbit with associated transverse rotation number~\eqref{transverse_rotation_so3} equal to $\eta_i'$, $i = 0,1$;
\item for this pair $(p,q)$ we have $$ (\eta_0',1) < (p,q) < (1,\eta_1') \ \mbox{ if } \ (\eta_0(f),1) < (p,q) < (1,\eta_1(f)), $$ or $$ (1,\eta_1') < (p,q) < (\eta_0',1) \ \mbox{ if } \ (1,\eta_1(f)) < (p,q) < (\eta_0(f),1); $$
\item $ch_+<f<h_+$ pointwise for some $0<c<1$;
\item only two orbits in the $S^1$-family of closed Reeb orbits of $\bar\lambda_{\eta'_0,\eta'_1}$ as described in Lemma~\ref{lem-s1family} representing the homotopy class $(\frac{p+q}{2},\frac{p-q}{2})$ survive as closed Reeb orbits of $h_+\bar\lambda_0$, up to action $T/c + 1$, and these correspond to the elements in the sets $\P^{\leq T,(\frac{p+q}{2},\frac{p-q}{2})}(h_+\bar\lambda_0)$ and $\P^{\leq T,(\frac{p+q}{2},\frac{p-q}{2})}(ch_+\bar\lambda_0)$;
\item $h_+ \bar \lambda_0$ is non-degenerate up to action $T/c$ and has no Reeb orbits of action less than $T/c$ which are contractible in $T^1 S^2 \setminus l$;
\item $J$ induces $\R$-invariant almost complex structures $\jhat_+ \in \J(h_+\bar\lambda_0)$ and $\jhat_- \in \J(ch_+\bar\lambda_0)$ which are Fredholm regular with respect to the homotopy class $(\frac{p+q}{2},\frac{p-q}{2})$ and action bound $T/c$. This notion of regularity is defined exactly as in Section~\ref{chain_complex_subsection}. Hence the corresponding cylindrical contact homologies of $h_+ \bar \lambda_0$ and of $ch_+\bar\lambda_0$ in $T^1 S^2 \setminus l$ up to action $T$ in the homotopy class $(\frac{p+q}{2},\frac{p-q}{2})$ are well-defined and isomorphic to $H_*(S^1;\mathbb{F}_2)$ up to a common grade-shift:
\[
HC^{\leq T, (\frac{p+q}{2},\frac{p-q}{2})}_*(h_+ \bar \lambda_0,J) \simeq HC^{\leq T, (\frac{p+q}{2},\frac{p-q}{2})}_*(c h_+ \bar \lambda_0,J) \simeq H_{*-s}(S^1;\mathbb F_2),
\]
\item the map $j_* \circ \iota_*$ explained in Section~\ref{inclusions} is non-zero
\[
HC^{\leq T, (\frac{p+q}{2},\frac{p-q}{2})}_*(h_+ \bar \lambda_0,J) \xrightarrow{j_* \circ \iota_* \neq 0} HC^{\leq T, (\frac{p+q}{2},\frac{p-q}{2})}_*(c h_+ \bar \lambda_0,J).
\]
\end{itemize}

Following the argument in the $S^3$-case, consider $\jhat \in \J(f\bar\lambda_0)$. Recall the sets $\J(\jhat_-,\jhat)$, $\J(\jhat,\jhat_+)$ defined in Section~\ref{cyl_ends_alm_cpx_str}. The sets $\J(\jhat_-,\jhat:l)$, $\J(\jhat,\jhat_+:l)$ of almost complex structures for which $\tau^{-1}(l)$ is a pseudo-holomorphic surface are defined as in Section~\ref{restricted_class_alm_cpx_str}. We select $J_1 \in \J(\jhat_-,\jhat:l)$ and $J_2 \in \J(\jhat,\jhat_+:l)$. The family of almost-complex structures $\bar J_R = J_1 \circ_R J_2$ in $W_{\bar\xi_0}$, splitting along the hypersurface defined by the contact form $f \bar \lambda_0$, is defined as in Section~\ref{splitting_alm_cpx_str}.

For each $R>0$, there must exist a finite-energy $\bar J_R$-holomorphic cylinder $\util_R$ contained in the complement of $\tau^{-1}(l)$ with one positive and one negative puncture. In the positive puncture it is asymptotic to an orbit in $\P^{\leq T,(\frac{p+q}{2},\frac{p-q}{2})}(h_+\bar\lambda_0)$, and in the negative puncture it is asymptotic to an orbit in $\P^{\leq T,(\frac{p+q}{2},\frac{p-q}{2})}(ch_+\bar\lambda_0)$. To see this recall that, as in the $S^3$-case, $\bar J_R$ is diffeomorphic to some element in $\bar J'_R \in \J(\jhat_-,\jhat_+:l)$ and, if there are no $\bar J_R$-cylinders as claimed, we conclude that $\bar J'_R$ is Fredholm regular (for homotopy class $(\frac{p+q}{2},\frac{p-q}{2})$ and action bound $T$) and induces the zero map from $HC^{\leq T, (\frac{p+q}{2},\frac{p-q}{2})}_*(h_+ \bar \lambda_0,J)$ to $HC^{\leq T, (\frac{p+q}{2},\frac{p-q}{2})}_*(c h_+ \bar \lambda_0,J)$. However this map equals $j_*\circ\iota_* \not= 0$ by versions of Lemma~\ref{j_iota_relation_chain} and Proposition~\ref{comp_j_iota} in the $T^1S^2$-case, a contradiction.

Considering a sequence $R_n \to +\infty$, we may assume that the asymptotic orbits of the cylinders $\util_{R_n}$ do not depend on $n$, which guarantees uniform energy bounds for the sequence $\util_{R_n}$. This is so since the sets $\P^{\leq T,(\frac{p+q}{2},\frac{p-q}{2})}(h_+\bar\lambda_0)$ and $\P^{\leq T,(\frac{p+q}{2},\frac{p-q}{2})}(ch_+\bar\lambda_0)$ have two elements. We denote these limiting orbits by $\bar P_+$ and $\bar P_-$ at the positive and negative punctures, respectively.

The double cover $D$ induces a double cover $\widetilde D : W_{\xi_0} \to W_{\bar \xi_0}$. Since $p+q$ is even, the loops $t\mapsto \tau\circ \util_{R_n}(s,t)$ are non-contractible on $T^1S^2$. We can lift the maps
\[
\begin{array}{cc}
\util_{R_n}^2 : \R\times\R/2\Z \to W_{\bar\xi_0}, & (s,t) \mapsto \util_{R_n}(s,t)
\end{array}
\]
to finite-energy cylinders $\widetilde U_n : \R\times\R/2\Z \to W_{\xi_0}$ holomorphic with respect to $\widetilde D^*(J_1\circ_R J_2)$ with uniform energy bounds. There exists a SFT-convergent subsequence $\widetilde U_{n_j}$. Denoting by $\widetilde \tau : W_{\xi_0} \to S^3$ the projection onto the base point, the loops $t \mapsto \widetilde\tau \circ \widetilde U_{n_j}(s,t)$ link $p$ times with $L_0$ and $q$ times with $L_1$. Arguing just as in the proof of Lemma~\ref{existence_lemma}, using the relations satisfied by the numbers $q$, $p$, $\eta_0(f)$, $\eta_1(f)$, $\eta_0'$ and $\eta_1'$, we find a closed orbit $\widetilde P_{p,q} \subset S^3\setminus (L_0\cup L_1)$ of the Reeb flow associated to the contact form $D^*(f\bar\lambda_0) = (f\circ D)4\lambda_0$ satisfying $\link(\widetilde P_{p,q},L_0)=p$ and $\link(\widetilde P_{p,q},L_1)=q$. Moreover, the orbit $\widetilde P_{p,q}$ can be approximated in $C^\infty$ by loops of the form $t\mapsto \widetilde\tau \circ \widetilde U_{n_j}(s_j,t+t_j)$, $t\in\R/2\Z$, for suitable $s_j,t_j$ and $j$ large. Since these loops project down to $T^1S^2$ via the map $D$ to doubly covered loops, the same is true for the loop $\widetilde P_{p,q}$. This means that $D\circ \widetilde P_{p,q}$ is the double cover of a prime closed orbit $(f\bar\lambda_0)$-Reeb orbit $\bar P_{p,q}$. It follows that $\bar P_{p,q}$ is in the homotopy class $(\frac{p+q}{2},\frac{p-q}{2})$ since $\wind_0(\bar P_{p,q}) = p/2$ and $\wind_1(\bar P_{p,q}) = q/2$.
\end{proof}

\begin{proof}[Proof of Corollary~\ref{cor-geod2}] Corollary \ref{cor-geod2} is immediate from Theorem \ref{teot}, since the Reeb flow on the unit cotangent bundle of a Finsler metric $F$ is the geodesic flow of $F$. \end{proof}

In the case that the metric is reversible, if an embedded curve is a geodesic when traversed in one direction it will automatically be a geodesic when traversed in the opposite direction so that Corollary \ref{cor-geod2} applies, and moreover one finds the relation $\eta_0 = \eta_1$.  We will explore this case in greater detail in the next section.

\subsection{Reversible simple geodesics of Finsler metrics} \label{sec-reversibleFinsler}

We now recover Angenent's theorem at the level of homotopy classes for the more general framework of Finsler metrics on $S^2$ with reversible simple geodesics and prove Corollary~\ref{cor-geod3}.  Although it is a particular case of Corollary \ref{cor-geod2}, we will be more explicit; in particular we will clarify the relationship between the geodesics we find and the satellites found by Angenent \cite{angenent}.

A simple closed geodesic $\gamma$ with unit speed of a Finsler metric $F$ on $S^2$ is called \emph{reversible} if the curve $t\mapsto \gamma(-t)$ is a reparametrization of another geodesic $\gamma_r$ and if, in addition, the inverse rotation numbers $\rho = \rho(\gamma) = \rho(\gamma_r)$ coincide.  The geodesics $\gamma$ and $\gamma_r$ determine a Hopf link $l_\gamma = \bar\gamma \cup \bar\gamma_r \subset T^1S^2$, where the special lifts $\bar\gamma = G^{-1}(\dot\gamma)$, $\bar\gamma_r = G^{-1}(\dot\gamma_r)$ are defined in~\eqref{lift_gamma_bar} using the special contactomorphism $G$ described in~\eqref{diffeo_G}. Now consider a contactomorphism $\varphi$ of $(T^1S^2,\bar\xi_0)$ such that $\varphi(\bar\gamma) = l_0$ and $\varphi(\bar\gamma_r) = l_1$, where $l = l_0 \cup l_1$ is the standard Hopf link in $T^1S^2$ defined by $l_i = D(L_i)$, see~\eqref{hopfl}. Such a contactomorphism exists since $l_\gamma$ is transversally isotopic to $l$.

\begin{lemma}\label{lempq}
Let $\gamma_{p,q}$ be a $(p,q)$-satellite of $\gamma$ with unit speed with respect to $F$, and consider $\varphi(\bar \gamma_{p,q}) \subset T^1S^2\setminus l$.
\begin{enumerate}
\item If $q>0$ then
$$
  \begin{aligned}
    \wind_0(\varphi(\bar \gamma_{p,q})) & = |p|-q/2, \\
    \wind_1(\varphi(\bar \gamma_{p,q}))& = q/2.
  \end{aligned}
$$
\item If $q<0$ then
$$
  \begin{aligned}
    \wind_0(\varphi(\bar \gamma_{p,q}))& = q/2, \\
    \wind_1(\varphi(\bar \gamma_{p,q}))& = |p|-q/2.
  \end{aligned}
$$
\end{enumerate}
\end{lemma}

\begin{proof}
Assume $q>0$. First we work directly on $F^{-1}(1)$ and recall some basic facts about Finsler geometry. $F$ determines inner-products $g_v(\cdot,\cdot)$ on $T_{\Pi(v)}S^2$ for each $v\not = 0$ by $$ g_v(w_1,w_2) = \frac{1}{2} \left. \frac{\partial^2}{\partial s \partial t}\right|_{s=t=0} F^2(v+tw_1+sw_2). $$ The 1-form $\bar\lambda_F$ defined in the introduction is written as $\bar\lambda_F|_v \cdot \zeta = g_v(v,d\Pi \cdot \zeta)$. We choose a Riemannian metric $h$ such that $h_{\gamma(t)} = g_{\dot\gamma(t)}$ for every $t \in \R/T\Z$ and $\gamma(t)$ is an $h$-geodesic. Here $T>0$ denotes the prime period of $\gamma$. We assume $T=1$ for simplicity. The 1-form $\bar\lambda_h$ on $TS^2$, defined by $\bar\lambda_h|_v \cdot \zeta = h_x(v,d\Pi \cdot \zeta)$ where $x = \Pi(v)$, coincides with $\bar\lambda_F$ on $T_{\dot\gamma}TS^2$, and in particular on $T_{\dot\gamma}F^{-1}(1)$. Let $N(t)$ be a vector field along $\gamma(t)$ such that $\{\dot\gamma(t),N(t)\}$ is a positive orthonormal basis of $T_{\gamma(t)}S^2$, and consider the $(p,q)$-satellite $$ \alpha_\epsilon(t) = \exp_{\gamma(t)}(\epsilon\sin(2\pi pt/q)N(t)) $$ defined in $\R/q\Z$. Here $\exp$ is the exponential map associated to $h$. As usual, there is a connection map $K:TTS^2 \to TS^2$ associated to $h$ inducing a splitting $TTS^2 = VTS^2 \oplus HTS^2$ where $VTS^2 = \ker d\Pi$ and $HTS^2 = \ker K$. Moreover, for every $v\not=0$ there are isomorphisms $i_v : T_{\Pi(v)}S^2 \to V_vTS^2$ and $(d\Pi|_{HTM})^{-1} : T_{\Pi(v)}S^2 \to H_vTS^2$, where $i_v(w) = \left. \frac{d}{dt} \right|_{t=0} (v+tw)$. So we always view a vector in $T_vTS^2$ as a pair of vectors in $T_{\Pi(v)}S^2$ referred to as the vertical and horizontal parts.

These objects allow us to understand the velocity vector $\dot\alpha_\epsilon$, in fact, $$ \dot\alpha_\epsilon(t) = J(\epsilon,t) $$ where $s \mapsto J(s,t)$ is the Jacobi field along the $h$-geodesic $$ s \mapsto \exp_{\gamma(t)}(s\sin(2\pi pt/q)N(t)) $$ with initial conditions $J(0,t) = \dot\gamma(t)$ and $\left. \frac{DJ}{ds} \right|_{s=0} = \frac{D}{dt}(\sin(2\pi pt/q)N(t))$. All covariant derivatives are taken with respect to metric $h$. Thus $\left. \frac{d}{d\epsilon} \right|_{\epsilon=0} \dot\alpha_\epsilon(t)$ is a vector in $T_{\dot\gamma(t)}TS^2$ with vertical part equal to $(2\pi p/q)\cos(2\pi pt/q)N(t)$ and horizontal part $\sin(2\pi pt/q)N(t)$. Consider the vector $\zeta(t) := \left. \frac{d}{d\epsilon} \right|_{\epsilon=0} \frac{\dot\alpha_\epsilon(t)}{F(\dot\alpha_\epsilon(t))}$. Then the horizontal part of $\zeta(t)$ is still equal to $\sin(2\pi pt/q)N(t)$ and its vertical part equals
\[
(2\pi p/q)\cos(2\pi pt/q)N(t) + \left( \left. \frac{d}{d\epsilon} \right|_{\epsilon=0} \frac{1}{F(J(\epsilon,t))} \right) \dot\gamma(t).
\]
However $\zeta(t)$ must be tangent to $F^{-1}(1)$, which implies $\left( \left. \frac{d}{d\epsilon} \right|_{\epsilon=0} \frac{1}{F(J(\epsilon,t))} \right) = 0$, so that $\zeta(t) = \left. \frac{d}{d\epsilon} \right|_{\epsilon=0} \dot\alpha_\epsilon(t)$. Let $N^{\rm hor}(t),N^{\rm vert}(t) \in T_{\dot\gamma(t)}TS^2$ be the horizontal and vertical lifts of $N(t)$. Then $\{N^{\rm vert}(t),N^{\rm hor}(t)\}$ is a $d\bar\lambda_F$-positive basis along $\bar\xi_F|_{\dot\gamma(t)}$ which extends to a global $d\bar\lambda_F$-positive basis of $\bar\xi_F$. In particular, we computed that $\zeta(t) \in \bar\xi_F|_{\dot\gamma(t)} \ \forall t$ and the representation of $\zeta(t)$ in this global frame as vector in $\R^2$ is $$ \left( \frac{2\pi p}{q}\cos(2\pi pt/q) , \sin(2\pi pt/q) \right). $$ Hence its winding equals $|p|$ when $t$ does one full turn in the circle $\R/q\Z$.

Defining $\bar\zeta(t) = dG^{-1}|_{\dot\gamma} \cdot \zeta(t)$ we obtain a section of $\bar\xi_0|_{\bar\gamma}$ which winds $|p|$ times with respect to a global positive frame when $t$ goes from $0$ to $q$. We have the \emph{a priori} fixed contactomorphism $\varphi$ of $(T^1S^2,\bar\xi_0)$ that brings $\bar\gamma$ into normal position, that is, $\hat\gamma(t) = \varphi \circ\bar\gamma(t)$, $t\in[0,1]$, is a reparametrization of the knot $l_0 = D(L_0)$ where $D:S^3 \to T^1S^2$ is the double covering map discussed before and $L_0 = S^3 \cap (0\times \C)$. Again, $\hat\zeta = d\varphi|_{\bar\gamma} \cdot \bar\zeta$ winds $|p|$ times with respect to a global positive frame when $t$ goes from $0$ to $q$.

We see $\hat\gamma(t)$ as a smooth $1$-periodic function of $t\in\R$, and $\hat\zeta(t)$ as smooth and $q$-periodic. Consider $\mathcal D \subset S^3$ an embedded disk spanning $L_0$ and $W:S^3 \to \xi_0$ a smooth non-vanishing section which is symmetric with respect to the covering group of $D$: $A_*W = W$ where $A$ is the antipodal map. A choice of lift $\hat\Gamma(t)$ of $\hat\gamma$ must be $2$-periodic and equivariant: $\hat\Gamma(t+1) = A \circ \hat\Gamma(t)$. Choose also a lift $\hat Z(t) \in \xi_0|_{\hat\Gamma(t)}$ of $\hat\zeta(t)$. Then $\hat Z(t+q) = dA|_{\hat\Gamma(t)} \cdot \hat Z(t)$ if $q$ is odd, or $\hat Z(t+q) = \hat Z(t)$ if $q$ is even.

In the following all windings of sections of $\xi_0$ along curves are computed using the orientation of $\xi_0$ induced by the standard symplectic form in $\C^2 \supset S^3$. By the symmetry of $W$, $D_*W$ is a well-defined non-vanishing section of $\bar\xi_0$. Our previous computations imply that
\[
  \wind_{[0,q]}(\hat Z(t),W \circ \hat\Gamma(t)) = |p|.
\]
Let $Y(t) \in \xi_0|_{\hat\Gamma(t)} \cap T_{\hat\Gamma(t)}\mathcal D$ be a non-vanishing $2$-periodic vector. Since $L_0$ has self-linking number $-1$ we have $\wind_{[0,2]}(W\circ\hat\Gamma(t),Y(t)) = -1$. This implies
\begin{equation}\label{wind_eqn_sat}
  \begin{aligned}
    \wind_{[0,2q]}(\hat Z(t),Y(t)) &= \wind_{[0,2q]}(\hat Z(t),W\circ \hat\Gamma(t)) \\
    &+ \wind_{[0,2q]}(W\circ \hat\Gamma(t),Y(t)) \\
    &= 2|p|-q.
  \end{aligned}
\end{equation}
Consider the point $c(\epsilon,t) \in S^3$ given by lifting $\varphi \circ G^{-1} (\dot\alpha_\epsilon(t) / F(\dot\alpha_\epsilon(t) ))$. Choosing the correct lift we obtain
\begin{equation}\label{Z_hat_formula}
\hat Z(t) = \partial_{\epsilon}c(0,t).
\end{equation}
Let us consider $q_j+ip_j = r_je^{i\phi_j}$ polar coordinates in $\C^2$, $j=0,1$, and write $$ c(\epsilon,t) = (r_0(\epsilon,t)e^{i\phi_0(\epsilon,t)},r_1(\epsilon,t)e^{i\phi_1(\epsilon,t)}) $$ where $\phi_0(\epsilon,t),\phi_1(\epsilon,t)$ are continuous lifts of the angles to $\R$. This is well-defined since $c(\epsilon,t) \not\in L_0 \cup L_1$. Note that $\mathcal D$ and $Y(t)$ can be chosen to satisfy $Y = \partial_{q_0}$. Using~\eqref{wind_eqn_sat} and~\eqref{Z_hat_formula} we get $(\phi_0(\epsilon,2q)-\phi_0(\epsilon,0))/2\pi = 2|p|-q$. By symmetry we get
\[
\wind_0 \left( \varphi \circ G^{-1} \left( \frac{\dot\alpha_\epsilon(t)}{F(\dot\alpha_\epsilon(t))} \right) \right)|_{t\in[0,q]} = \frac{\phi_0(\epsilon,q)-\phi_0(\epsilon,0)}{2\pi} = |p|- \frac{q}{2},
\]
which is the desired conclusion. Since for $\epsilon$ small the curve $\varphi \circ G^{-1} \left( \dot\alpha_\epsilon(t)/F(\dot\alpha_\epsilon(t)) \right)$ is $C^\infty$-close to $l_0^q$, its $\wind_1$ is equal to $q/2$.

If $q<0$ then $\dot\gamma_{p,q}$ is close to a $q$-fold cover of the curve $\dot\gamma_r \in F^{-1}(1)$ and the proof follows analogously.
\end{proof}

\begin{proof}[Proof of Corollary \ref{cor-geod3}]
By reversibility we can assume $q>0$. By hypothesis, we get $p>0$. Choosing a suitable contactomorphism $\varphi$ of $(T^1S^2,\bar\xi_0)$, as in the proof of Lemma~\ref{lempq}, we can assume $l_0 = \varphi\circ G^{-1}(\dot\gamma)$, $l_1 = \varphi\circ G^{-1}(\dot\gamma_r)$ is the standard Hopf link~\eqref{hopfl}, where $G$ is the diffeomorphism~\eqref{diffeo_G}. The contact form $\lambda_F$ gets transformed by pushing forward via the map $\varphi \circ G^{-1}$ to $f\bar\lambda_0$, for some $f>0$. If $q$ is odd let $p'=2p-q$ and $q'=q$, otherwise let $p'=p-q/2$, $q'=q/2$. The integers $p'$ and $q'$ are relatively prime. Assume first that $\rho > 1$. Then $$ 1 < \frac{p}{q}< \rho \Leftrightarrow 1 < \frac{2p-q}{q} = \frac{p'}{q'} < 2\rho-1. $$ This implies that
\[
(\eta_0,1) = (2\rho - 1,1) < (p',q') < (1,1) < (1,2\rho - 1) = (1,\eta_1).
\]
Therefore we can apply Theorem~\ref{teot} to the pair $(p',q')$ to find a simple closed $f\lambda_0$-Reeb orbit on $T^1S^2$, denoted here by $c_{p,q}$, such that $$ \begin{aligned} \wind_0(c_{p,q}) & = p-q/2, \\ \wind_1(c_{p,q}) & = q/2. \end{aligned} $$ By Lemma~\ref{lempq}, the closed geodesic $\gamma_{p,q}$ satisfying $\varphi\circ G^{-1}(\dot\gamma_{p,q}) = c_{p,q}$ is in the same homotopy class in $F^{-1}(1) \setminus (\dot\gamma \cup \dot\gamma_r)$ as the velocity vector of a  $(p,q)$-satellite of $\gamma$ when normalized by $F$. In case $0\leq\rho<1$ we have
\begin{equation}
\rho < \frac{p}{q}< 1 \Leftrightarrow 2\rho-1 < \frac{2p-q}{q} = \frac{p'}{q'} < 1
\end{equation}
and this implies
\[
(1,\eta_1) = (1,2\rho - 1) < (1,1) < (p',q') < (2\rho - 1,1) = (\eta_0,1).
\]
Applying Theorem~\ref{teot} to $(p',q')$ we obtain the desired closed geodesic $\gamma_{p,q}$.
\end{proof}

\appendix

\section{Theorems on contact homology}\label{app_thms_contact_hom}

In this appendix we provide self-contained proofs of the theorems from Section~\ref{section_contact_homology}.

\subsection{Proofs of Theorems~\ref{compactness_index_0_differential} and~\ref{d2=0}}\label{proofs_thms_chain_complex}

The argument relies on a careful analysis of the compactification of moduli spaces of the form $\M^{\leq T,(p,q)}_{\jhat}(P,P'')/\R$, where $P,P'' \in \P^{\leq T,(p,q)}(\lambda)$ satisfy
\begin{itemize}
\item[{\bf A)}] $\mu_{CZ}(P'') + 1 = \mu_{CZ}(P)$, or
\item[{\bf B)}] $\mu_{CZ}(P'') + 2 = \mu_{CZ}(P)$.
\end{itemize}

We denote by $\P(\lambda)$ the set of all closed Reeb orbits of $\lambda$. As remarked in Section~\ref{chain_complex_subsection}, the space $\M^{\leq T,(p,q)}_{\jhat}(P,P'')/\R$ has the structure of a smooth manifold of dimension equal to 0 in case A, or equal to 1 in case B, since $\jhat \in \J_{\rm reg}(\lambda)$ and all cylinders are somewhere injective. It admits a compactification described in~\cite{sftcomp}, which is obtained by adding holomorphic buildings of height $\geq 1$. In our particular situation where the buildings arise as limits of cylinders, they can be given a slightly different and simpler description as a finite collection $\{\util_v\}$ of finite-energy $\jhat$-holomorphic spheres with one positive puncture, where $v$ ranges in the set of vertices of a finite tree $\mathcal T$ with a root $\overline v$ and a distinguished leaf $\underline v$. Every map $\util_v$ is not a trivial cylinder over some periodic orbit. After a reparametrization we will always assume that $\infty \in \C\sqcup \{\infty\}\simeq \CP^1$ is the positive puncture of each $\util_v$. The edges are oriented as going away from the root, so at each vertex $v\neq \overline v$ there is exactly one incoming edge from its parent, and possibly many outgoing edges to its children. The negative punctures of $\util_v$ are in 1-1 correspondence with the outgoing edges of $v$, so that all leaves are planes, with the exception of the distinguished leaf $\underline v$ which has one negative puncture where $\util_{\underline v}$ is asymptotic to $P''$. The curve $\util_{\overline v}$ is asymptotic to $P$ at its positive puncture, and we have the following compatibility conditions:
\begin{itemize}
  \item Let $e$ be an edge from $v$ to $v'$. There exists an orbit $P_e \in \P(\lambda)$ such that $\util_v$ is asymptotic to $P_e$ at the negative puncture corresponding to $e$, and $\util_{v'}$ is asymptotic to $P_e$ at its positive puncture.
  \item Let $e$, $v$, $v'$, $P_e = (x_e,T_e)$ be as above and $z\in \CP^1$ be the negative puncture of $\util_v$ corresponding to $e$. There is an orientation reversing isometry\footnote{The conformal structure naturally induces a metric structure and an orientation on each circle $(T_z\CP^1\setminus 0)/\R^+$.} $$ r_e:(T_z\CP^1\setminus 0)/\R^+ \to (T_\infty\CP^1\setminus 0)/\R^+ $$ such that if $c,C:[0,\epsilon) \to \CP^1$ are curves satisfying $c(0) = z$, $\dot c(0)\neq0$, $C(0)=\infty$, $\dot C(0)\neq 0$ and $r_e(\R^+\dot c(0))=\R^+\dot C(0)$ then $\tau\circ \util_v(c(t))$ and $\tau\circ \util_{v'}(C(t))$ converge to the same point in $x_e(\R)$ as $t\to 0$.
\end{itemize}
The set $\{r_e\}$ of isometries as above is called, in the language of SFT, a decoration of the underlying nodal sphere. We say that $\util_{v}$ is in level $k$ if the number of edges connecting $v$ to $\overline v$ is $k-1$. We shall briefly refer to $\{\util_v\}$ as a bubbling-off tree of finite-energy spheres. Clearly, the structure just described is different but equivalent to that of a holomorphic building in the boundary of $\M^{\leq T,(p,q)}_{\jhat}(P,P'')/\R$ explained in~\cite{sftcomp}.

Let $\{[\util_n]\} \subset \M^{\leq T,(p,q)}_{\jhat}(P,P'')/\R$ be a sequence. We will see underlying maps $\util_n$ representing this sequence as defined in $\CP^1 \setminus \{[0:1],[1:0]\} \simeq \C\setminus 0$, where $\infty$ is the positive puncture and $0$ is the negative puncture. Then $\{[\util_n]\}$ will be said to converge to a bubbling-off tree $\{\util_v\}$ as above if for each vertex $v$ one finds constants $c_n\in \R$, $A_n,B_n \in \C$, $A_n \neq0$, such that
\begin{equation}\label{general_vertex_parametrization}
  \{z\mapsto g_{c_n} \circ \util_n(A_nz+B_n)\} \to \util_v \ \text{ in } \ C^\infty_{\rm loc}(\C\setminus \Gamma_v)
\end{equation}
as $n\to\infty$, where $\Gamma_v \subset \C$ is the set of negative punctures of $\util_v$. The limiting tree has a stem\footnote{Alternatively, the stem consists precisely of the vertices which have $\underline v$ among its descendants.} $S = (v_0,\dots,v_N)$ which is the unique path connecting the root $\overline v = v_0$ to the distinguished leaf $\underline v = v_N$, where $v_{i+1}$ is a child of $v_i$. The edge connecting $v_i$ to $v_{i+1}$ will be denoted by $e_i$ ($i=0,\dots,N-1$). Since the positive and negative punctures of $\util_n$ are $\infty,0$ respectively, we may assume, after further reparametrization and without loss of generality, that for every $i=0,\dots,N-1$ we have $0\in \Gamma_{v_i}$, the edge $e_i$ corresponds to $0$, and there are constants $c_n \in\R$, $A_n \in \C\setminus 0$ such that
\begin{equation}\label{stem_parametrization}
  \{z\mapsto g_{c_n} \circ \util_n(A_nz)\} \to \util_{v_i}(z)
\end{equation}
in $C^\infty_{\rm loc}(\C\setminus \Gamma_{v_i})$ as $n\to \infty$.

From now on we assume that $\mathcal T$ has more than one vertex, and split the argument into a few steps. \\

\noindent {\bf (I)} Every orbit $P_{e_i}$, corresponding to an edge $e_i$ connecting vertices $v_i$ and $v_{i+1}$ in the stem, is not contained in $K_0$. \\

By our assumptions, there are $\lambda$-Reeb trajectories $x_j$ with minimal periods $T_j>0$ such that $L_j = x_j(\R)$, $j=0,1$. We write $L_j = (x_j,T_j)$. Arguing indirectly, suppose that $P_{e_i} \subset K_0$ for some $i$ and set
\[
  i_0 = \min \{ i=0,\dots,N-1 \mid P_{e_i} \subset K_0\}.
\]
We treat the case $P_{e_{i_0}} \subset L_0$, the case $P_{e_{i_0}} \subset L_1$ is analogous. Then $\exists k\geq 1$ such that $P_{e_{i_0}} = (x_0,kT_0) = L_0^k$. Since $P,P''$ do not intersect $K_0$ we can define
\[
  i_1 = \max \{ i=0,\dots,N-1 \mid P_{e_i} = L_0^k \} \geq i_0.
\]

The $d\lambda$-energy of the curves $\util_{v_{i_0}},\util_{v_{i_1+1}}$ do not vanish. This is obvious for the curve $\util_{v_{i_0}}$ in view of the definition of $i_0$. The curve $\util_{v_{i_1+1}}$ is asymptotic to $L_0^k$ at its positive puncture $\infty$. If its $d\lambda$-energy vanishes then $\util_{v_{i_1+1}}$ is asymptotic to $L_0^r$ at the negative puncture $0$ for some $r<k$, in view of the definition of $i_1$. If $\rho>0$ is large enough and $t',t''$ are suitably chosen the loops
\[
  \begin{array}{ccc}
    c(t) = \tau \circ \util_{v_{i_1+1}}(\rho e^{i2\pi (t+t')}) & \text{and} & C(t) = \tau\circ \util_{v_{i_1+1}}(\rho^{-1} e^{i2\pi (t+t'')})
  \end{array}
\]
are $C^0$-close to $x_0(kT_0t)$ and $x_0(rT_0t)$, respectively. However, by~\eqref{stem_parametrization}, $c$ and $C$ can be approximated by loops of the form $t\mapsto \tau\circ\util_n (\rho_ne^{i2\pi(t+t_n)})$, with suitable $\rho_n,t_n$, so that they are homotopic to each other in $S^3\setminus L_1$, which implies that $t\mapsto x_0(kT_0t)$ and $t\mapsto x_0(rT_0t)$ have the same linking number with $L_1$, contradicting $k\neq r$.

Let $\U$ be a tubular neighborhood of $L_0$ and $\Phi: \U \to \R/\Z\times B$ be a diffeomorphism, where $B\subset \R^2$ is a small ball centered at the origin, so that $(\U,\Phi)$ is a Martinet tube for $L_0$, as in Definition~\ref{martinet_tube_def}. The coordinates in $\R/\Z\times B$ will be denoted by $(\theta,x,y)$. We write
\[
  \begin{aligned}
    U_0(s,t) &= \tau \circ \util_{v_{i_0}}(\est), \\
    U_1(s,t) &= \tau \circ \util_{v_{i_1+1}}(\est),
  \end{aligned}
\]
and
\[
  \begin{aligned}
    (\theta_0(s,t),x_0(s,t),y_0(s,t)) &= \Phi \circ U_0(s,t) \ \text{ for } s\ll -1, \\
    (\theta_1(s,t),x_1(s,t),y_1(s,t)) &= \Phi \circ U_1(s,t) \ \text{ for } s\gg +1.
  \end{aligned}
\]
Let $A$ be the asymptotic operator at $L_0^k$. Then, by Theorem~\ref{precise_asymptotics}, we find eigenvalues $\nu_+>0$ and $\nu_-<0$ of $A$, and corresponding eigenfunctions $\eta_+,\eta_-$ satisfying $A\eta_\pm = \nu_\pm\eta_\pm$, such that the following holds: if $\zeta_\pm(t) : \R/\Z \to \R^2\setminus 0$ are the representations of $\eta_\pm$ in the frame $\{\partial_x,\partial_y\}$ of $(0\times \R^2)|_{\R/\Z\times 0} \simeq \xi_0|_{x_0(\R)}$, respectively, then
\[
  \begin{aligned}
    (x_0(s,t),y_0(s,t)) &= e^{\int_{s_0}^s\alpha^+(q)dq} (\zeta_+(t) + R_+(s,t)) \ \text{ for } \ s\leq s_0, \\
    (x_1(s,t),y_1(s,t)) &= e^{\int_{-s_0}^s\alpha^-(q)dq} (\zeta_-(t) + R_-(s,t)) \ \text{ for } \ s\geq -s_0
  \end{aligned}
\]
where $s_0 \ll -1$, $|R_\pm(s,t)| \to 0$ and $|\alpha^\pm(s)-\nu_\pm|\to0$ as $s\to \mp\infty$. Moreover we have
\begin{equation}\label{functions_theta_0_theta_1}
  \begin{aligned}
    \theta_0(s,t) \to kt + t_0 & \ \text{ as } \ s\to-\infty, \ \text{uniformly in} \ t \\
    \theta_1(s,t) \to kt + t_1 & \ \text{ as } \ s\to+\infty, \ \text{uniformly in} \ t
  \end{aligned}
\end{equation}
for some $t_0,t_1$. By our assumptions $\mu_{CZ}(L_0^m)$ is odd, for every $m\geq 1$. Thus, in view of the definition of the Conley-Zehnder index in terms of asymptotic eigenvalues discussed in Section~\ref{analytical_description_section}, see~\eqref{analytical_cz_def}, the winding numbers of $\zeta_+$ and $\zeta_-$ are different. Consequently, by the local representations above, we find that, for $s\gg1$, the loops $t\mapsto U_0(-s,t)$ and $t\mapsto U_1(s,t)$ are not homotopic in $S^3\setminus (L_0 \cup L_1)$. Here we used the existence of an isotopy of embeddings $f_t : S^3\setminus L_1 \to S^3\setminus L_1$, $t\in[0,1]$, satisfying $f_0 = id$, $f_1(S^3\setminus L_1) = \mathcal U$ and $f_t(L_0) = L_0, \ \forall t$. Let us fix $s_0 \gg 1$. By~\eqref{stem_parametrization}, the loop $t\mapsto U_0(-s_0,t)$ can be approximated by loops of the form $t\mapsto \tau\circ \util_n(r_ne^{i2\pi(t+t_n)})$ for suitable values of $r_n$ and $t_n$. Also the loop $t\mapsto U_1(s_0,t)$ can be approximated by loops of the same form. This is a contradiction to the fact that $\tau\circ\util_n(\R\times\R/\Z) \subset S^3\setminus (L_0\cup L_1)$ $\forall n$. Thus, all orbits $P_{e_i}$ lie in $S^3\setminus K_0$ and belong to the homotopy class $(p,q)$ since the cylinders $\util_n$ do not touch $\tau^{-1}(K_0)$. \\

\noindent {\bf (II)} $\Gamma_{v_i} = \{0\}$ for every $i=1,\dots,N$, that is, the tree coincides with the stem. \\

Arguing by contradiction, let $z\in \Gamma_{v_i}\setminus \{0\}$ be a negative puncture corresponding to an edge $e\neq e_i$ connecting $v_i$ to one of its children $v^*\neq v_{i+1}$. The vertex $v^*$ together with all its descendants and all edges connecting them form a proper sub-tree $\mathcal T_1 \subset \mathcal T$ with root $v^*$ that does not contain the distinguished leaf $\underline v$ ($e$ is not an edge of $\mathcal T_1$). Consider a leaf $\hat v$ of $\mathcal T_1$. The curve $\util_{\hat v}$ is a finite-energy plane asymptotic to some $\hat P \in \P(\lambda)$ at its positive puncture. We claim that $D := \{z\in\C \mid \util_{\hat v}(z) \in \tau^{-1}(K_0)\}$ is non-empty and discrete. Either $\hat P$ is contained in $K_0$ or not. In the first case, $\hat P$ lies in one component of $K_0$ and, consequently, $\exists z\in\C$ such that $\util_{\hat v}(z)$ belongs the other component of $K_0$. In the second case, note that $\hat P$ is not contractible in $S^3\setminus K_0$, which again implies $D\not=\emptyset$. If $D$ accumulates at some point of $\C$ then Carleman's similarity principle implies that $\util_{\hat v}$ is a branched cover of some trivial cylinder (over $K_0$), which is absurd since $\util_{\hat v}$ is a plane. By~\eqref{general_vertex_parametrization} we obtain intersections of the image of $\tau\circ\util_n$ with $K_0$ for large $n$, a contradiction. \\

\noindent {\bf (III)} We conclude the argument. \\

It is simple to conclude from steps {\bf (I)} and {\bf (II)}, using positivity of intersections and Carleman's similarity principle, that the images of the cylinders $\util_{v_i}$ do not intersect $\tau^{-1}(K_0)$. It follows that $P_{e_i} \in \P^{\leq T,(p,q)}(\lambda)$ for every $i=0,\dots,N-1$ (in particular the $P_{e_i}$ are simply covered) and
\begin{itemize}
  \item $\util_{v_0} \in \M^{\leq T,(p,q)}_{\jhat}(P,P_{e_0})/\R$,
  \item $\util_{v_i} \in \M^{\leq T,(p,q)}_{\jhat}(P_{e_i},P_{e_{i+1}})/\R$ for $i=0,\dots,N-1$,
  \item $\util_{v_N} \in \M^{\leq T,(p,q)}_{\jhat}(P_{e_{N-1}},P'')/\R$.
\end{itemize}

All these cylinders are somewhere injective since $(p,q)$ is a relatively prime pair of integers, and regular since $\jhat \in \J_{\rm reg}(\lambda)$ by assumption. Usual arguments using the additivity properties of the Fredholm indices show that we get a contradiction in case A if $\mathcal T$ has more than one vertex. In this case $\M^{\leq T,(p,q)}_{\jhat}(P,P'')/\R$ is therefore compact. Moreover, by regularity, it is also discrete, and hence finite. This proves Theorem~\ref{compactness_index_0_differential}. Analogously, $\mathcal T$ has precisely two vertices $\overline v,\underline v$ in case B, and the boundary of $\M^{\leq T,(p,q)}_{\jhat}(P,P'')/\R$ consists of bubbling-off trees with precisely two vertices: these will be shortly denoted by $[\util_+,\util_-]$ where $\util_+ \in \M^{\leq T,(p,q)}_{\jhat}(P,P')/\R$ and $\util_- \in \M^{\leq T,(p,q)}_{\jhat}(P',P'')/\R$, for some $P' \in \P^{\leq T,(p,q)}(\lambda)$ satisfying $\mu_{CZ}(P') = \mu_{CZ}(P)-1 = \mu_{CZ}(P'')+1$ (the decoration is not explicit in the notation $[\util_+,\util_-]$ but should not be forgotten). Conversely, every $[\util_+,\util_-]$ as above is a boundary point of $\M^{\leq T,(p,q)}_{\jhat}(P,P'')/\R$. This is proved using the gluing map and noting that the closures of the images of the maps $\tau \circ \util_\pm$ are contained in $S^3\setminus K_0$: note here that all curves involved are somewhere injective and regular and, consequently, gluing can be done in a standard fashion, moreover, the glued cylinders have image in $\tau^{-1}(S^3\setminus K_0)$. As in Floer theory, the number of terms $q_{P''}$ appearing in $\partial^2 (q_{P})$ is even, proving Theorem~\ref{d2=0}.

\subsection{Proofs of Theorems~\ref{compactness_index_0_chain_map} and~\ref{chain_homotopy_thm}}\label{proofs_thms_chain_homotopy}

As in the proof of Theorems~\ref{compactness_index_0_differential} and~\ref{d2=0}, the argument relies on the careful analysis of the compactification of the moduli space $\M^{\leq T,(p,q)}_{\bar J}(P,P'')$ where $P\in \P^{\leq T,(p,q)}(\lambda_+)$ and $P'' \in \P^{\leq T,(p,q)}(\lambda_-)$ satisfy
\begin{itemize}
\item[{\bf A)}] $\mu_{CZ}(P'') = \mu_{CZ}(P)$, or
\item[{\bf B)}] $\mu_{CZ}(P'') + 1 = \mu_{CZ}(P)$.
\end{itemize}

Taking $\bar J$ regular, these are smooth manifolds of dimension 0 in case~A, and dimension 1 in case~B, since there are no orbifold points (all cylinders are somewhere injective because their asymptotic limits are simply covered Reeb orbits).

By results of~\cite{sftcomp}, $\M^{\leq T,(p,q)}_{\bar J}(P,P'')$ is compactified by adding certain holomorphic buildings but, in our particular situation, these will be given the simpler description of a bubbling-off tree of finite-energy spheres, similarly as was done in the proofs of Theorems~\ref{compactness_index_0_differential} and~\ref{d2=0} above. These are again collections $\{\util_v\}$ of finite-energy spheres with one positive puncture, where the index $v$ runs on the set of vertices of a finite tree with a root $\overline v$ and a distinguished leaf $\underline v$. Each sphere is pseudo-holomorphic with respect to $\jhat_+$, $\jhat_-$ or $\bar J$, and the $\jhat_\pm$-spheres are not trivial cylinders over periodic orbits (although they might be branched covers over such trivial cylinders). Moreover, for each path $w^1,w^2,\dots,w^m$ connecting the root $w^1 = \overline v$ to a leaf $w^m$ ($w^{i+1}$ is a child of $w^i$) there is at most one vertex $w^{m_*}$ such that $\util_{w^{m_*}}$ is $\bar J$-holomorphic, in which case $\util_{w^j}$ is $\jhat_+$-holomorphic if $j<m_*$ or $\jhat_-$-holomorphic if $j>m_*$. The stem is a particular example of such a path, and it must always contain a $\bar J$-holomorphic (punctured) sphere.

As before, we may assume that $\infty \in \C \sqcup\{\infty\} \simeq \CP^1$ is the positive puncture of all $\util_v$. All the other data described in the proofs of Theorems~\ref{compactness_index_0_differential} and~\ref{d2=0} is still present. The negative punctures of $\util_v$ are in 1-1 correspondence with the outgoing edges of $v$. For every edge $e$ connecting $v$ to its child $v'$ there is an associated periodic trajectory $P_e$ which is a $\lambda_+$-Reeb orbit if $\util_{v'}$ is $\jhat_+$ or $\bar J$-holomorphic, or it is a $\lambda_-$-Reeb orbit if $\util_v$ is $\jhat_-$ or $\bar J$-holomorphic. Moreover, $\util_v$ is asymptotic to $P_e$ at the corresponding negative puncture, and $\util_{v'}$ is asymptotic to $P_e$ at its positive puncture. A set of decorations $\{r_e\}$ is given exactly as before.

Consider any sequence $[\util_n] \in \M^{\leq T,(p,q)}_{\bar J}(P,P'')$, where $\bar J \in \J_{\rm reg}(\jhat_-,\jhat_+:K_0)$ and $\jhat_\pm \in \J_{\rm reg}(\lambda_\pm)$. We make all assumptions explained in Section~\ref{chain_map_paragraph}, and the argument that follows strongly relies on~\eqref{crucial_chain_map}. By results from~\cite{sftcomp} we may assume $[\util_n]$ converges to a bubbling-off tree of finite-energy spheres as just described, in the sense that for each vertex $v$ one finds constants $c_n\in \R$, $A_n,B_n \in \C$, $A_n \neq0$, such that~\eqref{general_vertex_parametrization} holds. For vertices in the stem we may take $B_n = 0$ as in~\eqref{stem_parametrization}.

We assume that this tree has more than one vertex, let $$ S = (v_0 = \overline v,v_1,\dots,v_{N-1},v_N = \underline v) $$ be the stem ($v_{i+1}$ is a child of $v_i$), and let $e_i$ be the edge connecting $v_i$ to $v_{i+1}$. First we argue that all orbits $P_{e_i}$ are not contained in $K_0$, then we proceed to show that the tree coincides with the stem, and then we argue that the image of the spheres do not intersect $\tau^{-1}(K_0)$. \\

\noindent {\bf (I)} All orbits $P_{e_i}$ are not contained in $K_0$. \\

For $j=0,1$ there are $\lambda_\pm$-Reeb trajectories $x_j^\pm$ of minimal period $T^\pm_j$ such that $L_j = x_j^\pm(\R)$. We may write $L_j = (x_j^-,T_j^-)$ or $L_j = (x_j^+,T_j^+)$ depending on whether we want to see $L_j$ as a closed $\lambda_-$-Reeb orbit or as a closed $\lambda_+$-Reeb orbit. Arguing indirectly, we assume that some $P_{e_i}$ is contained in $K_0$ and define
\[
  i_0 = \min \{ i=0,\dots, N-1 \mid P_{e_i} \subset K_0 \}.
\]
We only treat the case $P_{e_{i_0}} \subset L_0$, the other case is analogous, and find $k>0$ such that $P_{e_{i_0}} = L_0^k$ (we could have $P_{e_{i_0}} = (x_0^+,kT^+_0)$ or $P_{e_{i_0}} = (x_0^-,kT^-_0)$ depending on the value of $i_0$, but always write $L_0^k$ to denote one of these orbits). Now we set
\[
  i_1 = \max \{ i=0,\dots, N-1 \mid P_{e_i} = L_0^k \} \geq i_0.
\]

The image of the curves $\util_{v_{i_0}}$ and $\util_{v_{i_1+1}}$ are not contained in $\tau^{-1}(K_0)$. In fact, if $\util_{v_{i_1+1}}$ has image contained in $\tau^{-1}(K_0)$ then it is contained in $\tau^{-1}(L_0)$. It is important to note that $\tau^{-1}(L_0)$ is an embedded cylinder with tangent space invariant under $\jhat_-$, $\bar J$ or $\jhat_+$. Thus, using Carleman's similarity principle, we conclude that $\util_{v_{i_1+1}}$ is asymptotic to $L_0^r$ at the distinguished negative puncture $0$ and, by the definition of $i_1$, we must have $r\neq k$. As before, we find $\rho$ very large and $t',t'' \in \R/\Z$ such that $t\mapsto \tau \circ \util_{v_{i_1+1}}(\rho e^{i2\pi(t+t')})$ is a loop close to $L_0^k$ and $t\mapsto \tau \circ \util_{v_{i_1+1}}(\rho^{-1}e^{i2\pi (t+t'')})$ is a loop close to $L_0^r$. Each of these loops can be approximated by loops of the form $t\mapsto \tau\circ \util_n (\rho_ne^{i2\pi (t+t_n)})$, with suitable $\rho_n$, $t_n$. Thus $L_0^k$ is homotopic to $L_0^r$ in $S^3\setminus L_1$, contradicting $k\neq r$. The image of $\util_{v_{i_0}}$ is not contained $\tau^{-1}(L_0)$ since we can use again the similarity principle to conclude that $\util_{v_{i_0}}$ is asymptotic to an orbit contained in $K_0$ at its positive puncture, contradicting the definition of $i_0$.

Let $(\U_\pm,\Phi_\pm)$ be Martinet tubes for $L_0$ as in Definition~\ref{martinet_tube_def} with respect to $\lambda_\pm$, and let $(\theta,x,y)$ denote standard coordinates in $\R/\Z\times \R^2$, so that $\Phi_\pm(x^\pm_0(T^\pm_0t)) = (t,0,0)$. We have sections $Y_\pm(t) = d\Phi_\pm^{-1} \cdot \partial_x|_{(t,0,0)}$ of the $d\lambda_\pm$-symplectic vector bundles $({x^\pm_0}_{T^\pm_0})^*\xi_0 \to \R/\Z$, and $\Phi_\pm$ may be constructed so that the loops $t\mapsto \exp(\epsilon Y_\pm(t))$ ($\epsilon>0$ small) have linking number $0$ with $L_0$. Then $Y_\pm$ can be completed to $d\lambda_\pm$-symplectic frames of $({x^\pm_0}_{T^\pm_0})^*\xi_0$ whose homotopy classes are denoted by $\beta_\pm$. As remarked in the beginning of Section~\ref{section_contact_homology}, $\rho(L_0,\lambda_\pm)$ are the rotation numbers computed with respect to a global frame, so that we get
\begin{equation}\label{}
\rho(L_0,\lambda_\pm,\beta_\pm) = \rho(L_0,\lambda_\pm)-1 = \theta_0(h_\pm)
\end{equation}
where here we used that the self-linking number of $L_0$ is $-1$.

We denote
\[
  \begin{aligned}
    U_0(s,t) &= \tau \circ \util_{v_{i_0}}(\est) \ \text{ for } s\ll -1, \\
    U_1(s,t) &= \tau \circ \util_{v_{i_1+1}}(\est) \ \text{ for } s\gg 1.
  \end{aligned}
\]
Let $i^* \in \{1,\dots,N-1\}$ be the unique index such that $\util_{v_{i^*}}$ is $\bar J$-holomorphic, and write
\[
\begin{aligned}
& (\theta_0(s,t),x_0(s,t),y_0(s,t)) = \left\{ \begin{aligned} \Phi_+ \circ U_0(s,t) \ \text{ if } i_0 < i^* \\ \Phi_- \circ U_0(s,t) \ \text{ if } i_0 \geq i^* \end{aligned} \right. \ \ \ (s\ll -1) \\
& (\theta_1(s,t),x_1(s,t),y_1(s,t)) = \left\{ \begin{aligned} \Phi_+ \circ U_1(s,t) \ \text{ if } i_1 < i^* \\ \Phi_- \circ U_1(s,t) \ \text{ if } i_1 \geq i^* \end{aligned} \right. \ \ \ (s\gg 1).
\end{aligned}
\]
Let $A_0$ be the asymptotic operator at the orbit $L_0^k$ with respect to $\lambda_+$ if $i_0<i^*$ or with respect to $\lambda_-$ if $i_0 \geq i^*$. We can use Theorem~\ref{precise_asymptotics} to find an eigenvalue $\nu_0>0$ of $A_0$ and an eigensection $\eta_0$ satisfying $A_0\eta_0 = \nu_0\eta_0$ such that if $t\mapsto \zeta_0(t) \in \R^2\setminus 0$ is the representation of $\eta_0$ with respect to the frame $\{\partial_x,\partial_y\}$ of $(0\times \R^2)|_{\R/\Z\times (0,0)}$ (here we use $\Phi_+$ or $\Phi_-$ depending on $i_0$) then
\[
  (x_0(s,t),y_0(s,t)) = e^{\int_{s_0}^s\alpha_0(q)dq}(\zeta_0(t)+R_0(s,t)), \ \ \text{ for } s\leq s_0 \ll -1,
\]
where $|R_0(s,t)| + |\alpha_0(s)-\nu_0| \to 0$ as $s\to -\infty$. Analogously we consider the asymptotic operator $A_1$ at the orbit $L_0^k$ with respect to $\lambda_+$ if $i_1<i^*$ or with respect to $\lambda_-$ if $i_1 \geq i^*$, and find eigenvalue $\nu_1<0$ of $A_1$ and corresponding eigensection $\eta_1$ such that if we represent $\eta_1$ as $t\mapsto \zeta_1(t)$ similarly then
\[
  (x_1(s,t),y_1(s,t)) = e^{\int_{-s_0}^s\alpha_1(q)dq}(\zeta_1(t)+R_1(s,t)), \ \ \text{ for } s\geq -s_0 \gg1,
\]
where $|R_1(s,t)| + |\alpha_1(s)-\nu_1| \to 0$ as $s\to +\infty$. The functions $\theta_0(s,t),\theta_1(s,t)$ behave exactly as in~\eqref{functions_theta_0_theta_1}. Since $L_0$ is irrationally elliptic (for both $\lambda_\pm$), $\nu_0>0$ and $\nu_1<0$ we can apply Lemma~\ref{inequalities_windings} to obtain, if $s\gg 1$, that
\begin{equation}\label{est_chain_homot_0}
  \link(t\mapsto U_0(-s,t),L_0) = \wind(\zeta_0) > \left\{ \begin{aligned} &k\theta_0(h_+) \ \text{ if } i_0<i^* \\ &k\theta_0(h_-) \ \text{ if } i_0\geq i^* \end{aligned} \right.
\end{equation}
and
\begin{equation}\label{est_chain_homot_1}
  \link(t\mapsto U_1(s,t),L_0) = \wind(\zeta_1) < \left\{ \begin{aligned} &k\theta_0(h_+) \ \text{ if } i_1<i^* \\ &k\theta_0(h_-) \ \text{ if } i_1\geq i^* \end{aligned} \right.
\end{equation}
Since $i_0 \geq i^*$ implies $i_1 \geq i^*$ we get from~\eqref{est_chain_homot_0}-\eqref{est_chain_homot_1} that
\begin{equation}\label{contradiction_link_1}
\link(t\mapsto U_0(-s,t),L_0) > \link(t\mapsto U_1(s,t),L_0) \ \ \ (s\gg1)
\end{equation}
in view of the crucial assumption $\theta_0(h_-)\leq \theta_0(h_+)$. But, as done before, the loops $t\mapsto U_0(-s,t)$ and $t\mapsto U_1(s,t)$ ($s\gg 1$) can be approximated by loops of the form $t\mapsto \tau \circ \util_n (R_ne^{i2\pi (t+t_n)})$ for suitable $R_n,t_n$, which implies that $t\mapsto U_0(-s,t)$ and $t\mapsto U_1(s,t)$ must be homotopic to each other in $S^3\setminus K_0$ and, consequently, must have the same linking number with $L_0$. We used that the cylinders $\util_n$ do not touch $\tau^{-1}(K_0)$. This contradicts~\eqref{contradiction_link_1} and (I) is proved. \\

\noindent {\bf (II)} $\Gamma_{v_i} = \{0\}$ for every $i=1,\dots,N$, that is, the tree coincides with the stem. \\

\noindent {\bf (III)} Each cylinder does not intersect $\tau^{-1}(K_0)$ and the asymptotic orbits corresponding to the edges lie in $\P^{\leq T,(p,q)}(\lambda_+)$ or in $\P^{\leq T,(p,q)}(\lambda_-)$.  \\

The arguments to prove (II) and (III) are entirely analogous to those explained in Section~\ref{proofs_thms_chain_complex}, we do not repeat them here. Note only that, since $\tau^{-1}(K_0)$ is a pair of disjoint embedded cylinders in $W_{\xi_0}$ with a $\C$-invariant tangent space regardless of the almost complex structure $\jhat_-$, $\bar J$ or $\jhat_+$, we can repeat all the steps using positivity of intersections and the similarity principle. The assumption that no closed $\lambda_\pm$-Reeb orbit contained in $S^3\setminus K_0$ is contractible in $S^3\setminus K_0$ is crucial.

We assumed that the tree has at least two vertices. In case~A, additivity of the Fredholm indices gets us a contradiction. Thus, in case~A, the limiting tree has exactly one vertex which is an element of $\M^{\leq T,(p,q)}_{\bar J}(P,P'')$, the sequence $[\util_n]$ is eventually constant and Theorem~\ref{compactness_index_0_chain_map} is proved. In case~B again additivity of the Fredholm indices shows that the tree has two vertices, so the boundary of $\M^{\leq T,(p,q)}_{\bar J}(P,P'')$ consists of broken cylinders, one level is $\bar J$-holomorphic index 0 cylinder corresponding to a term of $\Phi(\bar J)$, and the other is either a $\jhat_-$ or a $\jhat_+$-holomorphic cylinder corresponding to a term of $\partial(\lambda_+,J_+)$ or of $\partial(\lambda_-,J_-)$. These facts together with a glueing argument, which is standard in view of regularity of the curves involved, show the converse: any such broken cylinder is in the boundary. The important observation here is that when we glue two regular cylinders with images in $\tau^{-1}(S^3\setminus K_0)$ asymptotic to orbits in the complement of $K_0$ we obtain a cylinder with image in $\tau^{-1}(S^3\setminus K_0)$. Consequently, every generator $q_{P''} \in C_{*-1}^{\leq T,(p,q)}(h_-\lambda_0)$ appears an even number of times in the chain $\Phi_{*-1}(\bar J) \circ \partial_*(\lambda_+,J_+)(q_P) - \partial_*(\lambda_-,J_-) \circ \Phi_*(\bar J)(q_P)$, for every generator $q_P \in C_{*}^{\leq T,(p,q)}(h_+\lambda_0)$, as in Floer theory. Theorem~\ref{chain_homotopy_thm} follows.

\subsection{Proofs of Theorems~\ref{degree_1_map_compactness} and~\ref{comparing_chain_maps_thm}}\label{proofs_thms_degree_1_map}

We need to study the compactifi\-ca\-tion of moduli spaces $\M^{\leq T,(p,q)}_{\{\bar J_t\}}(P,P'')$, where $\jhat_\pm \in \J_{\rm reg}(\lambda_\pm)$, $\bar J_0,\bar J_1 \in \J_{\rm reg}(\jhat_-,\jhat_+:K_0)$, $\{\bar J_t\} \subset \widetilde \J_{\rm reg}(\jhat_-,\jhat_+ : K_0)$, and $P \in \P^{\leq T,(p,q)}(\lambda_+)$, $P'' \in \P^{\leq T,(p,q)}(\lambda_-)$ satisfy
\begin{itemize}
\item[{\bf A)}] $\mu_{CZ}(P) - \mu_{CZ}(P'') = -1$, or
\item[{\bf B)}] $\mu_{CZ}(P) - \mu_{CZ}(P'') = 0$.
\end{itemize}

We describe this compactification again appealing to the notion of bubbling-off tree of finite-energy spheres: any sequence $(t_n,[\util_n]) \in \M^{\leq T,(p,q)}_{\{\bar J_t\}}(P,P'')$ admits a subsequence, still denoted $(t_n,[\util_n])$, which converges to a pair $(t_*,\{\util_v\})$ where $t_n \to t_* \in [0,1]$ and $\{\util_v\}$ is a bubbling-off tree exactly as described in the proofs of theorems~\ref{compactness_index_0_chain_map} and~\ref{chain_homotopy_thm}. All $\util_v$ are finite-energy spheres with one positive puncture, each being pseudo-holomorphic with respect to either $\jhat_-$, $\bar J_{t_*}$ or $\jhat_+$. As before, we take representatives $\util_n$ with domains $\C\setminus \{0\} \simeq \CP^1 \setminus \{[0:1],[1:0]\}$, assume $0,\infty$ are the negative and positive puncture, respectively, and ask that for each vertex $v$ there are constants $c_n,A_n\neq0,B_n$ such that~\eqref{general_vertex_parametrization} holds. For vertices in the stem we can arrange that~\eqref{stem_parametrization} holds. All the other properties and compatibility conditions hold as before.

We can then repeat the same argument, using the important assumption that $\theta_j(h_-) \leq \theta_j(h_+)$ for $j=0,1$ to conclude that the tree coincides with its stem, all edges correspond to orbits either in $\P^{\leq T,(p,q)}(\lambda_+)$ or in $\P^{\leq T,(p,q)}(\lambda_-)$, and the image of the cylinders corresponding to the vertices do not intersect $\tau^{-1}(K_0)$.

Then we strongly rely on the genericity assumptions on the chosen almost complex structures to argue, using the additivity of the Fredholm indices, that in case~A the tree has precisely one vertex (the root) $\overline v$ and $(t_*,[\util_{\overline v}]) \in \M^{\leq T,(p,q)}_{\{\bar J_t\}}(P,P'')$. Since this space is 0-dimensional, its points are isolated and the sequence $(t_n,[\util_n])$ is eventually constant. This proves Theorem~\ref{degree_1_map_compactness}.

Similarly, in case~B the tree is one of the following two types:
\begin{itemize}
\item It has exactly one vertex $\overline v = \underline v$, $(t_*,[\util_{\overline v}]) \in \M^{\leq T,(p,q)}_{\{\bar J_t\}}(P,P'')$. If $t_* \in (0,1)$ then $(t_*,[\util_{\overline v}])$ is an interior point of the moduli space.
\item It has exactly two vertices $\overline v \neq \underline v$, $t_* \in (0,1)$ and there exists an orbit $P'$ such that either
\[
\begin{array}{ccc}
P' \in \P^{\leq T,(p,q)}(\lambda_-), & (t_*,[\util_{\overline v}]) \in \M^{\leq T,(p,q)}_{\{\bar J_t\}}(P,P'), & [\util_{\underline v}] \in \M^{\leq T,(p,q)}_{\jhat_-}(P',P'')
\end{array}
\]
or
\[
\begin{array}{ccc}
P' \in \P^{\leq T,(p,q)}(\lambda_+) & [\util_{\overline v}] \in \M^{\leq T,(p,q)}_{\jhat_+}(P,P'), & (t_*,[\util_{\underline v}]) \in \M^{\leq T,(p,q)}_{\{\bar J_t\}}(P',P'').
\end{array}
\]
\end{itemize}

A well-known argument, using the glueing map, proves that there is a bijective correspondence between boundary points of $\M^{\leq T,(p,q)}_{\{\bar J_t\}}(P,P'')$ and the set of trees which are of the second type  or of the first type with $t_* = 0,1$. Here it is crucial that all relevant curves are somewhere injective and regular. Theorem~\ref{comparing_chain_maps_thm} follows as in Floer theory.

\end{document}